\documentclass[10pt,french,english,a4paper]{amsart}
\usepackage[utf8]{inputenc}
\usepackage[T1]{fontenc}
\usepackage[french]{babel}
\usepackage{verbatim}

\usepackage{hyperref}
\hypersetup{pdftoolbar=true, pdftitle={Proj_Ellip}, pdffitwindow=true, colorlinks=true, citecolor=green, filecolor=black, linkcolor=blue, urlcolor=blue, hypertexnames=false}

\usepackage{amsmath}
\usepackage{amssymb}
\usepackage{amsmath}
\usepackage{amsthm}
\usepackage{lmodern}
\usepackage{latexsym}
\usepackage{amsfonts}
\usepackage{mathrsfs}
\usepackage[all]{xy}
\usepackage{bm}
\usepackage{mathalfa}

\renewcommand{\P}{\mathfrak{p}}

\usepackage{pstricks}
\usepackage{pstcol,pst-fill,pst-grad}
\usepackage{pstricks-add,pst-node}
\usepackage{graphicx}
\usepackage{xcolor}

\newcommand{\modm}{\mathrm{mod}}

\newtheorem{theorem}{Theorem}[section]
\newtheorem{cor}[theorem]{Corollary}
\newtheorem{proposition}[theorem]{Proposition}
\newtheorem{lemme}[theorem]{Lemma}
\usepackage{enumerate}
\usepackage{fge}

\numberwithin{equation}{section}

\usepackage[top=4cm, bottom=3.4cm, left=3.5cm , right=3.5cm]{geometry}

\theoremstyle{definition}

\newtheorem{remq}[theorem]{Remark}

\title{Mollification of the Fourth Moment of Dirichlet $L$-Functions}
\author{Raphaël Zacharias}
\date{\today}

\begin{document}

\maketitle

%%%%%%%%%%%%%%%%%%%%%%%%%%%%%%%%%%%%%%%%%%%%%%%%%%%%%%%%%%%%%%%%%%%%%%%%%%%%%%%%%%%%%%%%
%%%%%%%%%%%%%%%%%%%%%%%%%%%%%%%%%%%%%%%%%%%%%%%%%%%%%%%%%%%%%%%%%%%%%%%%%%%%%%%%%%%%%%%%
%%%%%%%%%%							ABSTRACT									%%%%%%%%%%
%%%%%%%%%%%%%%%%%%%%%%%%%%%%%%%%%%%%%%%%%%%%%%%%%%%%%%%%%%%%%%%%%%%%%%%%%%%%%%%%%%%%%%%%
%%%%%%%%%%%%%%%%%%%%%%%%%%%%%%%%%%%%%%%%%%%%%%%%%%%%%%%%%%%%%%%%%%%%%%%%%%%%%%%%%%%%%%%%

\renewcommand{\abstractname}{Abstract}
\begin{abstract}
We evaluate some twisted fourth moment of Dirichlet $L$-functions at the central point $s=\frac{1}{2}$ and for prime moduli $q$. The principal tool is a careful analysis of a shifted convolution problem involving the divisor function using spectral theory of automorphic forms. Having in mind simultaneous non vanishing results, we apply the Theorem to establish an asymptotic formula of a mollified fourth moment for this family of $L$-functions.
\end{abstract}

\renewcommand{\proofname}{Proof}
\renewcommand{\refname}{References}
\renewcommand{\contentsname}{Table of contents}
\tableofcontents

%%%%%%%%%%%%%%%%%%%%%%%%%%%%%%%%%%%%%%%%%%%%%%%%%%%%%%%%%%%%%%%%%%%%%%%%%%%%%%%%%%%%%%%%
%%%%%%%%%%%%%%%%%%%%%%%%%%%%%%%%%%%%%%%%%%%%%%%%%%%%%%%%%%%%%%%%%%%%%%%%%%%%%%%%%%%%%%%%
%%%%%%%%%%%%%%%%%%%%%%%%%%%%%%%%%%%%%%%%%%%%%%%%%%%%%%%%%%%%%%%%%%%%%%%%%%%%%%%%%%%%%%%%
%%%%%%%%%%%%%%%%%%%%%%%%%%%%%%%%%%%%%%%%%%%%%%%%%%%%%%%%%%%%%%%%%%%%%%%%%%%%%%%%%%%%%%%%
%%%%%%%%%%						INTRODUCTION									%%%%%%%%%%
%%%%%%%%%%%%%%%%%%%%%%%%%%%%%%%%%%%%%%%%%%%%%%%%%%%%%%%%%%%%%%%%%%%%%%%%%%%%%%%%%%%%%%%%
%%%%%%%%%%%%%%%%%%%%%%%%%%%%%%%%%%%%%%%%%%%%%%%%%%%%%%%%%%%%%%%%%%%%%%%%%%%%%%%%%%%%%%%%
%%%%%%%%%%%%%%%%%%%%%%%%%%%%%%%%%%%%%%%%%%%%%%%%%%%%%%%%%%%%%%%%%%%%%%%%%%%%%%%%%%%%%%%%
%%%%%%%%%%%%%%%%%%%%%%%%%%%%%%%%%%%%%%%%%%%%%%%%%%%%%%%%%%%%%%%%%%%%%%%%%%%%%%%%%%%%%%%%

\section{Introduction}\label{SectionIntro}
In 2010, Young established in a breakthrough paper \cite{young} the following asymptotic formula for the fourth moment of Dirichlet $L$-functions at $s=1/2$ and for prime moduli $q$ with a power saving error term 
\begin{equation}\label{Young}
\frac{1}{\phi^\ast (q)}\sideset{}{^*}\sum_{\chi \ (\mathrm{mod \ }q)}|L(\chi,\tfrac{1}{2})|^4=P(\log q) + O\left( q^{-\frac{5}{512}+\varepsilon}\right),
\end{equation}
for any $\varepsilon>0$ and where the symbol $^\ast$ means that we avoid $\chi=1$, $\phi^\ast(q)=q-2$ is the number of primitive characters modulo $q$, $P$ is a degree four polynomial with leading coefficient $(2\pi^2)^{-1}$ and $-5/512=(1-2\theta)1/80$ with $\theta=7/64$ is the best known approximation towards the Ramanujan-Petersson conjecture and it is due to Kim and Sarnak \cite{Kim-Sarnak}. 

More recently, Blomer, Fouvry, Kowalski, Michel and Mili\'{c}evi\'{c} revisited the problem in \cite{moments} by considering more general moments, namely of the form
$$\mathscr{M}_{f,g}(q):=\frac{1}{\phi^\ast (q)}\sideset{}{^\ast}\sum_{\chi \ (\mathrm{mod \ }q)}L(f\otimes\chi,\tfrac{1}{2})\overline{L(g\otimes \chi,\tfrac{1}{2})},$$
where $f$ and $g$ can be cuspidal Hecke eigenforms (holomorphic or Maa\ss) or $E(z)$, the central derivative of the unique non-holomorphic Eisentsein series for the full modular group $\mathrm{PSL}_2(\mathbb{Z})$. They obtained various asymptotic formulae depending on the nature of $f,g$ (see Theorem 1.1,1.2,1.3,\cite{moments}). In particular, the case $f=g=E$ corresponds to \eqref{Young} since the twisted $L$-function associated to $E$ is given by 
$$L(E\otimes\chi , s)=\sum_{n=1}^\infty \frac{\tau(n)\chi(n)}{n^s}=L(\chi,s)^2, \ \ \Re e (s)>1,$$
where $\tau(n)$ is the divisor function. Theorem $1.1$ \cite{moments} gives a significant improvement in the error term by passing to an exponent $-1/32$. They used on one hand, powerful algebrico-geometrico results concerning general bilinear forms involving trace functions associated to $\ell$-adic sheaves on $\mathbb{P}^1_{\mathbb{F}_q}$ \cite{algebraic,Kloosterman}. On the other hand, they managed to almost eliminate the dependence in $\theta$ in their error bound by using an average result concerning Hecke eigenvalues (see Lemma 2.4 in \cite{moto} and § 3.5 \cite{moments}). More recently, the five authors lowered the exponent to $-1/20$ in \cite{some} using a smooth version of a Theorem of Shparlinski and Zhang (Theorem 3.1, \cite{zhang}) where the trace function corresponds to rank $2$ Kloosterman sums.

For several applications in analytic number theory, we need to evaluate more general moments. In this paper, we focus in particular on one, called the twisted fourth moment. Let $q>2$ be a prime number, $\ell_1,\ell_2$ be integers such that $(\ell_1,\ell_2)=1$, $(\ell_1\ell_2,q)=1$ and $\ell_i\leqslant L$ with $\log L\asymp \log q$; we define 
\begin{equation}\label{Definition1TwistedFourth}
\mathscr{T}^4(\ell_1,\ell_2;q):= \frac{1}{\phi^\ast (q)}\sideset{}{^\ast}\sum_{\chi \ (\mathrm{mod \ }q)}|L(\chi,\tfrac{1}{2})|^4\chi(\ell_1)\overline{\chi}(\ell_2).
\end{equation}
As pointed out by Maksym Radziwill, Bob Hough in his paper on the angle of large values of L-functions \cite{hough} established a formula for the same moment (see Theorem 4 and the proof is in the Appendix) by adapting the methof of M.P. Young. Our present approach is different and allows us to deal with not necessarily squarefree integers $\ell_1,\ell_2$ and this is crucial for the application to mollification and further simultaneous non vanishing results.
Our first result is the following

%%%%%%%%%%%%%%%%%%%%%%%%%%%%%%%%%%%%%%%%%%%%%%%%%%%%%%%%%%%%%%%%%%%%%%%%%%%%%%%%%%%%%%%%
%%%%%%%%%%%%%%%%%%%%%%%%%%%%%%%%%%%%%%%%%%%%%%%%%%%%%%%%%%%%%%%%%%%%%%%%%%%%%%%%%%%%%%%%
%%%%%%%%%%						FIRST MAIN RESULT							%%%%%%%%%%
%%%%%%%%%%%%%%%%%%%%%%%%%%%%%%%%%%%%%%%%%%%%%%%%%%%%%%%%%%%%%%%%%%%%%%%%%%%%%%%%%%%%%%%%
%%%%%%%%%%%%%%%%%%%%%%%%%%%%%%%%%%%%%%%%%%%%%%%%%%%%%%%%%%%%%%%%%%%%%%%%%%%%%%%%%%%%%%%%

\begin{theorem}\label{FirstTheorem} Let $\ell_1,\ell_2\in\mathbb{N}$ coprime, $(\ell_1\ell_2,q)=1$ and cubefree. Then the twisted fourth moment \eqref{Definition1TwistedFourth} admits the following decomposition
\begin{equation}\label{FourthMomentDecompositionTheorem}
\mathscr{T}^4(\ell_1,\ell_2;q)=\mathscr{T}^4_D(\ell_1,\ell_2;q)+\mathcal{OD}^{MT}(\ell_1,\ell_2;q)+O\left(\frac{(\ell_1\ell_2)^{3/2}L^5}{q^{\eta-\varepsilon}}\right),
\end{equation}
where $\mathscr{T}^4_D(\ell_1,\ell_2;q)$, $\mathcal{OD}^{MT}(\ell_1,\ell_2;q)$ are main terms given respectively by \eqref{ExpressionDiagPart1}, \eqref{ValueODMT} and $\eta=1/14-3\theta/7$ with $\theta=7/64$.
\end{theorem}
The cubefree assumption is not essential but it simplifies a lot our treatment. Since the primary goal of this paper is mollification, we did not concentrate our efforts on the optimization of the power of $L$ and on the value of $\eta$, but rather on the computation of the main terms. We hope to get results as strong as if the Ramanujan-Petersson conjecture were true, especially by adapting the method of Blomer and Mili\'cevi\'c (c.f. Theorem 13, \cite{Blomer2015}).

%%%%%%%%%%%%%%%%%%%%%%%%%%%%%%%%%%%%%%%%%%%%%%%%%%%%%%%%%%%%%%%%%%%%%%%%%%%%%%%%%%%%%%%%
%%%%%%%%%%%%%%%%%%%%%%%%%%%%%%%%%%%%%%%%%%%%%%%%%%%%%%%%%%%%%%%%%%%%%%%%%%%%%%%%%%%%%%%%
%%%%%%%%%%%%%%%%%%%%%%%%%%%%%%%%%%%%%%%%%%%%%%%%%%%%%%%%%%%%%%%%%%%%%%%%%%%%%%%%%%%%%%%%
%%%%%%%%%%						OUTLINE OF THE PROOF							%%%%%%%%%%
%%%%%%%%%%%%%%%%%%%%%%%%%%%%%%%%%%%%%%%%%%%%%%%%%%%%%%%%%%%%%%%%%%%%%%%%%%%%%%%%%%%%%%%%
%%%%%%%%%%%%%%%%%%%%%%%%%%%%%%%%%%%%%%%%%%%%%%%%%%%%%%%%%%%%%%%%%%%%%%%%%%%%%%%%%%%%%%%%
%%%%%%%%%%%%%%%%%%%%%%%%%%%%%%%%%%%%%%%%%%%%%%%%%%%%%%%%%%%%%%%%%%%%%%%%%%%%%%%%%%%%%%%%

\subsection{Outline of the Proof}
In this section, we outline the proof of Theorem \ref{FirstTheorem}. Using the functional equation for $|L(\chi,1/2)|^4$ (c.f. \eqref{FunctionalEquation1}), we represent the twisted central values as a convergent series
$$|L(\chi,\tfrac{1}{2})|^4\chi(\ell_1)\overline{\chi}(\ell_2) = 2\mathop{\sum\sum}_{n,m\geqslant 1}\frac{\tau(n)\tau(m)}{(nm){1/2}}\chi(n\ell_1)\overline{\chi}(m\ell_2) V\left(\frac{nm}{q^2}\right),$$
for some function $V(t)$ which depends on the archimedean factor $L_\infty (\chi,s)$ and decays rapidly for $t\geqslant q^{\varepsilon}$. An important fact is that $V$ depends on the character $\chi$ only through its parity. It is therefore natural to separate the average into even and odd characters. Assuming we are dealing with the even case, the orthogonality relations (c.f. \eqref{OrthogonalityRelation}) gives 
$$\mathscr{T}^4(\ell_1,\ell_2;q)= \frac{1}{\phi^\ast (q)}\sum_{d|q}\phi(d)\mu\left(\frac{q}{d}\right)\mathop{\sum\sum}_{\ell_1 n \equiv \pm \ell_2 m \ (\mathrm{mod \ }d)}\frac{\tau(n)\tau(m)}{(nm)^{1/2}}V\left(\frac{nm}{q^2}\right).$$
A first main term $\mathscr{T}^4_D(\ell_1,\ell_2;q)$ is extracted from the diagonal contribution $\ell_1n=\ell_2m$ and is computed in section \ref{SectionComputationDiagPart}. Putting this part away, applying a partition of unity and we are reduced to the evaluation of the following expression
\begin{equation}\label{IntroODT}
\begin{split}
\mathscr{T}_{OD}^{4,\pm}(\ell_1,\ell_2,N,M;q) := & \ \frac{1}{(NM)^{1/2}\phi^\ast(q)}\sum_{d|q}\phi(d)\mu\left(\frac{q}{d}\right) \\ \times & \ \mathop{\sum\sum}_{\substack{\ell_2 n \equiv \pm \ell_1 m \ (\textnormal{mod }d) \\ \ell_1n \neq \ell_2 m}}\tau(n)\tau(m) W\left(\frac{n}{N}\right)W\left(\frac{m}{M}\right),
\end{split}
\end{equation} 
where $1\leqslant M\leqslant N$ (up to switch $\ell_1$ and $\ell_2$), $NM\leqslant q^{2+\varepsilon}$ and $W$ is a smooth and compactly supported function on $\mathbb{R}_{>0}$ satisfying $W^{(j)}\ll_j 1.$ Since $q$ is prime, the arithmetical sum over $d|q$ could be separated into two terms corresponding to $d=1$ and $d=q$. However, as expected by the beautiful work of Young, an off-diagonal main term $\mathcal{OD}^{MT,\pm}(\ell_1,\ell_2,N,M;q)$ arises when $N\asymp M$ and this sum facilitates its calculation since it cancels some poles whose contributions seem to be big (c.f. § \ref{Paragraphs-contour} and Lemma \ref{LemmaRR}). This technical step allows us to rebuild the partition of unity and to express the second main term as a contour integral of the form
$$\mathcal{OD}^{MT}(\ell_1,\ell_2;q)=\frac{1}{2\pi i}\int_{(\varepsilon)}\mathcal{F}(s,\ell_1,\ell_2;q)\frac{ds}{s},$$
for some function $\mathcal{F}(s,\ell_1,\ell_2;q)$ (c.f. \eqref{DefinitionMathcalF}). A critical feature of this term is that the $q^s$ has disappeared, making it impossible to evaluate the integral by standard contour shift on the left. The situation is similar to that of § 4.3 in \cite{mollification2000} where they study mollification of automorphic $L$-functions. Fortunately, using the functional equation for the Riemann zeta function, a crucial trigonometric identity for the gamma function (c.f. \eqref{CrucialIdentity}) and a careful analysis of a Dirichlet series involving the $\ell_i'$s variables, we show that the integrand is odd and therefore, we are able to evaluate explicitly this integral through a residue at $s=0$ (see Lemma \ref{LemmaParity} and Proposition \ref{PropositionParity}). Computing
such contour integral exactly using the symmetry properties of the integrand was also observed in the work of Soundararajan \cite{sound}.

For the rest of this outline, we only consider the case $d=q$ in \eqref{IntroODT} and we put this off-diagonal main term aside by writing 
\begin{alignat}{1}
\mathcal{OD}^{E,\pm}(\ell_1,\ell_2,N,M;q)= & \ \frac{1}{(MN)^{1/2}}\mathop{\sum\sum}_{\substack{\ell_1n\equiv \ell_2m \ (\mathrm{mod \ }q) \nonumber\\ \ell_1n\neq \ell_2 m}}\tau(n)\tau(m)W\left(\frac{n}{N}\right)W\left(\frac{m}{M}\right) \\ & - \ \mathcal{OD}^{MT,\pm}(\ell_1,\ell_2,N,M;q)\label{IntroEq2}.
\end{alignat}
We mention that $\mathcal{OD}^{MT}$ becomes small when $N\gg M$. More precisely, we prove in Lemma \ref{LemmeAdding}
$$\mathcal{OD}^{MT,\pm}(\ell_1,\ell_2,N,M;q)\ll L^2 q^\varepsilon\left(\frac{M}{N}\right)^{1/2}.$$
The conclusion of Theorem \ref{FirstTheorem} will follow as soon as we prove that 
$$\mathcal{OD}^{E,\pm}(\ell_1,\ell_2,N,M;q)\ll L^A q^{-\eta}$$
for some absolute constants $A,\eta>0$. By the trivial bound 
$$\mathcal{OD}^{E,\pm}(\ell_1,\ell_2,N,M;q)\ll L q^\varepsilon\frac{(NM)^{1/2}}{q},$$
we may assume that $NM$ is close to $q^{2+\varepsilon}$. We will treat \eqref{IntroEq2} by different methods depending on the relative ranges of $N$ and $M$.

%%%%%%%%%%%%%%%%%%%%%%%%%%%%%%%%%%%%%%%%%%%%%%%%%%%%%%%%%%%%%%%%%%%%%%%%%%%%%%%%%%%%%%%%
%%%%%%%%%%%%%%%%%%%%%%%%%%%%%%%%%%%%%%%%%%%%%%%%%%%%%%%%%%%%%%%%%%%%%%%%%%%%%%%%%%%%%%%%
%%%%%%%%%%				THE SHIFTED CONVOLUTION PROBLEM						%%%%%%%%%%
%%%%%%%%%%%%%%%%%%%%%%%%%%%%%%%%%%%%%%%%%%%%%%%%%%%%%%%%%%%%%%%%%%%%%%%%%%%%%%%%%%%%%%%%
%%%%%%%%%%%%%%%%%%%%%%%%%%%%%%%%%%%%%%%%%%%%%%%%%%%%%%%%%%%%%%%%%%%%%%%%%%%%%%%%%%%%%%%%

\subsubsection{The shifted convolution problem} When $M$ and $N$ are relatively close to each other, we interpret the congruence condition $\ell_1 n\equiv \pm\ell_2 m$ (mod $q$) ($\ell_1n\neq\ell_2m$) as $\ell_1n \mp \ell_2 m - hq=0$ for $h\neq 0$. Hence for each $h\neq 0$, we need to analyze the shifted convolution problem for the divisor function. This problem is interesting in its own right and has a long history (see for example \cite{Park} for an overview). 
%%%%%%%%%%%%%%%%%%%%%%%%%%%%%%%%%%%%%%%%%%%%%%%%%%%%%%%%%%%%%%%%%%%%%%%%%%%%%%%%%%%%%%%%
\begin{comment}
Blomer and Mili\'cevi\'c in \cite{Blomer2015} used Jutila's variant of the circle method. This has the main advantage to have a certain degree of freedom with respect to the choice of the moduli and one can deal directly with the congruence subgroup $\Gamma_0(\ell_1\ell_2)$ in the trace formula. Unfortunately, this method is useless here essentially because the uniform estimate for Hecke eigenvalues of cuspidal forms (Wilton's bound) fails for the divisor function.

In the case $\ell_1=\ell_2=1$, Young used an approximate functional equation for the divisor function (Lemma 5.4, \cite{young}) to separate the variables $n$ and $m$. Adapting this technique to our case involves the choice of a lift of a multiplicative inverse $\overline{\ell_2}$ (mod $q$) whose location is hard to control. 

Another possibility would be to use a recent method of Topacogullari \cite{tupac}, but in the end we would face similar issue .
\end{comment}
%%%%%%%%%%%%%%%%%%%%%%%%%%%%%%%%%%%%%%%%%%%%%%%%%%%%%%%%%%%%%%%%%%%%%%%%%%%%%%%%%%%%%%%%
The first thing to do is to smooth the condition $\ell_1n\mp\ell_2m -hq=0$. We choose to return to the classical $\delta$-symbol method which was developed by Duke, Friedlander and Iwaniec in \cite{duke1,duke2}. We follow closely the first steps of \cite{duke1994} and are reduced to estimate sums of the shape (see \eqref{SumShape})
\begin{equation}\label{IntroEq3}
\frac{Q^{-1}}{(MN)^{1/2}}\sum_{d_i|\ell_i}\sum_{h}\sum_{n,m}\tau(n)\tau(m)\sum_{\substack{(c,\ell_1'\ell_2')=1 \\ c\asymp Q}}\frac{S(hq,d_1\overline{\ell_1'}n-d_2\overline{\ell_2'}m;cd_1d_2)}{c} G(n,m,cd_1d_2),
\end{equation}
where $Q$ is the parameter of the delta symbol, $G$ is a weight function, $\ell_i'=\ell_i/d_i$ and the inverse of $\ell_1'$ (resp $\ell_2'$) have to be taken modulo $cd_2$ (resp $cd_1$). 
%%%%%%%%%%%%%%%%%%%%%%%%%%%%%%%%%%%%%%%%%%%%%%%%%%%%%%%%%%%%%%%%%%%%%%%%%%%%%%%%%%%%%%%%
\begin{comment}
In \cite{duke1994}, they made the choice of $Q=\sqrt{M}$ because $M=\min (N,M)$ and estimate the above sum using Weil bound for Kloosterman sums. Applying their method to our case leads to a  bound of the form 
$$\mathcal{OD}^{E,\pm}\ll q^\varepsilon \frac{N}{M^{1/4}q},$$
but this is not sufficient when $N$ is larger compared to $M$. 
\end{comment}
%%%%%%%%%%%%%%%%%%%%%%%%%%%%%%%%%%%%%%%%%%%%%%%%%%%%%%%%%%%%%%%%%%%%%%%%%%%%%%%%%%%%%%%%
For this, we exploit cancellations in the Kloosterman sums using spectral theory of automorphic forms. 
%%%%%%%%%%%%%%%%%%%%%%%%%%%%%%%%%%%%%%%%%%%%%%%%%%%%%%%%%%%%%%%%%%%%%%%%%%%%%%%%%%%%%%%%
\begin{comment}
We mention that the choice of $Q=\sqrt{M}$ is not optimal on the spectral side, so we instead chose $Q=\sqrt{N}$ in order to have a good control on higher derivatives of the Kuznetsov transform (c.f. Proposition \ref{Propositionfunctionf}).
\end{comment}
%%%%%%%%%%%%%%%%%%%%%%%%%%%%%%%%%%%%%%%%%%%%%%%%%%%%%%%%%%%%%%%%%%%%%%%%%%%%%%%%%%%%%%%%
Returning to \eqref{IntroEq3}, we focus on the quantity
$$\sum_{(c,\ell_1'\ell_2')=1}\frac{S(hq,d_1\overline{\ell_1'}n-d_2\overline{\ell_2'}m;cd_1d_2)}{c} G(n,m,cd_1d_2).$$
At this step, we cannot apply directly the usual Kuznetsov formula because of the different inverses $\overline{\ell_1'}$, $\overline{\ell_2'}$ which are not with respect to the modulus (they are mod $cd_2$ (resp $cd_1$)) and we need first to transform the Kloosterman sum. Inspired by \cite{tupac}, we factor in a unique way $d_i=d_i^\ast d_i'$ with $(d_i^\ast , \ell_i')=1$, $d_i' |(\ell_i')^\infty$ and use the twisted multiplicativity to obtain the factorization (we set $v:= d_1'd_2'$),
\begin{alignat*}{1}
S(hq,d_1\overline{\ell_1'}n-d_2\overline{\ell_2'}m;cd_1d_2)= & \  S(hq, \overline{v}^2(d_1\overline{\ell_1'}n-d_2\overline{\ell_2'}m); cd_1^\ast d_2^\ast) \\ \times & \ S(hq,\overline{(cd_1^\ast d_2^\ast)^2}(d_1\overline{\ell_1'}n-d_2\overline{\ell_2'}m);v),
\end{alignat*}
where all multiplicative inverses are this time modulo the modulus of the Kloosterman sum. We then exploit an idea of Blomer and Mili\'cevi\'c \cite{Blomer2015} to separate the variable $c$ 
$$S(hq,\overline{(cd_1^\ast d_2^\ast)^2}(d_1\overline{\ell_1'}n-d_2\overline{\ell_2'}m); v)= \frac{1}{\phi(v)}\sum_{\chi (v)}\overline{\chi}(cd_1^\ast d_2^\ast) \hat{S}_v (\overline{\chi},n,m,\ell_i,hq),$$
with 
$$\hat{S}_v(\chi,n,m,\ell_i,hq):= \sum_{\substack{y(v) \\ (y,v)=1}}\overline{\chi}(y)S(hq\overline{y},(d_1\overline{\ell_1'}n-d_2\overline{\ell_2'}m)\overline{y};v).$$
In this way we obtain sum of Kloosterman sums twisted by Dirichlet characters to which we can evaluate using Kuznetsov formula for automorphic forms with non trivial nebentypus. We finally obtain the bound
$$\mathcal{OD}^{E,\pm}(\ell_1,\ell_2,N,M;q)\ll L^A q^{\varepsilon-1/2+\theta}\left(\frac{N}{M}\right)^{1/2},$$
for some $A>0$ and it is exactly the exepted error term (modulo the power of $L$) according to the treatment of Young. We thus obtain Theorem \ref{FirstTheorem} as long as 
$$\frac{N}{M}\leqslant q^{1-2\theta-2\eta}.$$

\subsection{Mollification and further arithmetic applications}
In section \ref{SectionMollification}, we use Theorem \ref{FirstTheorem} to compute an asymptotic formula for a mollified fourth moment of Dirichlet $L$-functions. More precisely, let $L=q^\lambda$ with $\lambda>0$, $(\bm{x}_\ell)_\ell$ a sequence of complex numbers, $P(X)\in\mathbb{C}[X]$ and $\chi$ a character modulo $q$. We introduce the mollifier
$$M(\chi):=\sum_{\ell\leqslant L}\frac{\bm{x}_\ell\chi(\ell)}{\ell^{1/2}}P\left(\frac{\log \left(\frac{L}{\ell}\right)}{\log L}\right),$$
and 
$$\mathscr{M}^4(q):= \frac{1}{\phi^\ast(q)}\sideset{}{^\ast}\sum_{\chi \ (\mathrm{mod \ }q)}|L(\chi, \tfrac{1}{2}) M(\chi)|^4.$$
Our second main result is the following

%%%%%%%%%%%%%%%%%%%%%%%%%%%%%%%%%%%%%%%%%%%%%%%%%%%%%%%%%%%%%%%%%%%%%%%%%%%%%%%%%%%%%%%%
%%%%%%%%%%						SECOND MAIN THEOREM							%%%%%%%%%%
%%%%%%%%%%%%%%%%%%%%%%%%%%%%%%%%%%%%%%%%%%%%%%%%%%%%%%%%%%%%%%%%%%%%%%%%%%%%%%%%%%%%%%%%

\begin{theorem}\label{SecondTheorem}Set $\bm{x}_\ell=\mu(\ell)$ and $P(X)=X^2$. Then for any $0<\lambda<\frac{11}{8064}$, we have the asymptotic formula
\begin{equation}\label{AsymptoticFormula}
\mathscr{M}^4(q)=\sum_{i=0}^4 c_i \lambda^{-i} +O_\lambda\left(\frac{1}{\log q}\right),
\end{equation}
for some calculable coefficients $c_i\in\mathbb{R}$.
\end{theorem}
In the sequel of this paper and in the same spirit of \cite{mollification2000}, we will use Theorem \ref{SecondTheorem} in conjunction with an asymptotic formula for some cubic moment to prove that for a positive proportion of Dirichlet characters $\chi$ (mod $q$), the triple product $L(\chi,\frac{1}{2})L(\chi\chi_1,\frac{1}{2})L(\chi\chi_2,\frac{1}{2})$ is not zero for any $\chi_1,\chi_2$ modulo $q$.

\subsection{Notations and conventions} In this paper, we use the $\varepsilon$-convention, according to which $\varepsilon$ > 0 is an arbitrarily small positive number whose value may change from line to line. Moreover, although it will not always be specified, most of the implied constants in $\ll$ will depend on such $\varepsilon$.

%%%%%%%%%%%%%%%%%%%%%%%%%%%%%%%%%%%%%%%%%%%%%%%%%%%%%%%%%%%%%%%%%%%%%%%%%%%%%%%%%%%%%%%%
%%%%%%%%%%%%%%%%%%%%%%%%%%%%%%%%%%%%%%%%%%%%%%%%%%%%%%%%%%%%%%%%%%%%%%%%%%%%%%%%%%%%%%%%
%%%%%%%%%%%%%%%%%%%%%%%%%%%%%%%%%%%%%%%%%%%%%%%%%%%%%%%%%%%%%%%%%%%%%%%%%%%%%%%%%%%%%%%%
%%%%%%%%%%%%%%%%%%%%%%%%%%%%%%%%%%%%%%%%%%%%%%%%%%%%%%%%%%%%%%%%%%%%%%%%%%%%%%%%%%%%%%%%
%%%%%%%%%%						BACKGROUND									%%%%%%%%%%
%%%%%%%%%%%%%%%%%%%%%%%%%%%%%%%%%%%%%%%%%%%%%%%%%%%%%%%%%%%%%%%%%%%%%%%%%%%%%%%%%%%%%%%%
%%%%%%%%%%%%%%%%%%%%%%%%%%%%%%%%%%%%%%%%%%%%%%%%%%%%%%%%%%%%%%%%%%%%%%%%%%%%%%%%%%%%%%%%
%%%%%%%%%%%%%%%%%%%%%%%%%%%%%%%%%%%%%%%%%%%%%%%%%%%%%%%%%%%%%%%%%%%%%%%%%%%%%%%%%%%%%%%%
%%%%%%%%%%%%%%%%%%%%%%%%%%%%%%%%%%%%%%%%%%%%%%%%%%%%%%%%%%%%%%%%%%%%%%%%%%%%%%%%%%%%%%%%
\section{Background}

%%%%%%%%%%%%%%%%%%%%%%%%%%%%%%%%%%%%%%%%%%%%%%%%%%%%%%%%%%%%%%%%%%%%%%%%%%%%%%%%%%%%%%%%
%%%%%%%%%%%%%%%%%%%%%%%%%%%%%%%%%%%%%%%%%%%%%%%%%%%%%%%%%%%%%%%%%%%%%%%%%%%%%%%%%%%%%%%%
%%%%%%%%%%%%%%%%%%%%%%%%%%%%%%%%%%%%%%%%%%%%%%%%%%%%%%%%%%%%%%%%%%%%%%%%%%%%%%%%%%%%%%%%
%%%%%%%%%%				BESSEL FUNCTIONS AND VORONOI FORMULA					%%%%%%%%%%
%%%%%%%%%%%%%%%%%%%%%%%%%%%%%%%%%%%%%%%%%%%%%%%%%%%%%%%%%%%%%%%%%%%%%%%%%%%%%%%%%%%%%%%%
%%%%%%%%%%%%%%%%%%%%%%%%%%%%%%%%%%%%%%%%%%%%%%%%%%%%%%%%%%%%%%%%%%%%%%%%%%%%%%%%%%%%%%%%
%%%%%%%%%%%%%%%%%%%%%%%%%%%%%%%%%%%%%%%%%%%%%%%%%%%%%%%%%%%%%%%%%%%%%%%%%%%%%%%%%%%%%%%%

\subsection{Bessel Functions and Voronoi Summation Formula}
We collect here some facts about Bessel functions and their integral transforms.
Let $\phi:[0,\infty)\rightarrow \mathbb{C}$ be a smooth function satisfying $\phi(0)=\phi'(0)=0$ and $\phi^{(i)}\ll_i (1+x)^{-3}$ for all $0\leqslant i\leqslant 3$. For $\kappa\in\{0,1\}$, we define the following three integral transforms (the signification of $\kappa$ will be clear in section \ref{SectionPreliminary})
\begin{equation}\label{definitionBesselTransform}
\begin{split}
\dot{\phi}(k) := & \ 4i^k\int_0^\infty \phi(x)J_{k-1}(x)\frac{dx}{x}, \\
\widehat{\phi} (t) := & \ \frac{2\pi i t^{\kappa}}{\sinh(\pi t)} \int_0^\infty(J_{2it}(x)-(-1)^\kappa J_{-2it})\phi(x)\frac{dx}{x}, \\
\check{\phi}(t):= & \ 8i^{-\kappa}\int_0^\infty f(x)\cosh (\pi t)K_{2it}(x)\frac{dx}{x}.
\end{split}
\end{equation}
Here are some useful estimates concerning the above Bessel transforms.

%%%%%%%%%%%%%%%%%%%%%%%%%%%%%%%%%%%%%%%%%%%%%%%%%%%%%%%%%%%%%%%%%%%%%%%%%%%%%%%%%%%%%%%%
%%%%%%%%%%					LEMMA BESSEL TRANSFORMS							%%%%%%%%%%
%%%%%%%%%%%%%%%%%%%%%%%%%%%%%%%%%%%%%%%%%%%%%%%%%%%%%%%%%%%%%%%%%%%%%%%%%%%%%%%%%%%%%%%%
\begin{lemme}\label{LemmaBesselTransform} Let $\phi$ be a smooth and compactly supported function in $(X,2X)$ satisfying
$$\phi^{(i)}\ll_i X^{-i}$$
for any $i\geqslant 0$. Then for all $t\geqslant 0$ and real $k>1$, we have 
\begin{alignat}{1}
\frac{\widehat{\phi}(t)}{(1+t)^\kappa}, \check{\phi}(t), \dot{\phi}(t) & \ll  \frac{1+|\log X|}{1+X}, \ \ t\geqslant 0, \label{BesselTransform1}\\
\frac{\widehat{\phi}(t)}{(1+t)^\kappa}, \check{\phi}(t), \dot{\phi}(t) & \ll_k  \left( \frac{1}{t}\right)^k\left(\frac{1}{t^{1/2}}+\frac{X}{t}\right), \ \ t\geqslant \max (2X,1), \label{BesseTransform3}
\end{alignat}
where all implied constants are absolute. 
\end{lemme}
\begin{proof} The case $\kappa=0$ is covered in \cite{burgess}, Lemma $2.1$. The proof carry over to the case $\kappa = 1$ with minimal changes.
\end{proof}
We give now a version of the Voronoi summation formula for the divisor function (c.f. Theorem $1.7$, \cite{jutila}).

%%%%%%%%%%%%%%%%%%%%%%%%%%%%%%%%%%%%%%%%%%%%%%%%%%%%%%%%%%%%%%%%%%%%%%%%%%%%%%%%%%%%%%%%
%%%%%%%%%%					PROPOSITION VORONOI								%%%%%%%%%%
%%%%%%%%%%%%%%%%%%%%%%%%%%%%%%%%%%%%%%%%%%%%%%%%%%%%%%%%%%%%%%%%%%%%%%%%%%%%%%%%%%%%%%%%
\begin{proposition}\label{Voronoi} Let $g$ be a smooth and compactly supported function in $\mathbb{R}^+$ and $(d,\ell)=1$. Then 
\begin{alignat*}{1}
\sum_{n=1}^\infty \tau(n)e\left(\frac{dn}{\ell}\right)g(n)= & \ \frac{1}{\ell}\int_0^\infty (\log x+2\gamma - 2\log \ell)g(x)dx \\ + & \ \sum_{\pm}\sum_{n=1}^\infty \tau(n)e\left(\frac{\pm \overline{d}n}{\ell}\right)g^{\pm}(n),
\end{alignat*}
where $\overline{d}$ denotes the inverse modulo $\ell$,
\begin{alignat*}{1}
g^+(y):= & \ \frac{4}{\ell}\int_0^\infty g(x)K_0\left(\frac{4\pi\sqrt{xy}}{\ell}\right)dx, \\
g^{-}(y):= & -\frac{2\pi}{\ell}\int_0^\infty g(x)Y_0\left(\frac{4\pi\sqrt{xy}}{\ell}\right)dx,
\end{alignat*}
and $K_0$, $Y_0$ are the Bessel functions.
\end{proposition}

\subsection{Preliminaries on Automorphic Forms}\label{SectionPreliminary}
In this section, we briefly compile the main results from the theory of automorphic forms which we shall need in section \ref{SectionSpectral}. An exhaustive account of the theory can be found in \cite{spectral} and \cite{classical} from which we borrow much of the notations. We can also find a good summary in \cite{Blomer}.

%%%%%%%%%%%%%%%%%%%%%%%%%%%%%%%%%%%%%%%%%%%%%%%%%%%%%%%%%%%%%%%%%%%%%%%%%%%%%%%%%%%%%%%%
%%%%%%%%%%%%%%%%%%%%%%%%%%%%%%%%%%%%%%%%%%%%%%%%%%%%%%%%%%%%%%%%%%%%%%%%%%%%%%%%%%%%%%%%
%%%%%%%%%%						HECKE EIGENBASIS								%%%%%%%%%%	%%%%%%%%%%%%%%%%%%%%%%%%%%%%%%%%%%%%%%%%%%%%%%%%%%%%%%%%%%%%%%%%%%%%%%%%%%%%%%%%%%%%%%%%
%%%%%%%%%%%%%%%%%%%%%%%%%%%%%%%%%%%%%%%%%%%%%%%%%%%%%%%%%%%%%%%%%%%%%%%%%%%%%%%%%%%%%%%%	
\subsubsection{Hecke eigenbases}
Let $\ell \geqslant 1$ be an integer, $\chi$ a Dirichlet character of modulus $\ell$, $\kappa=\frac{1-\chi(-1)}{2}\in\{0,1\}$ and $k\geqslant 2$ satisfying $k\equiv \kappa$ (mod $2$). 
%%%%%%%%%%%%%%%%%%%%%%%%%%%%%%%%%%%%%%%%%%%%%%%%%%%%%%%%%%%%%%%%%%%%%%%%%%%%%%%%%%%%%%%%
\begin{comment}
We denote by $\mathcal{S}_k(\ell,\chi)$, $\mathcal{L}^2(\ell,\chi)$ and $\mathcal{L}_0^2(\ell,\chi)\subset\mathcal{L}^2(\ell,\chi)$, respectively, the Hilbert spaces (with respect to the Petersson inner product) of holomorphic cusp forms of weight $k$, of Maa\ss \ forms of weight $\kappa$, of Maa\ss \ cusp forms of weight $\kappa$, with respect to the Hecke congruence group $\Gamma_0(\ell)$ and with nebentypus $\chi$. These spaces are endowed with an action of the commutative algebra $\mathbf{T}$ generated by the Hecke operators $\{ T_n \ | \ n\geqslant 1\}$. Among these, the $T_n$ with $(n,\ell)=1$ generate a subalgebra $\mathbf{T}^{(\ell)}$ of $\mathbf{T}$ made of normal operators. As an immediate consequence, the spaces $\mathcal{S}_k(\ell,\chi)$ and $\mathcal{L}_0^2(\ell,\chi)$ have an orthonormal basis made of eigenforms of $\mathbf{T}^{(\ell)}$. We denote this basis respectively by $\mathcal{B}_k(\ell,\chi)$ and $\mathcal{B}(\ell,\chi)$. 
\end{comment}
%%%%%%%%%%%%%%%%%%%%%%%%%%%%%%%%%%%%%%%%%%%%%%%%%%%%%%%%%%%%%%%%%%%%%%%%%%%%%%%%%%%%%%%%
We denote by $\mathcal{B}_k(\ell,\chi)$ (resp. $\mathcal{B}(\ell,\chi)$) a Hecke basis of the Hilbert space of holomorphic cusp forms of weight $k$ (resp. of Maass cusp forms of weight $\kappa$) with respect to the Hecke congruence group $\Gamma_0(\ell)$ and with nebentypus $\chi$.

\vspace{0.1cm}

Instead of using the classical Eisenstein series $E_{\mathfrak{a}}(\cdot,1/2+it)$, where $\mathfrak{a}$ runs over singular cusps of $\Gamma_0(\ell)$, as a basis of the continuous spectrum, we employ another basis of Eisenstein series indexed by a set of parameters of the form
\begin{equation}\label{BasisEisenstein}
\{ (\chi_1,\chi_2,f) \ | \ \chi_1\chi_2=\chi , \ f\in\mathcal{B}(\chi_1,\chi_2) \},
\end{equation}
%%%%%%%%%%%%%%%%%%%%%%%%%%%%%%%%%%%%%%%%%%%%%%%%%%%%%%%%%%%%%%%%%%%%%%%%%%%%%%%%%%%%%%%%
\begin{comment}
The orthogonal complement of $\mathcal{L}_0^2(\ell,\chi)$ in $\mathcal{L}^2(\ell,\chi)$ is the Eisenstein spectrum (plus a constant if the character is trivial) and it's denoted by $\mathcal{E}(\ell,\chi)$. The space $\mathcal{E}(\ell,\chi)$ is continuously spanned by the Eisenstein series $E_{\mathfrak{a}}(\cdot, 1/2+it)$ where $\mathfrak{a}$ runs over singular cusps (with respect to $\chi$) of $\Gamma_0(\ell)$. Such a basis has the advantage to be explicit. On the other hand, it will be usefull for us to employ another Eisenstein basis made of Hecke eigenforms. The adelic reformulation of the theory of modular forms provides a natural spectral expansion of the Eisenstein spectrum in which the basis of Eisenstein series is indexed by a set of parameters of the form 
\end{comment}
%%%%%%%%%%%%%%%%%%%%%%%%%%%%%%%%%%%%%%%%%%%%%%%%%%%%%%%%%%%%%%%%%%%%%%%%%%%%%%%%%%%%%%%%
where $(\chi_1,\chi_2)$ ranges over the pairs of characters of modulus $\ell$ such that $\chi_1\chi_2=\chi$ and $\mathcal{B}(\chi_1,\chi_2)$ is some finite set depending on $(\chi_1,\chi_2)$. We do not need to be more explicit here and we refer to \cite{gelbart} for a precise definition of these parameters. The main advantage of such a basis is that the Eisenstein series are eigenforms of the Hecke operators $T_n$ with $(n,\ell)=1$ : we have 
$$T_n E_{\chi_1,\chi_2,f}(z,1/2+it)=\lambda_{\chi_1,\chi_2}(n,t) E_{\chi_1,\chi_2,f}(z,1/2+it),$$
with
\begin{equation}\label{EigenvalueEisensein}
\lambda_{\chi_1,\chi_2}(n,t):=\sum_{ab=n}\chi_1(a)a^{it}\chi_2(b)b^{-it}.
\end{equation}

%%%%%%%%%%%%%%%%%%%%%%%%%%%%%%%%%%%%%%%%%%%%%%%%%%%%%%%%%%%%%%%%%%%%%%%%%%%%%%%%%%%%%%%%
%%%%%%%%%%%%%%%%%%%%%%%%%%%%%%%%%%%%%%%%%%%%%%%%%%%%%%%%%%%%%%%%%%%%%%%%%%%%%%%%%%%%%%%%
%%%%%%%%%%			EIGENVALUES, FOURIER AND BOUNDEDNESS						%%%%%%%%%%
%%%%%%%%%%%%%%%%%%%%%%%%%%%%%%%%%%%%%%%%%%%%%%%%%%%%%%%%%%%%%%%%%%%%%%%%%%%%%%%%%%%%%%%%
%%%%%%%%%%%%%%%%%%%%%%%%%%%%%%%%%%%%%%%%%%%%%%%%%%%%%%%%%%%%%%%%%%%%%%%%%%%%%%%%%%%%%%%%
\subsubsection{Hecke eigenvalues, Fourier coefficients and boundedness properties}
Let $f$ be a Hecke eigenform with eigenvalues $\lambda_f(n)$ for all $(n,\ell)=1$. We write the Fourier expansion of $f$ at a singular cusp $\mathfrak{a}$ as follows ($z=x+iy$) :
$$f_{|^k\sigma_{\mathfrak{a}}}(z)=\sum_{n\geqslant 1}\rho_{f,\mathfrak{a}}(n)(4\pi n)^{k/2}e(nz) \ \ \mathrm{for} \ f\in\mathcal{B}_k(\ell,\chi),$$
$$f_{|_{\kappa}\sigma_{\mathfrak{a}}}(z)=\sum_{n\neq 0}\rho_{f,\mathfrak{a}}(n)W_{\frac{n}{|n|}\frac{\kappa}{2}it_f}(4\pi |n|y)e(nx) \ \ \mathrm{for } \ f\in\mathcal{B}(\ell,\chi),$$
where $\sigma_{\mathfrak{a}}$ is the scaling matrix of $\mathfrak{a}$, 
\begin{comment}
i.e. $\sigma_{\mathfrak{a}}\in\mathrm{SL}_2(\mathbb{R})$ is such that $\sigma_{\mathfrak{a}}\infty=\mathfrak{a}$ and $\sigma_{\mathfrak{a}}^{-1}\Gamma_0(\ell)_{\mathfrak{a}}\sigma_{\mathfrak{a}}= B:= \left\{\left(\begin{smallmatrix} 1 & b \\ 0 & 1 \end{smallmatrix}\right) \ : \ b\in\mathbb{Z}\right\}$ where $\Gamma_0(\ell)_{\mathfrak{a}}$ denotes the isotropy subgroup of $\mathfrak{a}$,
\end{comment}
 $W_{\frac{n}{|n|}\frac{\kappa}{2}it_f}$ is the Wittaker function and $t_f$ is the spectral parameter of $f$. i.e. $\lambda_f = 1/4+t_f^2$ with $\lambda_f$ the eigenvalue for the hyperbolic Laplace operator. For any $\gamma=\left(\begin{smallmatrix} a & b \\ c & d \end{smallmatrix}\right)\in\mathrm{SL}_2(\mathbb{R})$, the two slash operators $|^k\gamma$ and $|_\kappa \gamma$ of weights $k$ and $\kappa$ are defined by
$$f_{|^k\gamma}(z):= (cz+d)^{-k}f(\gamma z) \ \ \mathrm{and} \ \ f_{|_\kappa \gamma}(z):= \left(\frac{cz+d}{|cz+d|}\right)^{-\kappa}f(\gamma z).$$
For an Eisenstein series $E_{\chi_1,\chi_2,f}(z,1/2+it)$, we write 
\begin{alignat*}{1}
E_{\chi_1,\chi_2,f|_\kappa \sigma_{\mathfrak{a}}}(z,1/2+it)= & \ c_{1,f,\mathfrak{a}}(t)y^{\frac{1}{2}+it}+c_{2,f,\mathfrak{a}}y^{\frac{1}{2}-it} \\  + & \ \sum_{n\neq 0}\rho_{f,\mathfrak{a}}(n,t)W_{\frac{n}{|n|}\frac{\kappa}{2}it_f}(4\pi |n|y)e(nx).
\end{alignat*}
When we are at the usual cusp $\mathfrak{a}=\infty$, there is a close relation between the Fourier coefficients and the Hecke eigenvalues : for $(m,\ell)=1$ and $n\geqslant 1$, one has
\begin{equation}\label{RelationHecke1}
\lambda_{f}(m)\sqrt{n}\rho_{f}(n)=\sum_{d|(m,n)}\sqrt{\frac{mn}{d^2}}\rho_f\left(\frac{mn}{d^2}\right).
\end{equation}
Using Möbius inversion on \eqref{RelationHecke1}, we obtain for $(m,\ell)=1$ and all $n\geqslant 1$
\begin{equation}\label{RelationHecke2}
\sqrt{mn}\rho_f(mn)=\sum_{d|(m,n)}\mu(d)\rho_f\left(\frac{n}{d}\right)\sqrt{\frac{n}{d}}\lambda_f\left(\frac{m}{d}\right).
\end{equation}
We now recall the boundedness result for the Hecke eigenvalues. When $f$ is either a holomorphic cusp form or an Eisentein series, one has 
\begin{equation}\label{BoundHeckeEigenvalue1}
|\lambda_f (n)|\leqslant \tau(n). 
\end{equation}
The result is clear for an Eisenstein serie by \eqref{EigenvalueEisensein}. In the holomorphic setting, it is due to Deligne and Serre (\cite{DE}, \cite{DS}). Finally, if $f$ is a Maa\ss \ cusp form, the best approximation toward the Ramanujan-Petersson conjecture for $\mathrm{GL}_2$ over $\mathbb{Q}$ is due to Kim and Sarnak \cite{Kim-Sarnak}
\begin{equation}\label{BoundHeckeEigenvalue2}
|\lambda_f(n)|\leqslant \tau(n)n^\theta , \ \ \theta=\frac{7}{64}.
\end{equation}
We also have a similar bound for the spectral parameter 
\begin{equation}\label{BoundSpectralParamater}
|\Im m (t_f)|\leqslant \theta.
\end{equation}

%%%%%%%%%%%%%%%%%%%%%%%%%%%%%%%%%%%%%%%%%%%%%%%%%%%%%%%%%%%%%%%%%%%%%%%%%%%%%%%%%%%%%%%%
%%%%%%%%%%%%%%%%%%%%%%%%%%%%%%%%%%%%%%%%%%%%%%%%%%%%%%%%%%%%%%%%%%%%%%%%%%%%%%%%%%%%%%%%
%%%%%%%%%%			     KUZNETSOV AND THE LARGE SIEVE						%%%%%%%%%%
%%%%%%%%%%%%%%%%%%%%%%%%%%%%%%%%%%%%%%%%%%%%%%%%%%%%%%%%%%%%%%%%%%%%%%%%%%%%%%%%%%%%%%%%
%%%%%%%%%%%%%%%%%%%%%%%%%%%%%%%%%%%%%%%%%%%%%%%%%%%%%%%%%%%%%%%%%%%%%%%%%%%%%%%%%%%%%%%%
\subsubsection{Kuznetsov formula and the spectral large sieve inequality}
 Let $\phi : [0,\infty) \rightarrow \mathbb{C}$ be a smooth function satisfying $\phi(0)=\phi'(0)=0$ and $\phi^{(j)}(x)\ll (1+x)^{-3}$ for $0\leqslant j\leqslant 3$. Recall the three integrals transform given by \eqref{definitionBesselTransform}. Then with the already established notations, the following spectral sum formula holds (see \cite{Blomer} for this version with this special Eisenstein basis).

%%%%%%%%%%%%%%%%%%%%%%%%%%%%%%%%%%%%%%%%%%%%%%%%%%%%%%%%%%%%%%%%%%%%%%%%%%%%%%%%%%%%%%%%
%%%%%%%%%%					PROPOSITION KUZNETSOV							%%%%%%%%%%
%%%%%%%%%%%%%%%%%%%%%%%%%%%%%%%%%%%%%%%%%%%%%%%%%%%%%%%%%%%%%%%%%%%%%%%%%%%%%%%%%%%%%%%%
\begin{proposition}[Kuznetsov trace formula]\label{Kuznetsov} Let $\phi$ be as in the previous paragraph, $\mathfrak{a},\mathfrak{b}$ two singular cusps for the congruence group $\Gamma_0(\ell)$ and $a,b>0$ be integers. Then
\begin{alignat}{1}
\sum_{\gamma}^{\Gamma_0(\ell)}\frac{1}{\gamma}S^{\chi}_{\mathfrak{a}\mathfrak{b}}(a,b;\gamma)\phi & \left(\frac{4\pi \sqrt{ab}}{\gamma}\right) =  \sum_{\substack{k\geqslant 2 \\ k\equiv\kappa \ (2)}}\sum_{f\in\mathcal{B}_k(\ell,\chi)}\dot{\phi}(k)\Gamma(k)\sqrt{ab}\overline{\rho_{f,\mathfrak{a}}}(a)\rho_{f,\mathfrak{b}}(b) \nonumber \\ + & \sum_{f\in\mathcal{B}(\ell,\chi)}\hat{\phi}(t_f)\frac{\sqrt{ab}}{\cos(\pi t_f)}\overline{\rho_{f,\mathfrak{a}}}(a)\rho_{f,\mathfrak{b}}(b)\label{Kuznetsov>0} \\ + & \frac{1}{4\pi}\mathop{\sum\sum}_{\substack{\chi_1\chi_2=\chi \\ f\in\mathcal{B}(\chi_1,\chi_2)}}\int\limits_{\mathbb{R}}\hat{\phi}(t)\frac{\sqrt{ab}}{\cosh(\pi t)}\overline{\rho_{f,\mathfrak{a}}}(a,t)\rho_{f,\mathfrak{b}}(b,t)dt, \nonumber
\end{alignat}
and 
\begin{alignat}{1}
\sum_{\gamma}^{\Gamma_0(\ell)}\frac{1}{\gamma}S^{\chi}_{\mathfrak{a}\mathfrak{b}}(a,-b;\gamma)\phi & \left(\frac{4\pi \sqrt{ab}}{\gamma}\right)  =   \sum_{f\in\mathcal{B}(\ell,\chi)}\check{\phi}(t_f)\frac{\sqrt{ab}}{\cos(\pi t_f)}\overline{\rho_{f,\mathfrak{a}}}(a)\rho_{f,\mathfrak{b}}(-b)\label{Kuznetsov<0} \\ + & \frac{1}{4\pi}\mathop{\sum\sum}_{\substack{\chi_1\chi_2=\chi \\ f\in\mathcal{B}(\chi_1,\chi_2)}}\int\limits_{\mathbb{R}}\check{\phi}(t)\frac{\sqrt{ab}}{\cosh(\pi t)}\overline{\rho_{f,\mathfrak{a}}}(a,t)\rho_{f,\mathfrak{b}}(-b,t)dt, \nonumber
\end{alignat}
where $S^\chi_{\mathfrak{ab}}(n,m;\gamma)$ is the generalized twisted Kloosterman sum and it is defined by 
$$S^\chi_{\mathfrak{ab}}(n,m;\gamma):=\sum_{\left(\begin{smallmatrix} \alpha & \beta \\ \gamma & \delta \end{smallmatrix}\right) \in B\setminus \sigma_{\mathfrak{a}}^{-1}\Gamma_0(\ell)\sigma_{\mathfrak{b}}/ B}\overline{\chi}\left(\sigma_{\mathfrak{a}}\left(\begin{matrix}\alpha & \beta \\ \gamma & \delta   \end{matrix}\right)\sigma_{\mathfrak{b}}^{-1}\right)e\left(\frac{n\alpha+m\delta}{\gamma}\right).$$
The notation $\sum_\gamma^{\Gamma_0(\ell)}$ means that we sum over all positive $\gamma$ such that $S^\chi_{\mathfrak{ab}}(n,m;\gamma)$ is not empty.
\end{proposition}
Often the Kuznetsov formula is used hand in hand with the spectral large sieve inequalities. Before stating the result, we denote by $\ell_0$ the conductor of $\chi$ and we also recall that each cusp for $\Gamma_0(\ell)$ (not necessarily singular) is equivalent to a fraction of the form $u/v$, where $v\geqslant 1$, $v|\ell$ and $(v,u)=1$. We define the following quantity :
\begin{equation}\label{mu(a)}
\mu(\mathfrak{a}):= \ell^{-1}\left(v,\frac{\ell}{v}\right).
\end{equation}
Furthermore, if $(a_n)$ is a sequence a complex numbers, we set
\begin{alignat*}{1}
||a_n||_N^2 := & \ \sum_{N<n\leqslant 2N}|a_n|^2,\\
\Sigma^{(\mathrm{H})}(k,f,N):= & \ \sqrt{(k-1)!}\sum_{N<n\leqslant 2N}a_n \rho_{f,\mathfrak{a}}(n)\sqrt{n}, \\
\Sigma_{\pm}^{(\mathrm{M})}(f,N):= & \ \frac{(1+|t_f|)^{\pm\frac{\kappa}{2}}}{\sqrt{\cosh(\pi t_f)}}\sum_{N<n\leqslant 2N}a_n\rho_{f,\mathfrak{a}}(\pm n)\sqrt{n}, \\
\Sigma^{(\mathrm{E})}_{\pm}(f,t,N):= & \ \frac{(1+|t|)^{\pm\frac{\kappa}{2}}}{\sqrt{\cosh(\pi t)}}\sum_{N<n\leqslant 2N}a_n\rho_{f,\mathfrak{a}}(\pm n,t)\sqrt{n}.
\end{alignat*}
Then the following bounds are known as the spectral large sieve inequalities.

%%%%%%%%%%%%%%%%%%%%%%%%%%%%%%%%%%%%%%%%%%%%%%%%%%%%%%%%%%%%%%%%%%%%%%%%%%%%%%%%%%%%%%%%
%%%%%%%%%%					PROPOSITION LARGE SIEVE							%%%%%%%%%%
%%%%%%%%%%%%%%%%%%%%%%%%%%%%%%%%%%%%%%%%%%%%%%%%%%%%%%%%%%%%%%%%%%%%%%%%%%%%%%%%%%%%%%%%
\begin{proposition}\label{TheoremSpectralLargeSieve}
Let $T\geqslant 1$ and $N\geqslant 1/2$ be real numbers, $(a_n)$ a sequence of complex numbers and $\mathfrak{a}$ a singular cusp for the group $\Gamma_0(\ell)$. Then 
\begin{alignat*}{1}
\sum_{\substack{2\leqslant k\leqslant T \\ k\equiv \kappa \ (2)}}\sum_{f\in\mathcal{B}_k(\ell,\chi)}\left|\Sigma^{(\mathrm{H})}(k,f,N)\right|^2\ll & \ \left(T^2+\ell_0^{\frac{1}{2}}\mu(\mathfrak{a})N^{1+\varepsilon}\right) ||a_n||_N^2, \\
\sum_{|t_f|\leqslant T}\left|\Sigma_\pm^{(\mathrm{M})}(f,N)\right|^2 \ll & \ \left(T^2+\ell_0^{\frac{1}{2}}\mu(\mathfrak{a})N^{1+\varepsilon}\right) ||a_n||_N^2, \\
\sum_{\substack{\chi_1\chi_2=\chi \\ f\in\mathcal{B}(\chi_1,\chi_2)}}\int_{-T}^T\left|\Sigma_{\pm}^{(\mathrm{E})}(f,t,N)\right|^2 dt \ll & \ \left(T^2+\ell_0^{\frac{1}{2}}\mu(\mathfrak{a})N^{1+\varepsilon}\right) ||a_n||_N^2,
\end{alignat*}
with all implied constants depending only on $\varepsilon$.
\end{proposition}
\begin{proof}
We refer to \cite{sieve}.
\end{proof}

%%%%%%%%%%%%%%%%%%%%%%%%%%%%%%%%%%%%%%%%%%%%%%%%%%%%%%%%%%%%%%%%%%%%%%%%%%%%%%%%%%%%%%%%
%%%%%%%%%%%%%%%%%%%%%%%%%%%%%%%%%%%%%%%%%%%%%%%%%%%%%%%%%%%%%%%%%%%%%%%%%%%%%%%%%%%%%%%%
%%%%%%%%%%					FUNCTIONAL EQUATION								%%%%%%%%%%
%%%%%%%%%%%%%%%%%%%%%%%%%%%%%%%%%%%%%%%%%%%%%%%%%%%%%%%%%%%%%%%%%%%%%%%%%%%%%%%%%%%%%%%%
%%%%%%%%%%%%%%%%%%%%%%%%%%%%%%%%%%%%%%%%%%%%%%%%%%%%%%%%%%%%%%%%%%%%%%%%%%%%%%%%%%%%%%%%
\subsubsection{Functional equation for Dirichlet L-functions}
\noindent Let $\chi$ be a non-principal Dirichlet character of modulus $q>2$ with $q$ prime, $\kappa\in\{0,1\}$ satisfying $\chi(-1)=(-1)^\kappa$ and define
$$\Lambda(\chi,s):= q^{s/2}L_\infty(\chi,s)L(\chi,s),$$
where
\begin{equation}\label{DefinitionLinfty}
L_\infty(\chi,s):=\pi^{-s/2}\Gamma\left( \frac{s+\kappa}{2}\right).
\end{equation}
We know that $\Lambda(\chi,s)$ admits an analytic continuation to the whole complex plane and satisfies a functional equation (c.f. Theorem 4.15, \cite{analytic}) from which we can deduce the following relation on the square $\Lambda^2(\chi,s)$ :
\begin{equation}\label{FunctionalEquation1}
\Lambda^2(\chi,s)=\chi(-1)\varepsilon_{\chi}^2\Lambda^2(\overline{\chi},1-s),
\end{equation}
where $\varepsilon_\chi$ is the normalized Gauss sum
$$\varepsilon_\chi := \frac{1}{q^{1/2}}\sum_{x \ (\mathrm{mod} \ q)}\chi(x)e\left(\frac{x}{q}\right).$$
From \eqref{FunctionalEquation1}, we obtain
$$\Lambda^2(\chi,s)\Lambda^2(\overline{\chi},s)=\Lambda^2(\overline{\chi},1-s)\Lambda^2(\chi,1-s),$$
and thus, the following formula, called an approximate functional equation, which represent $|L(\chi,1/2)|^4$ as a convergent series (see for example Theorem 5.3 in \cite{analytic} or §1.3.2 in \cite{Park}).
\begin{lemme}\label{LemmeApproximate} For $\chi$ a non principal Dirichlet character of modulus prime $q>2$, we have 
\begin{equation}\label{Approximate}
|L(\chi,\tfrac{1}{2})|^4 = 2\sum_{n=1}^\infty\sum_{m=1}^\infty \frac{\tau(n)\tau(m)}{(nm)^{1/2}}\chi(n)\overline{\chi}(m)V_{\chi} \left(\frac{mn}{q^2}\right),
\end{equation}
where $\tau(n)=\sum_{d|n}1$ and
\begin{equation}\label{DefinitionV}
V_{\chi}(x):= \frac{1}{2\pi i}\int_{(2)}\frac{L^2_\infty(\chi,1/2+s)L^2_\infty(\overline{\chi},1/2+s)}{L^2_\infty (\chi,1/2)L^2_\infty(\overline{\chi},1/2)}x^{-s}Q(s)\frac{ds}{s},
\end{equation}
with $Q(s)$ an even and holomorphic function with exponential decay in vertical stripes and satisfying $Q(0)=1$.
\end{lemme}

%%%%%%%%%%%%%%%%%%%%%%%%%%%%%%%%%%%%%%%%%%%%%%%%%%%%%%%%%%%%%%%%%%%%%%%%%%%%%%%%%%%%%%%%
%%%%%%%%%%%%%%%%%%%%%%%%%%%%%%%%%%%%%%%%%%%%%%%%%%%%%%%%%%%%%%%%%%%%%%%%%%%%%%%%%%%%%%%%
%%%%%%%%%%%%%%%%%%%%%%%%%%%%%%%%%%%%%%%%%%%%%%%%%%%%%%%%%%%%%%%%%%%%%%%%%%%%%%%%%%%%%%%%
%%%%%%%%%%%%%%%%%%%%%%%%%%%%%%%%%%%%%%%%%%%%%%%%%%%%%%%%%%%%%%%%%%%%%%%%%%%%%%%%%%%%%%%%
%%%%%%%%%%					THE TWISTED FOURTH MOMENT						%%%%%%%%%%
%%%%%%%%%%%%%%%%%%%%%%%%%%%%%%%%%%%%%%%%%%%%%%%%%%%%%%%%%%%%%%%%%%%%%%%%%%%%%%%%%%%%%%%%
%%%%%%%%%%%%%%%%%%%%%%%%%%%%%%%%%%%%%%%%%%%%%%%%%%%%%%%%%%%%%%%%%%%%%%%%%%%%%%%%%%%%%%%%
%%%%%%%%%%%%%%%%%%%%%%%%%%%%%%%%%%%%%%%%%%%%%%%%%%%%%%%%%%%%%%%%%%%%%%%%%%%%%%%%%%%%%%%%
%%%%%%%%%%%%%%%%%%%%%%%%%%%%%%%%%%%%%%%%%%%%%%%%%%%%%%%%%%%%%%%%%%%%%%%%%%%%%%%%%%%%%%%%
\section{The Twisted Fourth Moment}\label{SectionFourthMoment}
Let $\ell_1,\ell_2$ two cubefree integers such that $(\ell_1,\ell_2)=1$, $(\ell_1\ell_2,q)=1$ and $\ell_i\leqslant L$ with $L$ a small power of $q$. The fundamental quantity that we will study in this paper is the following twisted fourth moment
\begin{equation}\label{DefinitionFourthMoment}
\mathscr{T}^4(\ell_1,\ell_2;q):= \frac{2}{\phi^\ast(q)}\sideset{}{^+}\sum_{\substack{\chi \ (\mathrm{mod} \ q) \\ \chi\neq 1}}|L(\chi,\tfrac{1}{2})|^4\chi(\ell_1)\overline{\chi}(\ell_2),
\end{equation}
where the symbol $+$ over the summation means that we restrict ourselves to the case of even characters and $\phi^\ast (q)$ denotes the number of primitive characters modulo $q$. It is natural to split the family $\{\chi \ (\mathrm{mod} \ q)\}$ separately into even characters and odd characters because they have different gamma factors in their functional equations. In this work, we concentrate almost exclusively on the even characters because the case of the odd characters is similar (we could treat both cases simultaneously but it would clutter the notation). We briefly describe the necessary changes to treat the odd characters in § \ref{RemkOdd} since we need to take them in account for the symmetry of a certain function (see section \ref{SectionSymmetry}).

%%%%%%%%%%%%%%%%%%%%%%%%%%%%%%%%%%%%%%%%%%%%%%%%%%%%%%%%%%%%%%%%%%%%%%%%%%%%%%%%%%%%%%%%
%%%%%%%%%%%%%%%%%%%%%%%%%%%%%%%%%%%%%%%%%%%%%%%%%%%%%%%%%%%%%%%%%%%%%%%%%%%%%%%%%%%%%%%%
%%%%%%%%%%%%%%%%%%%%%%%%%%%%%%%%%%%%%%%%%%%%%%%%%%%%%%%%%%%%%%%%%%%%%%%%%%%%%%%%%%%%%%%%
%%%%%%%%%%			APPLYING THE APPROXIMATE FUNCTIONAL EQUATION				%%%%%%%%%%
%%%%%%%%%%%%%%%%%%%%%%%%%%%%%%%%%%%%%%%%%%%%%%%%%%%%%%%%%%%%%%%%%%%%%%%%%%%%%%%%%%%%%%%%
%%%%%%%%%%%%%%%%%%%%%%%%%%%%%%%%%%%%%%%%%%%%%%%%%%%%%%%%%%%%%%%%%%%%%%%%%%%%%%%%%%%%%%%%
%%%%%%%%%%%%%%%%%%%%%%%%%%%%%%%%%%%%%%%%%%%%%%%%%%%%%%%%%%%%%%%%%%%%%%%%%%%%%%%%%%%%%%%%
\subsection{Applying the Approximate Functional Equation}
Using the approximate functional equation \eqref{Approximate} from Lemma \ref{LemmeApproximate} (we omit the dependence in $\chi$ in the definition of $V_\chi$ since we deal with even characters and $V_\chi$ depends on $\chi$ only through its parity) and we can rewrite \eqref{DefinitionFourthMoment} as
\begin{alignat*}{1}
\mathscr{T}^4(\ell_1,\ell_2;q)= & \ \frac{4}{\phi^\ast(q)}\sum_{n,m}\frac{\tau(n)\tau(m)}{(nm)^{1/2}}V\left(\frac{nm}{q^2}\right)\sideset{}{^+}\sum_{\substack{\chi \ (\mathrm{mod} \ q) \\ \chi\neq 1}}\chi(n)\overline{\chi}(m)\chi(\ell_1)\overline{\chi}(\ell_2).
\end{alignat*}
We now use the following identity which allows us to average the sum over the characters and it is valid for $(m,q)=1$ (see for instance $(3.1)$-$(3.2)$, \cite{Iw-S})
\begin{equation}\label{OrthogonalityRelation}
\sideset{}{^+}\sum_{\substack{\chi \ (\textnormal{mod }q) \\ \chi\neq 1}}\chi(m) = \frac{1}{2} \sum_{\pm}\sum_{\substack{d|q \\ m \equiv \pm 1 (d)}}\phi(d)\mu\left(\frac{q}{d}\right).
\end{equation}
Hence we obtain $\mathscr{T}^4(\ell_1,\ell_2;q)=\sum_{\pm}\mathscr{T}^{4,\pm}(\ell_1,\ell_2;q)$ with 
\begin{equation}\label{FourthMoment2}
\begin{split}
\mathscr{T}^{4,\pm}(\ell_1,\ell_2;q):= & \ \frac{2}{\phi^\ast(q)}\sum_{d|q}\phi(d)\mu\left(\frac{q}{d}\right)\mathop{\sum\sum}_{\substack{\ell_1 n\equiv \pm \ell_2 m \ (d) \\ (mn,q)=1}}\frac{\tau(n)\tau(m)}{(nm)^{1/2}}V\left(\frac{nm}{q^2}\right).
\end{split}
\end{equation}
We now decompose $\mathscr{T}^4(\ell_1,\ell_2;q)$ into a diagonal part and a off-diagonal term by writing 
\begin{equation}\label{DecompositionDiag-OffDiag}
\mathscr{T}^4(\ell_1,\ell_2;q)=\sum_{\pm}\mathscr{T}_{OD}^{4,\pm}(\ell_1,\ell_2;q)+\mathscr{T}_D^{4}(\ell_1,\ell_2;q),
\end{equation}
where $\mathscr{T}_{OD}^{4,\pm}(\ell_1,\ell_2;q)$ is the same as in \eqref{FourthMoment2} but with the extra condition that $n\ell_1\neq m\ell_2$ and the diagonal part is given by 
\begin{equation}\label{DefinitionDiagonalPart}
\mathscr{T}_D^4(\ell_1,\ell_2;q):=2\mathop{\sum\sum}_{\substack{\ell_1 n=\ell_2 m \\ (nm,q)=1}}\frac{\tau(n)\tau(m)}{(nm)^{1/2}}V\left(\frac{nm}{q^2}\right).
\end{equation}

%%%%%%%%%%%%%%%%%%%%%%%%%%%%%%%%%%%%%%%%%%%%%%%%%%%%%%%%%%%%%%%%%%%%%%%%%%%%%%%%%%%%%%%%
%%%%%%%%%%%%%%%%%%%%%%%%%%%%%%%%%%%%%%%%%%%%%%%%%%%%%%%%%%%%%%%%%%%%%%%%%%%%%%%%%%%%%%%%
%%%%%%%%%%%%%%%%%%%%%%%%%%%%%%%%%%%%%%%%%%%%%%%%%%%%%%%%%%%%%%%%%%%%%%%%%%%%%%%%%%%%%%%%
%%%%%%%%%%				COMPUTATION OF THE DIAGONAL PART						%%%%%%%%%%
%%%%%%%%%%%%%%%%%%%%%%%%%%%%%%%%%%%%%%%%%%%%%%%%%%%%%%%%%%%%%%%%%%%%%%%%%%%%%%%%%%%%%%%%
%%%%%%%%%%%%%%%%%%%%%%%%%%%%%%%%%%%%%%%%%%%%%%%%%%%%%%%%%%%%%%%%%%%%%%%%%%%%%%%%%%%%%%%%
%%%%%%%%%%%%%%%%%%%%%%%%%%%%%%%%%%%%%%%%%%%%%%%%%%%%%%%%%%%%%%%%%%%%%%%%%%%%%%%%%%%%%%%%
\subsection{Computation of the Diagonal Part}\label{SectionComputationDiagPart}
In this section, we extract a main term coming from the diagonal part $\mathscr{T}_D^4(\ell_1,\ell_2;q)$. We use the standard technique consisting in shifting the contour of integration. We first remark that up to an error of size $O(L^{1/2}q^{-1+\varepsilon})$, we can remove the primality condition $(nm,q)=1$. Once we have done this, we write $V$ as inverse Mellin transform (see definition \eqref{DefinitionV}), obtaining (up to an error term of $O(L^{1/2}q^{-1+\varepsilon})$)
\begin{alignat}{1}\label{FourthMomentIntegral}
\mathscr{T}^4_D(\ell_1,\ell_2;q)=\frac{2}{2\pi i}\int_{(2)}G(s)q^{2s}\left(\mathop{\sum\sum}_{\ell_1 n=\ell_2 m}\frac{\tau(n)\tau(m)}{(nm)^{1/2+s}}\right)\frac{ds}{s},
\end{alignat}
where $G(s)$ is the integrand in $V$, i.e. (recall that $\kappa=0$ here)
\begin{equation}\label{DefinitionG(s)}
G(s) = \pi^{-2s}\frac{\Gamma\left(\frac{\frac{1}{2}+s}{2}\right)^4}{\Gamma(\frac{1}{4})^4}Q(s).
\end{equation}
%%%%%%%%%%%%%%%%%%%%%%%%%%%%%%%%%%%%%%%%%%%%%%%%%%%%%%%%%%%%%%%%%%%%%%%%%%%%%%%%%%%%%%%%
%%%%%%%%%%				FIRST LEMMA FACTORIZATION							%%%%%%%%%%
%%%%%%%%%%%%%%%%%%%%%%%%%%%%%%%%%%%%%%%%%%%%%%%%%%%%%%%%%%%%%%%%%%%%%%%%%%%%%%%%%%%%%%%%
\begin{lemme}\label{Lemma1} We have the factorization
\begin{equation}\label{Serie1}
\mathop{\sum\sum}_{\ell_1 n=\ell_2 m}\frac{\tau(n)\tau(m)}{(nm)^{1/2+s}}=\frac{f(\ell_1\ell_2;1+2s)}{(\ell_1\ell_2)^{1/2+s}}\frac{\zeta^4(1+2s)}{\zeta(2+4s)},
\end{equation}
where $n\mapsto f(n;s)$ is a multiplicative function supported on cubefree integers and whose values on $p$ and $p^2$ are given by
\begin{equation}\label{Valuesf}
f(p;s)=\frac{2}{1+p^{-s}} \ \ , \ \ f(p^2;s)=\frac{3-p^{-s}}{1+p^{-s}}.
\end{equation}
\end{lemme}
\begin{proof}
Since $(\ell_1,\ell_2)=1$, the condition $\ell_1n=\ell_2m$ is equivalent to $n=\ell_2 j$ and $m=\ell_1 j$ with $j\geqslant 1$. Thus, the left handside of $\eqref{Serie1}$ can be written as 
$$\frac{1}{(\ell_1\ell_2)^{1/2+s}}\sum_{j\geqslant 1}\frac{\tau(\ell_1 j)\tau(\ell_2 j)}{j^{1+2s}}.$$
Using the fact that the $\ell_i'$s are cubefree, we factorize the above sum as an infinite product over the primes
$$\prod_{p||\ell_1}L_p\prod_{p^2|\ell_1}L_p\prod_{p||\ell_2}L_p \prod_{p^2|\ell_2}L_p\prod_{p\nmid \ell_1\ell_2}L_p,$$
with
$$L_p = \left\{ \begin{array}{lcl} \sum_{\alpha\geqslant 0}\frac{(\alpha +2)(\alpha+1)}{p^{\alpha(1+2s)}} & \mathrm{if} & p||\ell_i, \\ \sum_{\alpha\geqslant 0}\frac{(\alpha+3)(\alpha+1)}{p^{\alpha(1+2s)}} & \mathrm{if} & p^2 |\ell_i, \\ \sum_{\alpha\geqslant 0}\frac{(\alpha+1)^2}{p^{\alpha(1+2s)}} & \mathrm{if} & p\nmid \ell_1\ell_2.
\end{array} \right.$$
Using
$$\sum_{\alpha\geqslant 0}\frac{(\alpha+1)}{p^{\alpha(1+2s)}}=\frac{1}{(1-p^{-1-2s})^2} \ \ \mathrm{and} \ \ \sum_{\alpha\geqslant 0}\frac{(\alpha+1)^2}{p^{\alpha(1+2s)}}=\frac{1+p^{-1-2s}}{(1-p^{-1-2s})^3}$$
and we get for $p|| \ell_i$
\begin{alignat*}{1}
L_p = & \ \sum_{\alpha\geqslant 0}\frac{(\alpha+1)}{p^{\alpha(1+2s)}}+\sum_{\alpha\geqslant 0}\frac{(\alpha+1)^2}{p^{\alpha(1+2s)}}=\frac{1}{(1-p^{-1-2s})^2}+\frac{1+p^{-1-2s}}{(1-p^{-1-2s})^3} \\ = & \ \left(\frac{1-p^{-1-2s}}{1+p^{-1-2s}}+1\right)\frac{1+p^{-1-2s}}{(1-p^{-1-2s})^3}=\frac{2}{1+p^{-1-2s}}\frac{1+p^{-1-2s}}{(1-p^{-1-2s})^3}. 
\end{alignat*}
We proceed in a similar way for $p^2 | \ell_i$ and we obtain 
$$L_p = \frac{3-p^{-1-2s}}{1+p^{-1-2s}}\frac{1+p^{-1-2s}}{(1-p^{-1-2s})^3}.$$
We conclude the lemma by the well known identity 
$$\sum_{n\geqslant 1}\frac{\tau(n)^2}{n^s}=\prod_p\frac{1+p^{-s}}{(1-p^{-s})^3}=\frac{\zeta^4(s)}{\zeta(2s)}.$$
\end{proof}
%%%%%%%%%%%%%%%%%%%%%%%%%%%%%%%%%%%%%%%%%%%%%%%%%%%%%%%%%%%%%%%%%%%%%%%%%%%%%%%%%%%%%%%%
%%%%%%%%%%					END OF THE LEMMA									%%%%%%%%%%
%%%%%%%%%%%%%%%%%%%%%%%%%%%%%%%%%%%%%%%%%%%%%%%%%%%%%%%%%%%%%%%%%%%%%%%%%%%%%%%%%%%%%%%%
We insert the factorization \eqref{Serie1} in \eqref{FourthMomentIntegral}, obtaining
$$\mathscr{T}_D^4(\ell_1,\ell_2;q)=\frac{2}{2\pi i}\int_{(2)}\frac{G(s)q^{2s}}{\zeta(2+4s)}\frac{f(\ell_1\ell_2;1+2s)}{(\ell_1\ell_2)^{1/2+s}}\zeta^4(1+2s)\frac{ds}{s},$$
Moving the $s$-line on the left to $s=-1/4+\varepsilon$, we pass a pole of order five at $s=0$. Note that for $\Re e (s)=\delta>-1/2$, we have uniformly $f(\ell_1\ell_2,1+2s)\ll_{\delta,\varepsilon} (\ell_1\ell_2)^\varepsilon$ and thus, we can bound the remaining integral by $O(q^{-1/2+\varepsilon}(\ell_1\ell_2)^{-1/4})$. Hence we obtain
%%%%%%%%%%%%%%%%%%%%%%%%%%%%%%%%%%%%%%%%%%%%%%%%%%%%%%%%%%%%%%%%%%%%%%%%%%%%%%%%%%%%%%%%
%%%%%%%%%%				PROPOSITION DIAGONAL MAIN TERM						%%%%%%%%%%
%%%%%%%%%%%%%%%%%%%%%%%%%%%%%%%%%%%%%%%%%%%%%%%%%%%%%%%%%%%%%%%%%%%%%%%%%%%%%%%%%%%%%%%%
\begin{proposition}\label{PropositionDiag}The diagonal part given by \eqref{DefinitionDiagonalPart} can be written as 
\begin{equation}\label{ExpressionDiagPart1}
\mathscr{T}^4_D(\ell_1,\ell_2;q)=2\mathrm{Res}_{s=0}\left\{\frac{G(s)q^{2s}}{s\zeta(2+4s)}\frac{f(\ell_1\ell_2;1+2s)}{(\ell_1\ell_2)^{1/2+s}}\zeta^4(1+2s)\right\}+O\left(\frac{q^{\varepsilon-1/2}}{(\ell_1\ell_2)^{1/4}}\right),
\end{equation}
where $f(\ell_1\ell_2,1+2s)$ is defined in Lemma \ref{Lemma1}.
\end{proposition}

%%%%%%%%%%%%%%%%%%%%%%%%%%%%%%%%%%%%%%%%%%%%%%%%%%%%%%%%%%%%%%%%%%%%%%%%%%%%%%%%%%%%%%%%%%%%%%%%%%%%%%%%%%%%%%%%%%%%%%%%%%%%%%%%%%%%%%%%%%%%%%%%%%%%%%%%%%%%%%%%%%%%%%%%%%%%%%%%%%%%%%%%%%%%%%%%%%%%%%%%%%%%%%%%%%%%%%%%%%%%%%%%%%%%%%%%%%%%%%%%%%%%%%%%%%%%%%%%%%%%%%%%%%%%%%%%%%%%%%%%%%%%%%%%%%%%%%%%%%%%%%%%%%%%%%%%%%%%%%%%%%%%%%%%%%%%%%%%%%%%%%%%%%%%%%%%
%%%%%%%%%%					THE OFF-DIAGONAL TERM							%%%%%%%%%%
%%%%%%%%%%%%%%%%%%%%%%%%%%%%%%%%%%%%%%%%%%%%%%%%%%%%%%%%%%%%%%%%%%%%%%%%%%%%%%%%%%%%%%%%%%%%%%%%%%%%%%%%%%%%%%%%%%%%%%%%%%%%%%%%%%%%%%%%%%%%%%%%%%%%%%%%%%%%%%%%%%%%%%%%%%%%%%%%%%%%%%%%%%%%%%%%%%%%%%%%%%%%%%%%%%%%%%%%%%%%%%%%%%%%%%%%%%%%%%%%%%%%%%%%%%%%%%%%%%%%%%%%%%%%%%%%%%%%%%%%%%%%%%%%%%%%%%%%%%%%%%%%%%%%%%%%%%%%%%%%%%%%%%%%%%%%%%%%%%%%%%%%%%%%%%%%
\section{The Off-Diagonal Term}
We evaluate in this section the off-diagonal part in decomposition \eqref{DecompositionDiag-OffDiag}. Removing the primality condition $(mn,q)=1$ in \eqref{FourthMoment2} for an error cost of $O(Lq^{-1/2+\varepsilon})$ and we are reduced to analyze the following quantity  
\begin{equation}\label{DefODT}
\mathscr{T}^{4,\pm}_{OD}(\ell_1,\ell_2;q)=\frac{2}{\phi^\ast(q)}\sum_{d|q}\phi(d)\mu\left(\frac{q}{d}\right)\mathop{\sum\sum}_{\substack{\ell_1n \equiv \pm \ell_2 m \ (\textnormal{mod }d) \\ \ell_1n \neq \ell_2 m}}\frac{\tau(n)\tau(m)}{(nm)^{1/2}}V\left(\frac{nm}{q^2}\right).
\end{equation}
It is convenient for the analysis of \eqref{DefODT} to localize the variables $n$ and $m$ by applying a partition of unity. We choose a partition on $\mathbb{R}_{>0}\times\mathbb{R}_{>0}$ as in the work of Young (c.f. \cite{young}), namely, of the form $\{W_{N,M}(x,y)\}_{N,M}$ where $N,M$ runs over power (positive and negative) of $2$. In consequence, the numbers of such $N,M$ such that $1\leqslant N,M\leqslant X$ is $O(\log^2 X)$. The functions  $W_{N,M}(x,y)$ are of the form $W_N(x)W_M(y)$ with $W_N$ a smooth function supported on $[N,2N]$. Moreover, it is possible to take $W_N(x)=W(x/N)$ with $W$ a fixed, smooth and compactly supported function on $\mathbb{R}_{>0}$ satisfying $W^{(j)}\ll_j 1.$ Applying this partition to \eqref{DefODT}, we obtain $\mathscr{T}_{OD}^{4,\pm}(\ell_1,\ell_2;q)=\sum_{N,M}\mathscr{T}_{OD}^{4,\pm}(\ell_1,\ell_2,N,M;q)$ with
\begin{alignat}{1}
\mathscr{T}_{OD}^{4,\pm}(\ell_1,\ell_2,N,M;q) := & \ \frac{2}{(NM)^{1/2}\phi^\ast(q)}\sum_{d|q}\phi(d)\mu\left(\frac{q}{d}\right) \nonumber \\ \times & \ \mathop{\sum\sum}_{\substack{\ell_2 n \equiv \pm \ell_1 m \ (\textnormal{mod }d) \\ \ell_1n \neq \ell_2 m}}\tau(n)\tau(m) W\left(\frac{n}{N}\right)W\left(\frac{m}{M}\right)V\left(\frac{nm}{q^2}\right),\label{eq1}
\end{alignat} 
where we made the substitution 
\begin{equation}\label{substitution}
W(x)\leftrightarrow x^{-1/2}W(x).
\end{equation}
Because of the fast decay of the function $V(y)$ as $y\rightarrow +\infty$ (easy to see by shifting the contour on the right in the definition \eqref{DefinitionV}), we can assume that $NM\leqslant q^{2+\varepsilon}$ at the cost of an error term $O(q^{-100})$. Furthermore, since each dependency in $\ell_1,\ell_2$ which will appear in the error terms will be of the form $(\ell_1\ell_2)^A$ or $L^B$, we can also assume that $N\geqslant M$. We will treat differently \eqref{eq1} according to the relative size of $M$ and $N$. We also note that the trivial bound is given by 
\begin{equation}\label{trivialbound}
\mathscr{T}_{OD}^{4,\pm}(\ell_1,\ell_2,N,M;q) \ll q^\varepsilon L\frac{(MN)^{1/2}}{q}.
\end{equation}

%%%%%%%%%%%%%%%%%%%%%%%%%%%%%%%%%%%%%%%%%%%%%%%%%%%%%%%%%%%%%%%%%%%%%%%%%%%%%%%%%%%%%%%%%%%%%%%%%%%%%%%%%%%%%%%%%%%%%%%%%%%%%%%%%%%%%%%%%%%%%%%%%%%%%%%%%%%%%%%%%%%%%%%%%%%%%%%%%%%%%%%%%%%%%%%%%%%%%%%%%%%%%%%%%%%%%%%%%%%%%%%%%%%%%%%%%%%%%%%%%%%%%%%%%%%%%%%%%%%%%%%%
%%%%%%%%%%		OFF-DIAGONAL TERM USING BILINEAR FORMS IN KLOOSTERMAN SUMS	%%%%%%%%%%
%%%%%%%%%%%%%%%%%%%%%%%%%%%%%%%%%%%%%%%%%%%%%%%%%%%%%%%%%%%%%%%%%%%%%%%%%%%%%%%%%%%%%%%%%%%%%%%%%%%%%%%%%%%%%%%%%%%%%%%%%%%%%%%%%%%%%%%%%%%%%%%%%%%%%%%%%%%%%%%%%%%%%%%%%%%%%%%%%%%%%%%%%%%%%%%%%%%%%%%%%%%%%%%%%%%%%%%%%%%%%%%%%%%%%%%%%%%%%%%%%%%%%%%%%%%%%%%%%%%%%%%%
\subsection{The Off-Diagonal Term when $N\gg M$}\label{Sectionl-adic}
In this section, we treat the shifted convolution sum \eqref{eq1} when $N$ and $M$ have relatively different sizes (see \eqref{munuassumption}). As a first step, we replace $\phi^\ast(q)$ by $\phi(q)$ for an error cost of $O(Lq^{-1+\varepsilon})$. Once we have done this, we separate the arithmetical sum over $d|q$. When $d=q$, since $(\ell_1,q)=1$, we detect the congruence condition $n\equiv \overline{\ell_1}\ell_2 m$ (mod $q$) using additive characters. We thus get (up to $O(Lq^{-1+\varepsilon})$)
\begin{alignat}{1}
\mathscr{T}^{4,\pm}_{OD}(\ell_1,\ell_2,& N,M;q)= \frac{2}{q(MN)^{1/2}} \sum_{m}\tau(m)W\left(\frac{m}{M}\right)\sideset{}{^*}\sum_{\substack{ a \ \textnormal{(mod }q)}} e\left(\frac{\pm a \overline{\ell_1}\ell_2 m}{q}\right) \nonumber \\ \times & \sum_{\substack{n \\ \ell_1 n\neq \ell_2 m}}\tau(n)e\left( \frac{a n}{q}\right) \tau(n)W\left( \frac{n}{N}\right)V\left(\frac{nm}{q^2}\right)\label{DefinitionS1S2} \\ 
+ & 2\frac{q^{-1}-\phi^\ast(q)^{-1}}{(MN)^{1/2}}\sum_{\substack{n,m \\ \ell_1n \neq \ell_2 m}}\tau(n)\tau(m)W\left(\frac{n}{N}\right)W\left(\frac{m}{M}\right)V\left(\frac{nm}{q^2}\right),\label{LastLine}  
\end{alignat}
where the line \eqref{LastLine} is the contribution of the trivial additive character and the case $d=1$ (the minus sign comes from the Möbius function) and is of size at most $O(q^{-1+\varepsilon})$. Hence we are reduced to the estimation of \eqref{DefinitionS1S2} and we call this expression $\mathcal{S}^{\pm}(\ell_1,\ell_2,N,M;q)$. It is also convenient to separate the variables $n,m$ in the test function $V$. This technical step can be achieved using the integral representation of $V$ (see for example \cite[§ 4.1]{young}). Hence, we are reduced to bound sums of the shape
$$\mathcal{K}^{\pm}(N,M;q)= \frac{1}{q(NM)^{1/2}}\sideset{}{^*}\sum_{a \ (\modm \ q)}\sum_{\substack{n,m\geqslant 1 \\ \ell_1n\neq \ell_2 m}}\tau(n)\tau(m)e\left(\frac{\pm a \overline{\ell_1}\ell_2 m}{q}\right)W_1\left(\frac{n}{N}\right)W_2\left(\frac{m}{M}\right),$$
where the functions $W_i$ are smooth, compactly supported on $\mathbb{R}_{>0}$ and satisfy $W_i^{(j)}\ll_{\varepsilon,j}q^{\varepsilon j}$ for every $\varepsilon>0$ and every $j\geqslant 0$.

Let $N=q^\nu$, $M=q^\mu$ and let $\eta>0$ be a small real number. By the fast decay of $V(y)$ as $y\rightarrow+\infty$ and the bound \eqref{trivialbound}, we can assume that $2-2\eta \leqslant \nu+\mu\leqslant 2+\varepsilon$. Anticipating the results of section \ref{SectionSpectral}, we also make the additional assumption that 
\begin{equation}\label{munuassumption}
\nu-\mu\geqslant 1-2\theta - 2\eta,
\end{equation}
where $\theta=7/64$ is the current best approximation toward the Ramanujan-Petersson conjecture (c.f. \eqref{BoundHeckeEigenvalue2}). Following the first step of \cite[§ 6.3]{moments} and then turn to \cite[§ 4]{some}, we obtain
\begin{proposition}\label{Propositionladique}
Assume that we are in the range \eqref{munuassumption}. Then for any $\varepsilon>0$, we have 
$$\mathcal{K}^{\pm}(N,M;q)\ll q^{-\eta+\varepsilon},$$
where the implied constant depends only on $\varepsilon$ and 
\begin{equation}
\eta = \frac{1-6\theta}{14}=\frac{11}{448}.
\end{equation}
\end{proposition}

\subsection{The Off-Diagonal Term Using Automorphic Forms}\label{SectionSpectral}
We analyze in this section the shifted convolution problem when $N,M$ are relatively close. More precisely, by the trivial bound \eqref{trivialbound} and the Proposition \ref{Propositionladique}, we can assume that $N=q^{\nu}$ and $M=q^{\mu}$ are located in the range 
\begin{equation}\label{Range1MN}
2-2\eta \leqslant \nu+\mu \leqslant 2+\varepsilon \hspace{0.3cm} \textnormal{ and }\hspace{0.3cm} \nu-\mu\leqslant 1-2\theta-2\eta.
\end{equation}
In particular, this restriction implies that
\begin{equation}\label{RangeminimalforM}
\mu\geqslant \frac{1}{2}+\theta+\eta \ \ \mathrm{and} \ \ 1-\eta\leqslant \nu\leqslant \frac{3}{2}-\theta-\eta+\varepsilon.
\end{equation}
After an application of the Voronoi summation formula, we will see that the off-diagonal part given by \eqref{eq1} decomposes as 
$$\mathscr{T}_{OD}^{4,\pm}(\ell_1,\ell_2,N,M;q)=\mathcal{OD}^{MT,\pm}(\ell_1,\ell_2,N,M;q) +\mathcal{OD}^{E,\pm}(\ell_1,\ell_2,N,M;q),$$ 
where the first is a main term and the second is an error term. We treat here the error term $\mathcal{OD}^{E,\pm}$ and evaluate $\mathcal{OD}^{MT,\pm}$ in section \ref{SectionOfDM}.

%%%%%%%%%%%%%%%%%%%%%%%%%%%%%%%%%%%%%%%%%%%%%%%%%%%%%%%%%%%%%%%%%%%%%%%%%%%%%%%%%%%%%%%%%%%%%%%%%%%%%%%%%%%%%%%%%%%%%%%%%%%%%%%%%%%%%%%%%%%%%%%%%%%%%%%%%%%%%%%%%%%%%%%%%%%%%%%%%%%%%%%%%%%%%%%%%%%%%%%%%%%%%%%%%%%%%%%%
%%%%%%%%%%%%%%%%%%%%%%%%%%%%%%%%%%%%%%%%%%%%%%%%%%%%%%%%%%%%%%%%%%%%%%%%%%%%%%%%%%%%%%%%%%%%%%%%%%%%%%%%%%%%%%%%%%%%%%%%%%%%%%%%%%%%%%%%%%%%%%%%%%%%%%%%%%%%%%%%%%%%%%%%%%%%%%%%%%%%%%%%%%%%%%%%%%%%%%%%%%%%%%%%%%%%%%%%
%%%%%%%%%%							THE DELTA SYMBOL										%%%%%%%%%%	
%%%%%%%%%%%%%%%%%%%%%%%%%%%%%%%%%%%%%%%%%%%%%%%%%%%%%%%%%%%%%%%%%%%%%%%%%%%%%%%%%%%%%%%%%%%%%%%%%%%%%%%%%%%%%%%%%%%%%%%%%%%%%%%%%%%%%%%%%%%%%%%%%%%%%%%%%%%%%%%%%%%%%%%%%%%%%%%%%%%%%%%%%%%%%%%%%%%%%%%%%%%%%%%%%%%%%%%%
%%%%%%%%%%%%%%%%%%%%%%%%%%%%%%%%%%%%%%%%%%%%%%%%%%%%%%%%%%%%%%%%%%%%%%%%%%%%%%%%%%%%%%%%%%%%%%%%%%%%%%%%%%%%%%%%%%%%%%%%%%%%%%%%%%%%%%%%%%%%%%%%%%%%%%%%%%%%%%%%%%%%%%%%%%%%%%%%%%%%%%%%%%%%%%%%%%%%%%%%%%%%%%%%%%%%%%%%

\subsubsection{The $\delta$-symbol}
Let $Q\geqslant 1$ be a real number and choose a smooth, even and compactly supported function $w$ in $[Q,2Q]$ satisfying $w(0)=0$, $w^{(i)}\ll Q^{-1-i}$ and $\sum_{r=1}^\infty w(r) =1.$ We can express the delta function in terms of additives characters in the following way 
$$\delta (n)=\sum_{\ell = 1}^\infty \sideset{}{^*}\sum_{\substack{k(\ell)}}e\left(\frac{kn}{\ell}\right)\Delta_\ell (n),$$
where the supscript $^*$ means that we restrict the summation to primitive classes modulo $\ell$ and 
$$\Delta_\ell (u):= \sum_{r=1}^\infty (r\ell)^{-1}\left( w(r\ell) - w\left(\frac{u}{r\ell}\right)\right).$$
The function $\Delta_\ell$ satisfies the following bound (c.f. Lemma $2$ \cite{duke1994})
\begin{equation}\label{BoundDelta}
\Delta_\ell (u)\ll \min \left( \frac{1}{Q^2}, \frac{1}{\ell Q}\right) + \min \left( \frac{1}{|u|},\frac{1}{\ell Q}\right).
\end{equation}
It is also convenient to keep partial track that $\ell_1n\pm \ell_2m-hd$ is not too large to pick $\varphi$ a smooth function such that $\varphi(0)=1$, $\varphi(u)=0$ for $|u|\geqslant U$ and $\varphi^{(i)}\ll U^{-i}$ for some $U$ satisfying $U\leqslant Q^2$. We thus remark that $\Delta_\ell (u)=0$ if $|u|\leqslant U$ and $\ell > 2Q$ (the parameters $U$ and $Q$ will be explicit in Lemma \ref{LemmeRestrictionQ}). We now return to the expression \eqref{eq1} and write the congruence condition $\ell_1n\equiv \pm\ell_2 m$ (mod $d$) as $\ell_1 n\mp \ell_2 m =hd$ for $h\neq 0$ (since $\ell_1n\neq\ell_2m$). We see that if $d=q$, we can assume that $(h,q)=1$ for a cost of $O(L q^{-1+\varepsilon})$, an extra condition that will be used only in § \ref{subsubsectionKuznetsov} and will not be precised under each $h$-summations until there. It follows that \eqref{eq1} can be written as
\begin{alignat}{1}
\mathscr{T}^{4,\pm}_{OD}(\ell_1,\ell_2,N,M;q) = & \ \frac{2}{(MN)^{1/2}\phi^\ast(q)}\sum_{d|q}\phi(q)\mu\left(\frac{q}{d}\right)\sum_{\ell\leqslant 2Q}\sum_{h\neq 0}\sideset{}{^*}\sum_{\substack{k(\ell)}}e\left(\frac{-khd}{\ell}\right) \nonumber\\ & \times \sum_{n=1}^\infty\sum_{m=1}^\infty \tau(n)\tau(m)e\left(\frac{k(\ell_1n\mp\ell_2 m)}{\ell}\right)E^{\mp}(n,m,\ell), \label{eq2}
\end{alignat}
with (omitting the dependance in $d$ and $\ell_i$ in these definitions)
\begin{equation}\label{DefinitionEpm}
E^{\mp}(x,y,\ell):=F^{\mp}(x,y)\Delta_\ell (\ell_1x\mp\ell_2 y-hd),
\end{equation}
and
$$F^{\mp}(x,y):=W\left(\frac{x}{N}\right)W\left(\frac{y}{M}\right)\varphi(\ell_1x\mp\ell_2 y-hd)V\left(\frac{xy}{q^2}\right).$$

%%%%%%%%%%%%%%%%%%%%%%%%%%%%%%%%%%%%%%%%%%%%%%%%%%%%%%%%%%%%%%%%%%%%%%%%%%%%%%%%%%%%%%%%%%%%%%%%%%%%%%%%%%%%%%%%%%%%%%%%%%%%%%%%%%%%%%%%%%%%%%%%%%%%%%%%%%%%%%%%%%%%%%%%%%%%%%%%%%%%%%%%%%%%%%%%%%%%%%%%%%%%%%%%%%%%%%%%
%%%%%%%%%%%%%%%%%%%%%%%%%%%%%%%%%%%%%%%%%%%%%%%%%%%%%%%%%%%%%%%%%%%%%%%%%%%%%%%%%%%%%%%%%%%%%%%%%%%%%%%%%%%%%%%%%%%%%%%%%%%%%%%%%%%%%%%%%%%%%%%%%%%%%%%%%%%%%%%%%%%%%%%%%%%%%%%%%%%%%%%%%%%%%%%%%%%%%%%%%%%%%%%%%%%%%%%%
%%%%%%%%%%			APPLICATION OF VORONOI FORMULA						%%%%%%%%%%
%%%%%%%%%%%%%%%%%%%%%%%%%%%%%%%%%%%%%%%%%%%%%%%%%%%%%%%%%%%%%%%%%%%%%%%%%%%%%%%%%%%%%%%%%%%%%%%%%%%%%%%%%%%%%%%%%%%%%%%%%%%%%%%%%%%%%%%%%%%%%%%%%%%%%%%%%%%%%%%%%%%%%%%%%%%%%%%%%%%%%%%%%%%%%%%%%%%%%%%%%%%%%%%%%%%%%%%%
%%%%%%%%%%%%%%%%%%%%%%%%%%%%%%%%%%%%%%%%%%%%%%%%%%%%%%%%%%%%%%%%%%%%%%%%%%%%%%%%%%%%%%%%%%%%%%%%%%%%%%%%%%%%%%%%%%%%%%%%%%%%%%%%%%%%%%%%%%%%%%%%%%%%%%%%%%%%%%%%%%%%%%%%%%%%%%%%%%%%%%%%%%%%%%%%%%%%%%%%%%%%%%%%%%%%%%%%

\subsubsection{Application of the Voronoi summation formula}\label{SectionVoronoi}
We apply the Voronoi summation formula (c.f. Proposition \ref{Voronoi}) on the $(m,n)$-sum in \eqref{eq2} and get eight error terms plus a principal term (see $(23)$, \cite{duke1994}). We write explicitly the principal term in Section \ref{SectionOfDM} (c.f. eq \eqref{DefinitionOfDiagMainTerm}). All error terms can be treated similarly, so we only focus here on the one which is of the form (recall that $(\ell_1,\ell_2)=1$)
\begin{alignat}{1}
\mathcal{OD}^{E,\pm}(\ell_1&,\ell_2,N,M;q):=  \frac{1}{(MN)^{1/2}\phi^\ast(q)}\sum_{d|q}\phi(d)\mu\left(\frac{q}{d}\right)\sum_{\ell\leqslant 2Q}\frac{(\ell_1\ell_2,\ell)}{\ell^2}\sum_{h\neq 0}\sideset{}{^*}\sum_{\substack{k(\ell)}} \label{eq3}\\ & \times e\left(\frac{-khd}{\ell}\right)\sum_{n=1}^\infty\sum_{m=1}^\infty \tau(n)\tau(m)e\left(-n\frac{\overline{\ell_1'k}}{\ell'}\right)e\left(\pm m\frac{\overline{\ell_2'k}}{\ell''}\right)I^{\mp}(n,m,\ell), \nonumber
\end{alignat}
where $\ell_i'=\ell_i/(\ell_i,\ell)$, $\ell'=\ell/(\ell_1,\ell)$, $\ell''=\ell /(\ell_2,\ell)$, the overlines denote the inverse modulo the respective denominators and where $I^{\mp}(n,m,\ell)$ involves the $Y_0$ Bessel function :
\begin{equation}\label{definitionI(n,m,l)}
I^{\mp}(n,m,\ell):=4\pi^2\int\limits_0^\infty\int\limits_0^\infty E^{\mp}(x,y,\ell)Y_0\left(\frac{4\pi d_1\sqrt{nx}}{\ell}\right)Y_0\left(\frac{4\pi d_2\sqrt{my}}{\ell}\right)dxdy,
\end{equation}
where we also set $d_i:= (\ell_i,\ell)$. The main result of this section is the following non-trivial bound.

%%%%%%%%%%%%%%%%%%%%%%%%%%%%%%%%%%%%%%%%%%%%%%%%%%%%%%%%%%%%%%%%%%%%%%%%%%%%%%%%%%%%%%%%%%%%%%%%%%%%%%%%%%%%%%%%%%%%%%%%%%%%%%%%%%%%%%%%%%%%%%%%%%%%%%%%%%%%%%%%%%%%%%%%%%%%%%%%%%%%%%%%%%%%%%%%%%%%%%%%%%%%%%%%%%%%%%%%
%%%			  		THEOREM NON TRIVIAL BOUND					%%%
%%%%%%%%%%%%%%%%%%%%%%%%%%%%%%%%%%%%%%%%%%%%%%%%%%%%%%%%%%%%%%%%%%%%%%%% 
\begin{theorem}\label{Theorem1} The quantity defined by \eqref{eq3} satisfies 
\begin{alignat*}{1}
\mathcal{OD}^{E,\pm}(\ell_1,\ell_2,N,M;q)\ll q^{\varepsilon-1/2+\theta}(\ell_1\ell_2)^{3/2}L^5\left(\frac{N}{M}\right)^{1/2} + q^\varepsilon L^{8}\left(\frac{N}{q^2}\right)^{1/4}. 
\end{alignat*}
where the implied constant only depends on $\varepsilon$.
\end{theorem}
%%%%%%%%%%%%%%%%%%%%%%%%%%%%%%%%%%%%%%%%%%%%%%%%%%%%%%%%%%%%%%%%%%%%%%%%
%%%							END THEOREM							   %%%

From now on, we only consider the case $\mathcal{OD}^{E,+}(\ell_1,\ell_2,N,M;q)$ since the other treatment is completly identical and we write $I(n,m,\ell)$ and $\mathcal{OD}^{E}$ instead of $I^{-}(n,m,\ell)$ and $\mathcal{OD}^{E,+}$. As in Section \ref{Sectionl-adic}, we can also remove the test function $V$ in the definition of $E(x,y,\ell)$ using its integral representation for an error cost of $q^\varepsilon$ and a minor change on the function $W$. To not clutter further notations and computations, we will assume that $W^{(j)}\ll 1$ instead of $\ll q^{\varepsilon j}$. The following Lemma allows us to assume that $\ell$ is not too small and further that $n,m$ are not too big for a suitable choice of the parameters $Q$ and $U$.

%%%%%%%%%%%%%%%%%%%%%%%%%%%%%%%%%%%%%%%%%%%%%%%%%%%%%%%%%%%%%%%%%%%%%%%%%%%%%%%%%%%%%%%%%%%%%%%%%%%%%%%%%%%%%%%%%%%%%%%%%%%%%%%%%%%%%%%%%%%%%%%%%%%%%%%%%%%%%%%%%%%%%%%%%%%%%%%%%%%%%%%%%%%%%%%%%%%%%%%%%%%%%%%%%%%%%%%%
%%%		   LEMMA ON THE LOWER BOUND FOR THE \ell-SUM                 %%%
%%%%%%%%%%%%%%%%%%%%%%%%%%%%%%%%%%%%%%%%%%%%%%%%%%%%%%%%%%%%%%%%%%%%%%%%
\begin{lemme}\label{LemmeRestrictionQ} Set $Q^-:=N^{1/2-\varepsilon}$, $Q=LN^{1/2+\varepsilon}$ and $U=LN$.
\begin{enumerate}
\item[$(a)$] The $\ell$-sum is very small $(\ll_C q^{-C}$ for any $C>0$) unless
$$Q^-\leqslant \ell \leqslant 2Q.$$
\item[$(b)$] If $Q^-\leqslant\ell\leqslant Q$, then the integral is negligible unless
$$n\leqslant\mathcal{N}_0:= \frac{Q^{2+\varepsilon}}{N} \hspace{0.5cm} , \hspace{0.5cm} m\leqslant \mathcal{M}_0:= \frac{Q^{2+\varepsilon}}{M}.$$
\end{enumerate}
\end{lemme}
\begin{proof}
The lemma is proved by successive integration by parts. We refer to \cite[Lemmas 4.1,4.2]{binary} for the details.
\end{proof}

%%%%%%%%%%%%%%%%%%%%%%%%%%%%%%%%%%%%%%%%%%%%%%%%%%%%%%%%%%%%%%%%%%%%%%%%%%%%%%%%%%%%%%%%%%%%%%%%%%%%%%%%%%%%%%%%%%%%%%%%%%%%%%%%%%%%%%%%%%%%%%%%%%%%%%%%%%%%%%%%%%%%%%%%%%%%%%%%%%%%%%%%%%%%%%%%%%%%%%%%%%%%%%%%%%%%%%%%
%%%%%%%%%%%%%%%%%%%%%%%%%%%%%%%%%%%%%%%%%%%%%%%%%%%%%%%%%%%%%%%%%%%%%%%%%%%%%%%%%%%%%%%%%%%%%%%%%%%%%%%%%%%%%%%%%%%%%%%%%%%%%%%%%%%%%%%%%%%%%%%%%%%%%%%%%%%%%%%%%%%%%%%%%%%%%%%%%%%%%%%%%%%%%%%%%%%%%%%%%%%%%%%%%%%%%%%%
%%%%%%%%%%	  	       APPLICATION ON KUZNETSOV FORMULA				%%%%%%%%%%
%%%%%%%%%%%%%%%%%%%%%%%%%%%%%%%%%%%%%%%%%%%%%%%%%%%%%%%%%%%%%%%%%%%%%%%%%%%%%%%%%%%%%%%%%%%%%%%%%%%%%%%%%%%%%%%%%%%%%%%%%%%%%%%%%%%%%%%%%%%%%%%%%%%%%%%%%%%%%%%%%%%%%%%%%%%%%%%%%%%%%%%%%%%%%%%%%%%%%%%%%%%%%%%%%%%%%%%%
%%%%%%%%%%%%%%%%%%%%%%%%%%%%%%%%%%%%%%%%%%%%%%%%%%%%%%%%%%%%%%%%%%%%%%%%%%%%%%%%%%%%%%%%%%%%%%%%%%%%%%%%%%%%%%%%%%%%%%%%%%%%%%%%%%%%%%%%%%%%%%%%%%%%%%%%%%%%%%%%%%%%%%%%%%%%%%%%%%%%%%%%%%%%%%%%%%%%%%%%%%%%%%%%%%%%%%%%

\subsubsection{Application of the Kuznetsov formula}\label{subsubsectionKuznetsov}
We go back to \eqref{eq3} (remember that we are dealing with $\mathcal{OD}^+$) and multiply the arguments of the exponential in the $n$-sum (resp the $m$-sum) to obtain the numerators $-nd_1\overline{\ell_1'}\overline{k}$ (resp $md_2\overline{\ell_2'}\overline{k}$) over the same denominator $\ell$. Once we have done this, we execute the $k$-summation over primitive class modulo $\ell$, obtaining the complete Kloosterman sums. Applying finally a partition of unity to the interval $[Q^-,2Q]$, we are reduced to estimate $O(\log q)$ sums of the shape
\begin{equation}\label{SumShape}
\begin{split}
\frac{1}{\mathcal{Q}(NM)^{1/2}}&\sum_{d_i | \ell_i}\frac{1}{\phi^\ast(q)}\sum_{d|q}\phi(d)\mu\left(\frac{q}{d}\right)\sum_{(c,\ell_1'\ell_2')=1}c^{-1}\vartheta\left(\frac{cd_1d_2}{\mathcal{Q}}\right)\sum_{h\neq 0}\sum_{n\geqslant 1}\sum_{m\geqslant 1} \\ & \times \tau (n)\tau (m)S(hd,d_1\overline{\ell_1'}n-d_2\overline{\ell_2'}m;cd_1d_2)I(n,m,cd_1d_2),
\end{split}
\end{equation}
where $Q^-\leqslant \mathcal{Q}\leqslant Q$ and $\vartheta$ is a smooth and compactly supported function on $\mathbb{R}_{>0}$ such that $\vartheta^{(j)}\ll_j 1$ for all $j\geqslant 0$. The first obstruction for the application of the trace formula is the presence of inverses in the Kloosterman sums which are not with respect to its modulus. Indeed, $\overline{\ell_2'}$ (resp $\overline{\ell_1'}$) need to be understood modulo $cd_1$ (resp $cd_2$). We note that if the original $\ell_i'$s were squarefree, then one could take these inverses to be modulo $cd_1d_2$. 

To solve this problem, we factorize in an unique way $d_i=d_i^\ast d_i'$ with $(d_i^\ast , \ell_i')=1$ and $d_i' | (\ell_i')^{\infty}$. Now since $(cd_1^\ast d_2^\ast,d_1'd_2')=1$, we may apply the twisted multiplicativity of the Kloosterman sums, getting 
\begin{alignat*}{1}
S(hd,d_1\overline{\ell_1'}n-d_2\overline{\ell_2'}m;cd_1d_2)= & \  S(hd, \overline{(d_1'd_2')^2}(d_1\overline{\ell_1'}n-d_2\overline{\ell_2'}m); cd_1^\ast d_2^\ast) \\ \times & \ S(hd,\overline{(cd_1^\ast d_2^\ast)^2}(d_1\overline{\ell_1'}n-d_2\overline{\ell_2'}m);d_1'd_2'),
\end{alignat*}
where the inverse of $d_1'd_2'$ (resp $cd_1^\ast d_2^\ast$) is taken modulo $cd_1^\ast d_2^\ast$ (resp $d_1'd_2'$). In the first line, both $\ell_i'$ are coprime with $cd_1^\ast d_2^\ast$ and therefore, we can take the inverse to be with respect to this modulus. In the second line, the quantity $d_1\overline{\ell_1'}n-d_2\overline{\ell_2'}m$ does not depend anymore on $c$ since we are modulo $d_1'd_2'$. Following an idea of Blomer and Mili\'{c}evi\'{c} \cite{Blomer2015} and used by Topacogullari in \cite{tupac}, we separate the dependence in $c$ in the second Kloosterman sum by exploiting the orthogonality of Dirichlet characters, namely writing $v:= d_1'd_2'$, we have
$$S(hd,\overline{(cd_1^\ast d_2^\ast)^2}(d_1\overline{\ell_1'}n-d_2\overline{\ell_2'}m); v)= \frac{1}{\phi(v)}\sum_{\chi (v)}\overline{\chi}(cd_1^\ast d_2^\ast) \hat{S}_v (\overline{\chi},n,m,\ell_i,hd),$$
with 
\begin{equation}\label{Shat}
\hat{S}_v(\chi,n,m,\ell_i,hd):= \sum_{\substack{y(v) \\ (y,v)=1}}\overline{\chi}(y)S(hd\overline{y},(d_1\overline{\ell_1'}n-d_2\overline{\ell_2'}m)\overline{y};v),
\end{equation}
and where the inverse of $\ell_1'$ (resp $\ell_2'$) has to be taken modulo $d_2'$ (resp $d_1'$). We note that the trivial bound for $\hat{S}_v$ is (recall that $d=1$ or $q$ and $(v,q)=1$ since $(\ell_i,q)=1$)
$$\hat{S}_v\ll q^\varepsilon (h,v)^{1/2}v^{3/2}.$$
Altough we do not really need it in our treatment, it is in fact possible to do better. In \cite{tupac}, they obtained (see eq (3.6))
\begin{equation}\label{NonTrivialBoundhatS}
\hat{S}_v\ll q^{\varepsilon}\left(h,\frac{v}{\mathrm{cond}(\chi)}\right)v.
\end{equation}
Inserting the previous factorization of the Kloosterman sums in \eqref{SumShape}, we obtain
\begin{alignat}{1}
&\frac{1}{\mathcal{Q}(NM)^{1/2}}\sum_{d_i|\ell_i}\frac{1}{\phi(v)}\sum_{\chi (v)}\overline{\chi}(d_1^\ast d_2^\ast)\frac{1}{\phi^\ast(q)}\sum_{d|q}\phi(d)\mu\left(\frac{q}{d}\right)\nonumber \\ \times & \sum_{h\neq 0}\mathop{\sum\sum}_{n,m\geqslant 1}\tau(n)\tau(m)\hat{S}_v(\overline{\chi},n,m,\ell_i,hd)\label{SumShape2} \\ \times & \sum_{(c,\ell_1'\ell_2')=1}\overline{\chi}(c)\frac{S(hd,\overline{v^2\ell_1'\ell_2'}(d_1\ell_2'n-d_2\ell_1'm);cd_1^\ast d_2^\ast)}{c}\vartheta\left(\frac{cd_1d_2}{\mathcal{Q}}\right) I(n,m,cd_1d_2)\nonumber.
\end{alignat}
The strategy is to analyze carefully the two last lines of \eqref{SumShape2} and then to average trivially over the first line. It is convenient from now on to localize the variables $n,m$ and $h$ by applying a partition of unity. Inspired by \cite{Blomer2015}, we also localize $b:= d_1\ell_2' n-d_2\ell_1'm$ and are therefore reduced to estimate $O(\log^4 q)$ sums of the form
\begin{alignat}{1}
\mathscr{D}(\mathcal{N}&,\mathcal{M},B,H;d,\chi):= \sum_{h\asymp H}\sum_{|b|\asymp B}\mathop{\sum\sum}_{\substack{d_1\ell_2'n-d_2\ell_1'm=b \\ n\asymp \mathcal{N}, m\asymp \mathcal{M}}}\tau(n)\tau(m)\hat{S}_v(\overline{\chi},n,m,\ell_i,hd)\nonumber \\ \times & \sum_{(c,\ell_1'\ell_2')=1}\overline{\chi}(c)\frac{S(hd,\overline{v^2\ell_1'\ell_2'} b;cd_1^\ast d_2^\ast)}{c}\vartheta\left(\frac{cd_1d_2}{\mathcal{Q}}\right) \mathscr{I}(n,m,,b,h,cd_1d_2)\nonumber \\ =: & \ \mathscr{D}^++\mathscr{D}^-+\mathscr{D}^0, \label{SumShape3}
\end{alignat}
where $1\leqslant \mathcal{N}\leqslant \mathcal{N}_0$, $1\leqslant\mathcal{M}\leqslant \mathcal{M}_0$, $1\leqslant H\leqslant LN/d$, and where $\mathscr{D}^0$ (respectively $\mathscr{D}^+$, $\mathscr{D}^-$) denotes the contribution of $b=0$ (respectively $b>0$, $b<0$) and $\mathscr{I}(n,m,b,h,cd_1d_2)=G(n,m,|b|,h)I(n,m,cd_1d_2)$ with $G$ a smooth and compactly supported function on $[\mathcal{N},2\mathcal{N}]\times [\mathcal{M},2\mathcal{M}]\times [B,2B]\times [H,2H]$ satisfying 
$$G^{(i,j,k,p)}\ll \mathcal{N}^{-i}\mathcal{M}^{-j}B^{-k}H^{-p}.$$
\begin{remq}\label{RemarkSizeB} The size if $B$ depends on the sign of $d_1\ell_2'n-d_2\ell_1'm=b.$ If $b>0$, then $B\leqslant d_1\ell_2'\mathcal{N}\leqslant L^2\mathcal{N}$ while for $b<0$, $B\leqslant L^2\mathcal{M}$ which is much larger (c.f. Lemma \ref{LemmeRestrictionQ} (b)). 
\end{remq}

\vspace{0.2cm}

%%%%%%%%%%%%%%%%%%%%%%%%%%%%%%%%%%%%%%%%%%%%%%%%%%%%%%%%%%%%%%%%%%%%%%%%%%%%%%%%%%%%%%%%%%%%%%%%%%%%%%%%%%%%%%%%%%%%%%%%%%%%%%%%%%%%%%%%%%%%%%%%
%%% 					PARAGRAPH :	EVALUATION OF D0                    %%%
%%%%%%%%%%%%%%%%%%%%%%%%%%%%%%%%%%%%%%%%%%%%%%%%%%%%%%%%%%%%%%%%%%%%%%%%%%%%%%%%%%%%%%%%%%%%%%%%%%%%%%%%%%%%%%%%%%%%%%%%%%%%%%%%%%%%%%%%%%%%%%%%

\noindent $\mathbf{Evaluation \ of}$ $\mathscr{D}^0$ : To estimate the contribution of $b=0$, we use the bound $Y_0(z)\ll z^{-1/2}$, the fact that $x\asymp N, y\asymp M, n\asymp \mathcal{N}, m\asymp \mathcal{M}$ and \eqref{BoundDelta} for the delta function which allows us to bound the integral : 
\begin{alignat*}{1}
I(n,m,cd_1d_2)\ll & \ \frac{c(d_1d_2)^{1/2}}{(MN\mathcal{MN})^{1/4}}\int_0^\infty\int_0^\infty |E(x,y,cd_1d_2)|dxdy\\ \ll & \ \frac{(MN)^{3/4}}{(\mathcal{MN})^{1/4}(d_1d_2)^{1/2} Q}.
\end{alignat*}
We now use \eqref{NonTrivialBoundhatS}, the Weil bound for Kloosterman sums and $H\leqslant LN/d$, $\mathcal{N}\leqslant\mathcal{N}_0$ to obtain (recall that $v=d_1'd_2'$)
\begin{alignat}{1}
\mathscr{D}^0\ll & \ q^\varepsilon \frac{(NM\mathcal{N})^{3/4}v(d_1^\ast d_2^\ast)^{1/2}}{\mathcal{M}^{1/4}Q}\sum_{h\asymp H}\sum_{c\leqslant\frac{\mathcal{Q}}{d_1d_2}}\frac{(h,v)(hd,cd_1^\ast d_2^\ast)^{1/2}}{c^{1/2}} \nonumber \\ \ll & \ q^\varepsilon Ld^{-1}\frac{(NM\mathcal{N})^{3/4} v^{1/2} N\mathcal{Q}^{1/2}}{(d_1d_2)^{1/2}\mathcal{M}^{1/4}Q}\ll q^{\varepsilon} L d^{-1} \frac{v^{1/2} M^{3/4} N\mathcal{Q}^2}{(d_1d_2)^{1/2}Q}. \label{BoundD0}
\end{alignat}
%%%%%%%%%%%%%%%%%%%%%%%%%%%%%%%%%%%%%%%%%%%%%%%%%%%%%%%%%%%%%%%%%%%%%%%%%%%%%%%%%%%%%%%%%%%%%%%%%%%%%%%%%%%%%%%%%%%%%%%%%%%%%%%%%%%%%%%%%%%%%%%%
%%%					END OF THE EVALUATION						   %%%

We come back to \eqref{SumShape3} and we write $\mathscr{D}^{\pm}$ in an uniform way (recall that $v | (\ell_1'\ell_2')^\infty$ and $d_i=d_i^\ast d_i'$)
\begin{equation}\label{SumShape4}
\begin{split}
\mathscr{D}^{\pm}=4\pi^2d_1d_2&\sqrt{\ell_1'\ell_2'}\sum_{h\asymp H}\sum_{b\asymp B}\mathop{\sum\sum}_{\substack{d_1\ell_2' n-d_2\ell_1'm=b \\ n\asymp \mathcal{N}, m\asymp\mathcal{M}}}\tau(n)\tau(m)\hat{S}_v(\overline{\chi},n,m,\ell_i,hd) \\ \times & \sum_{(c,\ell_1'\ell_2' v^2)=1}\overline{\chi}(c)\frac{S(hd,\overline{v^2\ell_1'\ell_2'} b;cd_1^\ast d_2^\ast)}{cd_1d_2\sqrt{\ell_1'\ell_2'}}\Phi \left( \frac{4\pi \sqrt{|b|hd}}{cd_1d_2\sqrt{\ell_1'\ell_2'}}\right),
\end{split}
\end{equation}
where  the function $\Phi$ depends also on the variables $n,m,b$ and $h$ and is defined by
\begin{alignat*}{1}
\Phi(z,n,m,b,h):=G(n,m,|b|,h)&\vartheta\left(\frac{4\pi\sqrt{|b|hd}}{z\mathcal{Q}\sqrt{\ell_1'\ell_2'}}\right)\int_0^\infty\int_0^\infty E\left(x,y,\frac{4\pi\sqrt{|b|hd}}{z\sqrt{\ell_1'\ell_2'}}\right) \\\times & Y_0\left(zd_1\sqrt{\frac{\ell_1'\ell_2'}{|b|hd}nx}\right)Y_0\left(zd_2\sqrt{\frac{\ell_1'\ell_2'}{|b|hd}my}\right)dxdy.
\end{alignat*}
\begin{remq} We can always assume that we are treating the case where $h\asymp H$ is positive since otherwise, we write $h\leftrightarrow -h$ and use $S(-hq,\overline{\ell_1'\ell_2'}b;cd_1d_2)=S(hq,\overline{\ell_1'\ell_2'}(-b);cd_1d_2)$.
\end{remq}

\vspace{0.2cm}

%%%%%%%%%%%%%%%%%%%%%%%%%%%%%%%%%%%%%%%%%%%%%%%%%%%%%%%%%%%%%%%%%%%%%%%%%%%%%%%%%%%%%%%%%%%%%%%%%%%%%%%%%%%%%%%%%%%%%%%%%%%%%%%%%%%%%%%%%%%%%%%%
%%%			PARAGRAPH : ANALYSIS OF THE FUNCTION PHI				   %%%
%%%%%%%%%%%%%%%%%%%%%%%%%%%%%%%%%%%%%%%%%%%%%%%%%%%%%%%%%%%%%%%%%%%%%%%%%%%%%%%%%%%%%%%%%%%%%%%%%%%%%%%%%%%%%%%%%%%%%%%%%%%%%%%%%%%%%%%%%%%%%%%%

%%%%%%%%%%%%%%%%%%%%%%%%%%%%%%%%%%%%%%%%%%%%%%%%%%%%%%%%%%%%%%%%%%%%%%%%%%%%%%%%%%%%%%%%%%%%%%%%%%%%%%%%%%%%%%%%%%%%%%%%%%%%%%%%%%%%%%%%%%%%%%%%
%%%		First Lemma
%%%%%%%%%%%%%%%%%%%%%%%%%%%%%%%%%%%%%%%%%%%%%%%%%%%%%%%%%%%%%%%%%%%%%%%%
\begin{proposition}\label{Lemmefunctionf} The function $\Phi$ is $C_c^\infty(\mathbb{R}^5)$ with each variable supported in 
$$z\asymp Z:= \frac{\sqrt{BHd}}{\mathcal{Q}\sqrt{\ell_1'\ell_2'}}, \ n\asymp\mathcal{N}, \ m\asymp\mathcal{M}, \ b\asymp B, \ h\asymp H.$$
Further, it satisfies the following bound on the partial derivatives
\begin{equation}\label{BoundDerivatives}
\Phi^{(\bm{\alpha})}\ll_{\bm{\alpha}}\frac{M^{3/4}N^{1/4}}{L(d_1d_2)^{1/2}(\mathcal{MN})^{1/4}} Z^{-\alpha_1}\mathcal{N}^{-\alpha_2}\mathcal{M}^{-\alpha_3}B^{-\alpha_4} H^{-\alpha_5},
\end{equation}
for any multi-index $\bm{\alpha}=(\alpha_1,...,\alpha_5)\in\mathbb{N}^5.$
\end{proposition}
\begin{proof}
Setting $\xi :=4\pi \sqrt{|b|hd/\ell_1'\ell_2'}$, using the bound $Y_0\ll z^{-1/2}$, the fact $\Delta_\ell(u)\ll (\ell Q)^{-1}$ provides by \eqref{BoundDelta}, the ranges $x\asymp N$, $y\asymp M$ and the choice of $Q=LN^{1/2+\varepsilon}$ lead to
\begin{alignat*}{1}
\Phi\ll & \ \frac{\xi}{z}(d_1d_2)^{-1/2}(MN\mathcal{NM})^{-1/4}\int_0^\infty\int_0^\infty |E(x,y,\xi/z)|dxdy \\ \ll & \ \frac{(d_1d_2)^{-1/2}(MN)^{3/4}}{(\mathcal{MN})^{1/4}Q} \ll \frac{M^{3/4}N^{1/4}}{L(d_1d_2)^{1/2}(\mathcal{NM})^{1/4}}.  
\end{alignat*}
For the second part, we take the derivatives under the sign of the integral and we use the following estimations \cite[Lemma C.2]{kmv}
$$z^i Y_0^{(i)}(z) \ll_j \frac{1+|\log z|}{(1+z)^{1/2}},$$
and
$$\frac{\partial^i}{\partial \ell^i}E(x,y,\ell) \ll_i \frac{1}{\ell^i(\ell Q)}.$$
We mention that when we derive $\varphi(\ell_1x-\ell_2y-hd)\Delta_\ell(\ell_1x-\ell_2y-hd)$ with respect to $h$, we catch a factor $d/LN$ but since $H\ll LN/d$, we get the desired $1/H$.
\end{proof}

\noindent $\mathbf{Applying \ the \ trace \ formula}$ : Before applying the Kuznetsov trace formula to the second line in \eqref{SumShape4}, we need the following identity (c.f. $(9.1)$-$(9.2)$, \cite{1982} or $(2.3)$ in \cite{tupac}) which allows us to get rid of the inverses in the Kloosterman sum by moving to a suitable cusp (apply this identity with $r=\ell_1'\ell_2'v^2$, $s=d_1^\ast d_2^\ast$ and $c=c$) :
\begin{alignat*}{1}
\sum_{(c,\ell_1'\ell_2'v^2)=1}\overline{\chi}(c)&\frac{S(hd,\overline{v^2\ell_1'\ell_2'}b;cd_1^\ast d_2^\ast)}{cd_1^\ast d_2^\ast\sqrt{v^2\ell_1'\ell_2'}}\Phi\left(\frac{4\pi\sqrt{|b|hd}}{cd_1^\ast d_2^\ast\sqrt{v^2\ell_1'\ell_2'}},n,m,b,h\right) \\ = & \ e\left(-\frac{b\overline{d_1^\ast d_2^\ast}}{v^2\ell_1'\ell_2'}\right)\sum_{\gamma}^{\Gamma_0(v\ell_1\ell_2)}\frac{S_{\infty\mathfrak{a}}^{\chi}(hd,b;\gamma)}{\gamma}\Phi\left(\frac{4\pi\sqrt{|b|hd}}{\gamma},n,m,b,h\right),
\end{alignat*}
where $\mathfrak{a}:= 1/d_1^\ast d_2^\ast$ is a singular cusp for the congruence group $\Gamma_0(v\ell_1\ell_2)$. 
\begin{comment}
The fact that $\mathfrak{a}$ is singular follows from the following observation : with the notations $s= d_1^\ast d_2^\ast$ and $r=\ell_1'\ell_2' v^2$, if we choose the scaling matrix of $\mathfrak{a}$ to be 
$$\tau_\mathfrak{a}:=\left( \begin{array}{cc} \sqrt{r} & 0 \\ s\sqrt{r} & 1/\sqrt{r} \end{array} \right) \in \mathrm{SL}_2(\mathbb{R}),$$
then the stabilizer of $\mathfrak{a}$ is generated by 
$$\gamma_\mathfrak{a} := \tau_{\mathfrak{a}}\left( \begin{array}{cc} 1 & 1 \\ 0 & 1 \end{array} \right)\tau_{\mathfrak{a}}^{-1}= \left( \begin{array}{cc} 1-sr & r \\ -s^2r & 1+sr \end{array} \right),$$
so that $\chi(\gamma_\mathfrak{a})=1$. 
\end{comment}
We now apply Kuznetsov formula (c.f. Proposition \ref{Kuznetsov}) to the $\gamma$-sum and we write separatly 
\begin{alignat*}{1}
\mathscr{D}^+ = & \ \frac{4\pi \sqrt{\ell_1'\ell_2'}}{(d_1d_2)^{-1}}\sum_{h\asymp H}\sum_{b\asymp B}e\left(-\frac{b\overline{d_1^\ast d_2^\ast}}{v^2\ell_1'\ell_2'}\right)\mathop{\sum\sum}_{\substack{d_1\ell_2' n-d_2\ell_1' m=b \\ n\asymp \mathcal{N}, m\asymp \mathcal{M}}}\tau(n)\tau(m)\hat{S}_v(\overline{\chi},n,m,hd) \\ 
& \times \left( \mathscr{H}(n,m,b,h)+\mathscr{M}^+(n,m,b,h)+\mathscr{E}^+(n,m,b,h)\right), \\ 
\mathscr{D}^- = & \ \frac{4\pi \sqrt{\ell_1'\ell_2'}}{(d_1d_2)^{-1}}\sum_{h\asymp H}\sum_{b\asymp B}e\left(-\frac{b\overline{d_1^\ast d_2^\ast}}{v^2\ell_1'\ell_2'}\right)\mathop{\sum\sum}_{\substack{d_1\ell_2' n-d_2\ell_1' m=b \\ n\asymp \mathcal{N}, m\asymp \mathcal{M}}}\tau(n)\tau(m)\hat{S}_v(\overline{\chi},n,m,hd) \\
& \times \left(\mathscr{M}^-(n,m,b,h)+\mathscr{E}^-(n,m,b,h)\right),
\end{alignat*}
where $\mathscr{H}$, $\mathscr{M}$ and $\mathscr{E}$ denote the contribution of the holomorphic part, the Maa\ss \ cusp forms and the Eisenstein spectrum and are given respectively by (c.f \eqref{Kuznetsov>0})
\begin{alignat*}{1}
\mathscr{H}^+(n,m,b,h)= & \ \sum_{\substack{k\geqslant 2 \\ k\equiv \kappa \ (2)}}\dot{\Phi}_{n,m,b,h}(k)\Gamma(k)\sum_{f\in\mathcal{B}_k(v\ell_1\ell_2,\chi)}\sqrt{bhd}\overline{\rho_{f,\infty}}(hd)\rho_{f,\mathfrak{a}}(b), \\
\mathscr{M}^+(n,m,b,h) = & \ \sum_{f\in\mathcal{B}(v\ell_1\ell_2,\chi)}\widehat{\Phi}_{n,m,b,h}(t_f)\frac{\sqrt{bhd}}{\cosh(\pi t_f)}\overline{\rho_{f,\infty}}(hd)\rho_{f,\mathfrak{a}}(b), \\ \mathscr{E}^+(n,m,b,h) = & \ \sum_{\substack{\chi_1\chi_2=\chi \\ f\in \mathcal{B}(\chi_1,\chi_2)}}\frac{1}{4\pi}\int\limits_\mathbb{R}\widehat{\Phi}_{n,m,b,h}(t)\frac{\sqrt{bhd}}{\cosh(\pi t)}\overline{\rho_{f,\infty}}(hd,t)\rho_{f,\mathfrak{a}}(b,t)dt.
\end{alignat*}
We have the same expressions for $\mathscr{M}^-$ and $\mathscr{E}^-$, but with $\check{\Phi}_{n,m,b,h}$ instead of $\widehat{\Phi}_{n,m,b,h}$ (see \eqref{Kuznetsov<0}). We will analyze in detail $\mathscr{D}^-$, whose contribution is bigger than the plus case. This is due to the fact that if $b>0$, then $B$ is at most $\mathcal{N}\ll q^\varepsilon L^{2}$ while for $b<0$, $B$ could be of size $\mathcal{M}\ll q^\varepsilon L^2 N/M$ (c.f. Remark \ref{RemarkSizeB}). Furthermore, the holomorphic setting and the continuous spectrum will give a better bound than the discrete part since the Ramanujan-Petersson conjecture is true for both of them. Finally, since the treatment of these three terms is similar, we only focus on the Maa\ss \ cusp forms in $\mathscr{D}^-$.

\vspace{0.2cm}

%%%			PARAGRAPH : SPECTRAL ANALYSIS OF D-					   %%%
%%%%%%%%%%%%%%%%%%%%%%%%%%%%%%%%%%%%%%%%%%%%%%%%%%%%%%%%%%%%%%%%%%%%%%%%%%%%%%%%%%%%%%%%%%%%%%%%%%%%%%%%%%%%%%%%%%%%%%%%%%%%%%%%%%%%%%%%%%%%%%%%

\noindent $\mathbf{Spectral \ analysis \ of}$ $\mathscr{D}^-$ : As said in the previous paragraph, we only focus on the discrete spectrum, writing  $\mathscr{D}^{-,\mathrm{M}}$ for its contribution to $\mathscr{D}^-$. By $\mathscr{D}_K^{-,\textnormal{M}}$, we mean that we restrict the spectral parameter to the dyadic interval $K\leqslant t_f< 2K$ in the definition of $\mathscr{D}^{-,\textnormal{M}}$. Using Proposition \ref{Lemmefunctionf} and Lemma \ref{LemmaBesselTransform} (eq \eqref{BesseTransform3}), we see that we can restrict our attention to $K\leqslant q^\varepsilon Z$ at the cost of $O(q^{-100})$. We now separate the variables in $\check{\Phi}_{n,m,b,h}(t)$ using the Mellin inversion formula in $n,m,b,h$ :
$$\check{\Phi}_{n,m,b,h}(t)=\frac{1}{(2\pi i)^4}\int_{(0)}\int_{(0)}\int_{(0)}\int_{(0)}\frac{\widetilde{\check{\Phi}(t)}(s_1,...,s_4)}{n^{s_1}m^{s_2}|b|^{s_3}h^{s_4}}ds_4ds_3ds_2ds_1,$$
where the Mellin transform equals
\begin{equation}\label{DefinitionMellin-Bessel}
\widetilde{\check{\Phi}(t)}(s_1,...,s_4) = \int_{(\mathbb{R}_{>0})^4}\check{\Phi}_{\mathfrak{n},\mathfrak{m},\mathfrak{b},\mathfrak{h}}(t)\mathfrak{n}^{s_1}\mathfrak{m}^{s_2}\mathfrak{b}^{s_3}\mathfrak{h}^{s_4}\frac{d\mathfrak{n}d\mathfrak{m}d\mathfrak{b}d\mathfrak{h}}{\mathfrak{n}\mathfrak{m}\mathfrak{b}\mathfrak{h}}.
\end{equation}
By virtue of Proposition \ref{Lemmefunctionf} (the bound \eqref{BoundDerivatives}), we see that we can restrict the supports of the integrals to $|\Im m (s_i)|\leqslant (Kq)^{\varepsilon}$ for a cost of $O((Kq)^{-100})$. We have therefore 
\begin{equation}\label{DefinitionDMaass}
\mathscr{D}_K^{-,\textnormal{M}}=\frac{4\pi^2d_1d_2\sqrt{\ell_1'\ell_2'}}{(4\pi i)^4}\iiiint\limits_{|\Im m (s_i)|\leqslant (Kq)^{\varepsilon}}\mathscr{B}_K^{-,\textnormal{M}}(s_1,...,s_4)ds_4ds_3ds_2 ds_1 + O\left((Kq)^{-100}\right),
\end{equation}
where we defined
\begin{equation}\label{DefinitionBMaass}
\begin{split}
\mathscr{B}_K^{-,\textnormal{M}}(s_1,...,s_4):= & \ \sum_{\substack{f\in\mathcal{B}(v\ell_1\ell_2,\chi) \\  K\leqslant |t_f|<2K}}\frac{\widetilde{\check{\Phi}(t_f)}(s_1,...,s_4)}{\cosh(\pi t_f)} \sum_{h\asymp H}h^{-s_4} \\ \times & \ \sum_{b\asymp B}|b|^{-s_3}\alpha(b,h,s_1,s_2)\sqrt{hd|b|}\overline{\rho_{f,\infty}}(hd)\rho_{f,\mathfrak{a}}(b),
\end{split}
\end{equation}
and  
\begin{equation}
\alpha(b,h,s_1,s_2):= e\left(-\frac{b\overline{d_1^\ast d_2^\ast}}{v^2\ell_1'\ell_2'}\right) \mathop{\sum\sum}_{\substack{d_1\ell_2'n-d_2\ell_1'm=b \\ n\asymp\mathcal{N} \\ m\asymp \mathcal{M}}}\frac{\tau(n)\tau(m)}{n^{s_1}m^{s_2}}\hat{S}_v(\overline{\chi},n,m,h).
\end{equation}
%%%%%%%%%%%%%%%%%%%%%%%%%%%%%%%%%%%%%%%%%%%%%%%%%%%%%%%%%%%%%%%%%%%%%%%%
%%%%%%%%%%%%%%%%%%%%%%%%%%%%%%%%%%%%%%%%%%%%%%%%%%%%%%%%%%%%%%%%%%%%%%%%
Since we want to apply Cauchy-Schwarz in \eqref{DefinitionBMaass} to make the square of the $h$ and $b$ sum appear in order to use the large sieve inequality, we need to separate $h$ from $b$ in $\alpha(b,h)$. Using the Definition \eqref{Shat} of $\hat{S}_v(\overline{\chi},n,m,h)$ and opening the Kloosterman sum, we have
%%%%%%%%%%%%%%%%%%%%%%%%%%%%%%%%%%%%%%%%%%%%%%%%%%%%%%%%%%%%%%%%%%%%%%%%
%%%%%%%%%%%%%%%%%%%%%%%%%%%%%%%%%%%%%%%%%%%%%%%%%%%%%%%%%%%%%%%%%%%%%%%%
\begin{equation}\label{phi^2}
\begin{split}
\mathscr{B}_K^{-,\mathrm{M}}(s_1,...,s_4) = \sum_{\substack{x,y (v) \\ (xy,v)=1}}\chi(y)\mathscr{A}_K(x,y,s_1,...,s_4),
\end{split}
\end{equation}
with this time
\begin{equation}\label{DefinitionAmass}
\begin{split}
\mathscr{A}_K(x,y,s_1,...,s_4) := & \ \sum_{\substack{f\in\mathcal{B}(v\ell_1\ell_2,\chi) \\ K\leqslant |t_f|<2K}}\frac{\widetilde{\check{\Phi}(t_f)}(s_1,...,s_4)}{\cosh(\pi t_f)}\sum_{h\asymp H}\delta (h,s_4) \\ & \times \sum_{b\asymp B}|b|^{-s_3}\omega(b,s_1,s_2)\sqrt{hd|b|}\overline{\rho_{f,\infty}}(hd)\rho_{f,\mathfrak{a}}(b),
\end{split}
\end{equation}
where
$$\delta(h,s_4) := h^{-s_4}e\left(\frac{hd\overline{y}x}{v}\right),$$
and
$$\omega(b,s_1,s_2,s_3):= e\left(-\frac{b\overline{d_1^\ast d_2^\ast}}{v^2\ell_1'\ell_2'}\right) \mathop{\sum\sum}_{\substack{d_1\ell_2'n-d_2\ell_1'm=b \\ n\asymp\mathcal{N} \\ m\asymp \mathcal{M}}}\frac{\tau(n)\tau(m)}{n^{s_1}m^{s_2}}e\left(\frac{\left(d_1\overline{\ell_1'}n-d_2\overline{\ell_2'}m\right)\overline{xy}}{v}\right).$$
%%%%%%%%%%%%%%%%%%%%%%%%%%%%%%%%%%%%%%%%%%%%%%%%%%%%%%%%%%%%%%%%%%%%%%%%
%%%%%%%%%%%%%%%%%%%%%%%%%%%%%%%%%%%%%%%%%%%%%%%%%%%%%%%%%%%%%%%%%%%%%%%%
Since the supports of the integrals in \eqref{DefinitionDMaass} are restricted to $|\Im m (s_i)|\leqslant (Kq)^{\varepsilon}$, we can just estimate the quantity \eqref{DefinitionBMaass} and then average trivially over the $s_i$-integrals for an error cost of $(Kq)^{4\varepsilon}$. In fact, we will analyze $\mathscr{A}_K(x,y,s_1,...,s_4)$ and then apply the trivial bound $\mathscr{B}_K \leqslant\phi(v)^2 \sup_{x,y,s_i}|\mathscr{A}_K(x,y,s_1,...,s_4)|$. Using Cauchy-Schwarz inequality, we infer
%%%%%%%%%%%%%%%%%%%%%%%%%%%%%%%%%%%%%%%%%%%%%%%%%%%%%%%%%%%%%%%%%%%%%%%%
%%%%%%%%%%%%%%%%%%%%%%%%%%%%%%%%%%%%%%%%%%%%%%%%%%%%%%%%%%%%%%%%%%%%%%%%
%%%		APPLICATION DE CAUCHY-SCHWARZ POUR LE GRAND CRIBLE
%%%%%%%%%%%%%%%%%%%%%%%%%%%%%%%%%%%%%%%%%%%%%%%%%%%%%%%%%%%%%%%%%%%%%%%%
\begin{alignat}{1}
&\left|\mathscr{A}_K(x,y,s_1,...,s_4)\right|\leqslant \sup_{\substack{K\leqslant t<2K \\ \Re e (s_i)=0}}\left|\widetilde{\check{\Phi}(t)}(s_1,...,s_4)\right| \label{CauchySBeforeLargeSieve} \\ 
 & \times  \ \left( \sum_{\substack{f\in\mathcal{B}(v\ell_1\ell_2,\chi) \\ K\leqslant |t_f|<2K}}\frac{(1+|t_f|)^{\kappa}}{\cosh(\pi t_f)}\left|\sum_{h\asymp H}\delta(h,s_4)\sqrt{hd}\overline{\rho_{f,\infty}}(hd)\right|^2\right)^{1/2} \nonumber\\ & \times \ \left( \sum_{\substack{f\in\mathcal{B}(v\ell_1\ell_2,\chi) \\ K\leqslant |t_f|<2K}}\frac{(1+|t_f|)^{-\kappa}}{\cosh(\pi t_f)}\left|\sum_{b\asymp B}|b|^{-s_3}\omega(b,s_1,s_2) \sqrt{|b|}\rho_{f,\mathfrak{a}}(b)\right|^2\right)^{1/2}\nonumber,
\end{alignat}
where $\kappa\in\{0,1\}$ satisfies $\chi(-1)=(-1)^\kappa$. We mention that we implicitly used the fact that $\cosh(\pi t_f)$ is always positive since $|\Im m (t_f)|\leqslant \theta=7/64$ by \eqref{BoundSpectralParamater} (it is enough to have $|\Im m (t_f)|<1/2$). Before applying the spectral large sieve, we need to control the size of the Mellin-Kuznetsov transform $$t\mapsto \widetilde{\check{\Phi}(t)}(s_1,...,s_4).$$
To do this, we return to Definitions \eqref{DefinitionMellin-Bessel} and \eqref{definitionBesselTransform} and note (by permutation of integrals) that this is in fact the Bessel transform of the function 
$$z\mapsto \Psi(z,s_1,...,s_4) :=\mathop{\iiiint}_{(\mathbb{R}_{>0})^4}\Phi(z,\mathfrak{n},\mathfrak{m},\mathfrak{b},\mathfrak{h})\mathfrak{n}^{s_1}\mathfrak{m}^{s_2}\mathfrak{b}^{s_3}\mathfrak{h}^{s_4}\frac{d\mathfrak{n}d\mathfrak{m}d\mathfrak{b}d\mathfrak{h}}{\mathfrak{n}\mathfrak{m}\mathfrak{b}\mathfrak{h}}.$$
Using again Proposition \ref{Lemmefunctionf}, we see that the support of $\Psi$ is $z\asymp Z$ and that it satisfies the uniform bound (recall that $\Re e (s_i)=0$)
$$\Psi^{(i)}\ll q^{\varepsilon}\frac{M^{3/4}N^{1/4}}{L(d_1d_2)^{1/2}(\mathcal{MN})^{1/4}} Z^{-i}.$$
Therefore, it follows from Lemma \ref{LemmaBesselTransform} (the bound \eqref{BesselTransform1}) that 
\begin{equation}\label{BoundMellin-Bessel}
\widetilde{\check{\Phi}(t)}(s_1,...,s_4)\ll q^{\varepsilon}\frac{M^{3/4}N^{1/4}}{ZL(d_1d_2)^{1/2}(\mathcal{MN})^{1/4}}.
\end{equation}
%%%%%%%%%%%%%%%%%%%%%%%%%%%%%%%%%%%%%%%%%%%%%%%%%%%%%%%%%%%%%%%%%%%%%%%%
%%%%%%%%%%%%%%%%%%%%%%%%%%%%%%%%%%%%%%%%%%%%%%%%%%%%%%%%%%%%%%%%%%%%%%%%
%%%          BEGINING OF THE MAJORATIONS
%%%%%%%%%%%%%%%%%%%%%%%%%%%%%%%%%%%%%%%%%%%%%%%%%%%%%%%%%%%%%%%%%%%%%%%%
We substitute the bound above in the first line of \eqref{CauchySBeforeLargeSieve}. In the second line, we exploit the fact that we are at the cusp $\infty$ by mean of the Hecke relation between the Fourier coefficients and the eigenvalues (c.f. \eqref{RelationHecke2}), namely (we recall that $(h,q)=1$ and we use this hypothesis only here)
$$\sqrt{hd}\rho_{f,\infty}(hd)=\lambda_f(d)\sqrt{h}\rho_{f,\infty}(h).$$
Note that we also used the fact that $q$ is coprime to the level of the group $\Gamma_0(v\ell_1\ell_2)$ since $(\ell_1\ell_2,q)=1$ and $v|\ell_1\ell_2$.
We now use the bound $|\lambda_f(d)|\leqslant 2d^\theta$ (c.f \eqref{BoundHeckeEigenvalue1} and recall that either $d=1$ or $d=q$ is prime), the large sieve inequality (c.f Theorem \ref{TheoremSpectralLargeSieve}), the fact that $\mu(\mathfrak{a})=(v\ell_1\ell_2)^{-1}$ (c.f \eqref{mu(a)}), cond$(\chi)\leqslant v$, the two bounds
$$||\delta||_H \leqslant H^{1/2} \ , \ ||\omega||_B \ll q^{\varepsilon} \mathcal{N}B^{1/2}$$
and we obtain
$$\mathscr{A}_K \ll q^{\varepsilon}d^\theta\frac{M^{3/4}N^{1/4}(BH)^{1/2}\mathcal{N}^{3/4}}{ZL(d_1d_2)^{1/2}\mathcal{M}^{1/4}}\left(K+\frac{v^{-1/4}B^{1/2}}{(\ell_1\ell_2)^{1/2}}\right)\left(K+\frac{v^{-1/4}H^{1/2}}{(\ell_1\ell_2)^{1/2}}\right).$$
For $K$, we have since $\mathcal{Q}\geqslant N^{1/2-\varepsilon}$ and $H\leqslant NL/d$,
$$K\leqslant q^\varepsilon Z = q^\varepsilon \frac{(BHd)^{1/2}}{\mathcal{Q}(\ell_1'\ell_2')^{1/2}}\ll q^\varepsilon\frac{(LN)^{1/2}}{\mathcal{Q}}\frac{B^{1/2}}{(\ell_1'\ell_2')^{1/2}}\ll q^\varepsilon \frac{(LB)^{1/2}}{(\ell_1'\ell_2')^{1/2}}.$$
Hence,
$$\mathscr{A}_K \ll q^{\varepsilon}d^\theta L^{-1/2}\frac{M^{3/4}N^{1/4}BH^{1/2}\mathcal{N}^{3/4}}{Z (\ell_1\ell_2)^{1/2}\mathcal{M}^{1/4}(\ell_1'\ell_2')^{1/2}}\left((LB)^{1/2}+H^{1/2}\right).$$
Using $Z\asymp \mathcal{Q}^{-1}(\ell_1'\ell_2')^{-1/2}(BHd)^{1/2}$ leads to 
\begin{alignat*}{1}
\mathscr{A}_K \ll & \  q^{\varepsilon}d^{-1/2+\theta}L^{-1/2}\frac{M^{3/4}N^{1/4}B^{1/2}\mathcal{N}^{3/4}\mathcal{Q}}{\mathcal{M}^{1/4}(\ell_1\ell_2)^{1/2}}\left((LB)^{1/2}+H^{1/2}\right) \\ =: & \ \mathscr{A}_K(B)+\mathscr{A}_K(H).
\end{alignat*}
For the first expression, we have using $B\leqslant L^2\mathcal{M}$ and the maximum values of $\mathcal{M}$ and $\mathcal{N}$ given by Lemma \ref{LemmeRestrictionQ} (b)
\begin{alignat*}{1}
\mathscr{A}_K(B) \ll & \  q^{\varepsilon}d^{-1/2+\theta}L^2\frac{(M\mathcal{MN})^{3/4}N^{1/4}\mathcal{Q}}{(\ell_1\ell_2)^{1/2}} \\  \ll & \ q^{\varepsilon}d^{-1/2+\theta}L^2\frac{Q^3 \mathcal{Q}}{(\ell_1\ell_2N)^{1/2}}\ll q^{\varepsilon}d^{-1/2+\theta}L^{5}\frac{N\mathcal{Q}}{(\ell_1\ell_2)^{1/2}}.
\end{alignat*}
For the second term, we have using also $H\leqslant LN/d$
\begin{alignat*}{1}
\mathscr{A}_K(H) = & \ q^{\varepsilon}d^{-1/2+\theta}L^{-1/2}\frac{M^{3/4}N^{1/4}(BH)^{1/2}\mathcal{N}^{3/4}\mathcal{Q}}{\mathcal{M}^{1/4}(\ell_1\ell_2)^{1/2}} \\ 
\ll & \ q^{\varepsilon}d^{-1+\theta}L\frac{(NM\mathcal{N})^{3/4}\mathcal{M}^{1/4}\mathcal{Q}}{(\ell_1\ell_2)^{1/2}}\ll q^{\varepsilon}d^{-1+\theta}L\frac{Q^2M^{1/2}\mathcal{Q}}{(\ell_1\ell_2)^{1/2}} \\ \ll & \ q^{\varepsilon}d^{-1+\theta}L^{3}\frac{NM^{1/2}\mathcal{Q}}{(\ell_1\ell_2)^{1/2}}.
\end{alignat*}

\vspace{0.2cm}

%%%    		  PARAGRAPH : CONCLUSION OF THEOREM 1
%%%%%%%%%%%%%%%%%%%%%%%%%%%%%%%%%%%%%%%%%%%%%%%%%%%%%%%%%%%%%%%%%%%%%%%%%%%%%%%%%%%%%%%%%%%%%%%%%%%%%%%%%%%%%%%%%%%%%%%%%%%%%%%%%%%%%%%%%%%%%%%%

\noindent $\mathbf{Conclusion \ of \ Theorem}$ \ref{Theorem1} : We insert these two estimations first in \eqref{phi^2}, so that it will be multiplied by $\phi(v)^2$. We next multiply by $d_1d_2\sqrt{\ell_1'\ell_2'}$ as in \eqref{DefinitionDMaass}. Finally, we replace the two last lines of \eqref{SumShape2} by these bounds and execute the first line summation, obtaining that the contribution of $\mathscr{D}^{-}$ to $\mathcal{OD}^{E+}(\ell_1,\ell_2,N,M;q)$ is 
\begin{equation}\label{Bound1Theorem}
q^{\varepsilon-1/2+\theta}(\ell_1\ell_2)^{3/2}\left(\left(\frac{L^{6}N}{q}\right)^{1/2}+\left(\frac{L^{10}N}{M}\right)^{1/2}\right)\ll q^{\varepsilon-1/2+\theta}L^5(\ell_1\ell_2)^{3/2}\left(\frac{N}{M}\right)^{1/2},
\end{equation}
since $M/q\leqslant 1$. We do exactly the same thing with $\mathscr{D}^0$ (see \eqref{BoundD0}), getting that its contribution to $\mathcal{OD}^{E,+}$ is at most 
\begin{equation}\label{Bound2Theorem}
q^\varepsilon L^{8}\frac{M^{1/4}N^{1/2}}{q}\ll q^\varepsilon L^{8}\left(\frac{N}{q^2}\right)^{1/4},
\end{equation}
which completes the proof of Theorem \ref{Theorem1}.

%%%					COMBINING SECTION 4.1 AND 4.2
%%%%%%%%%%%%%%%%%%%%%%%%%%%%%%%%%%%%%%%%%%%%%%%%%%%%%%%%%%%%%%%%%%%%%%%%%%%%%%%%%%%%%%%%%%%%%%%%%%%%%%%%%%%%%%%%%%%%%%%%%%%%%%%%%%%%%%%%%%%%%%%%

\subsubsection{Combining the error terms of sections \ref{Sectionl-adic} and \ref{SectionSpectral}}
We combine now the error terms coming from sections \ref{Sectionl-adic} and section \ref{SectionSpectral}. We remind that $N=q^\nu$ and $M=q^\mu$ with $\nu\geqslant \mu$. If we are in the case $\nu+\mu\leqslant 2-2\eta$, then we can apply the trivial bound \eqref{trivialbound}, obtaining
\begin{equation}\label{TrivialFinalBound}
\mathscr{T}_{OD}^{4,\pm}(\ell_1,\ell_2,N,M;q)\ll Lq^{\varepsilon-\eta}.
\end{equation}
Assume now that we are in the range $2-2\eta\leqslant \nu+\mu$. If $\nu-\mu\geqslant 1-2\theta-2\eta$, we apply Proposition \ref{Propositionladique}, getting the same as \eqref{TrivialFinalBound} (without the factor $L$). In the complementary case $\nu-\mu\leqslant 1-2\theta-2\eta$, we apply Theorem \ref{Theorem1} to the error term $\mathcal{OD}^{E}$ and we obtain 
$$\mathcal{OD}^{E,\pm}(\ell_1,\ell_2,N,M;q)\ll q^{\varepsilon}(\ell_1\ell_2)^{3/2} \frac{L^5}{q^{\eta}}+q^{\varepsilon}L^8q^{\frac{1}{4}(\nu-2)}.$$
Finally, applying \eqref{munuassumption} on the second term yields
\begin{equation}\label{FinalBoundErrorTerm}
\begin{split}
\mathcal{OD}^{E,\pm}(\ell_1,\ell_2,N,M;q)\ll q^\varepsilon \left( \frac{L^5(\ell_1\ell_2)^{3/2}}{q^\eta}+\frac{L^8}{q^{\frac{1}{4}(\frac{1}{2}+\theta+\eta)}}\right).
\end{split}
\end{equation}
Setting $L=q^\lambda$, the first term is automatically bigger than the second if $\lambda<\eta$, a condition we henceforth assume to hold and that gives the desired error term of Theorem \ref{FirstTheorem}.

%%%				THE OFF-DIAGONAL MAIN TERM						   %%%
%%%%%%%%%%%%%%%%%%%%%%%%%%%%%%%%%%%%%%%%%%%%%%%%%%%%%%%%%%%%%%%%%%%%%%%%%%%%%%%%%%%%%%%%%%%%%%%%%%%%%%%%%%%%%%%%%%%%%%%%%%%%%%%%%%%%%%%%%%%%%%%%%%%%%%%%%%%%%%%%%%%%%%%%%%%%%%%%%%%%%%%%%%%%%%%%%%%%%%%%%%%%%%%%%%%%%%%%%%%%%%%%%%%%%%%%%%%%%%%%%%%%%%%%%%%%%%%%%%%%%%%%%%%%%%%%%%%%%%%%%%%%%%%%

\section{The Off-Diagonal Main Term}\label{SectionOfDM}

We return to the main term that we left besides in the beginning of § \ref{SectionVoronoi}. This expression corresponds to the product of the two constants terms after the application of Voronoi summation formula to \eqref{eq2} and is given by (see also § 5 in \cite{duke1994})
\begin{equation}\label{DefinitionOfDiagMainTerm}
\begin{split}
\mathcal{OD}_{\pm}^{MT}(\ell_1,\ell_2,N,M;q) := \frac{2\phi^\ast(q)^{-1}}{(MN)^{1/2}}\sum_{d|q}\phi(d)\mu\left(\frac{q}{d}\right)\sum_{h\neq 0}\sum_{\ell\leqslant 2Q}\frac{(\ell_1\ell_2,\ell)}{\ell^2}S(hd,0;\ell) I_{\pm},
\end{split}
\end{equation}
and $I_{\pm}$ is the integral defined by 
\begin{equation}
\begin{split}
I_{\pm} := \int_0^\infty\int_0^\infty \left( \log \ell_1 x - \lambda_{\ell_1,\ell}\right)\left(\log \ell_2 y -\lambda_{\ell_2,\ell}\right) E^{\pm}(x,y,\ell)dxdy,
\end{split}
\end{equation}
with $E^\pm$ given by \eqref{DefinitionEpm} and 
$$\lambda_{\ell_i,\ell}:= \log\left( \frac{\ell_i\ell^2}{(\ell_i,\ell)^2}\right)-2\gamma.$$
%%%%%%%%%%%%%%%%%%%%%%%%%%%%%%%%%%%%%%%%%%%%%%%%%%%%%%%%%%%%%%%%%%%%%%%%
%%%%%%%%%%%%%%%%%%%%%%%%%%%%%%%%%%%%%%%%%%%%%%%%%%%%%%%%%%%%%%%%%%%%%%%%
%%%%%			SE DEBARASSER DE LA FONCTION DELTA				%%%%%%
%%%%%%%%%%%%%%%%%%%%%%%%%%%%%%%%%%%%%%%%%%%%%%%%%%%%%%%%%%%%%%%%%%%%%%%%
%%%%%%%%%%%%%%%%%%%%%%%%%%%%%%%%%%%%%%%%%%%%%%%%%%%%%%%%%%%%%%%%%%%%%%%%
As a first step, we need to remove the delta function $\Delta_\ell$ in our integral because this is an obstruction for the calculation. This can be done as follows : we make a first change of variables $\ell_1x\mapsto x$ and $\ell_2 y\mapsto y$ and then, $x=\mp y+hd+u$, getting
\begin{equation}\label{IntegralI}
I_{\pm}=\frac{1}{\ell_1\ell_2}\int_0^\infty \int_\mathbb{R} C(\mp y+hd+u,y)\Delta_\ell (u)dudy,
\end{equation}
where we defined 
$$C(x,y):= (\log x-\lambda_{\ell_1,\ell})(\log y -\lambda_{\ell_2,\ell})F\left(\frac{x}{\ell_1},\frac{y}{\ell_2}\right).$$
For the inner integral in \eqref{IntegralI}, we use equation $(18)$ in \cite{duke1994} and we obtain
$$\int_\mathbb{R} C(\mp y+hq+u,y)\Delta_\ell (u)du = C(\mp y+hd,y)+O\left(\left(\frac{\ell Q}{\ell_1 N}\right)^j\right),$$
where the implied constant depends on $j\geqslant 1$. Assuming $\ell< (\ell_1 N/\ell Q)^{1-\varepsilon}$, we make the error term above very small by choosing $j$ large enough. Therefore, we have for $\ell$ in this range
$$I_{\pm} = \frac{1}{\ell_1\ell_2}\int_0^\infty C(\mp y+hd,y)dy + O(q^{-100}).$$
On the other hand, we also have the bound (c.f. (30), \cite{duke1994})
$I_{\pm} \ll M\log Q $ which is valid for all $\ell$. Hence, using $|h|\leqslant LN/d$, the bound for the Ramanujan sum $S(hd,0;\ell)\ll (hd,\ell)$ and the definition of $Q=LN^{1/2+\varepsilon}$, we get
\begin{alignat}{1}
\mathcal{OD}_{\pm}^{MT}(\ell_1,\ell_2,N,M;q)= & \ \frac{2(\phi^\ast(q)\ell_1\ell_2)^{-1}}{(MN)^{1/2}}\sum_{d|q}\phi(d)\mu\left(\frac{q}{d}\right)\sum_{h\neq 0}\sum_{\ell =1}^\infty \frac{(\ell_1\ell_2,\ell)}{\ell^2}S(hd,0;\ell) \nonumber\\ & \times \ \int\limits_0^\infty C(\mp x+hd,x)dx + O(L^2q^{\varepsilon-1/2}), \label{DMT}
\end{alignat}
%%%%%%%%%%%%%%%%%%%%%%%%%%%%%%%%%%%%%%%%%%%%%%%%%%%%%%%%%%%%%%%%%%%%%%%%
%%%%%%%%%%%%%%%%%%%%%%%%%%%%%%%%%%%%%%%%%%%%%%%%%%%%%%%%%%%%%%%%%%%%%%%%
where the eror term takes care of the tail of the $\ell$-sum. We now recall that we have made the substitution $W(x)\leftrightarrow x^{-1/2}W(x)$ (c.f. \eqref{substitution}), so up to an error term of $O(L^2q^{\varepsilon-1/2})$, we have to compute the following expression
\begin{alignat*}{1}
\mathcal{OD}_{\pm}^{MT}(\ell_1,\ell_2,N,M;q)=\frac{2}{\phi^\ast(q)(\ell_1\ell_2)^{1/2}}\sum_{d|q}&\phi(d)\mu\left(\frac{q}{d}\right)\sum_{h\neq 0}\sum_{\ell =1}^\infty \frac{(\ell_1\ell_2,\ell)c_\ell(hd)}{\ell^2} \\ & \times \int_0^\infty \Lambda^{\pm}(x,hd,\ell_1,\ell_2,\ell;q)dx,
\end{alignat*}
where $c_\ell(hd)=S(hd,0;\ell)$ and where the function $\Lambda^{\pm}$ is defined by 
\begin{equation}\label{Lambda2}
\begin{split}
\Lambda^\pm (x,hd,\ell_1,\ell_2,\ell;q) := & \ \frac{(\log (\mp x+hd)-\lambda_{\ell_1,\ell})(\log x -\lambda_{\ell_2,\ell})}{(x(\mp x+hd))^{1/2}} \\  \times & \ V\left(\frac{x(\mp x+hd)}{\ell_1\ell_2q^2}\right)W\left( \frac{\mp x+hd}{\ell_1 N}\right)W\left(\frac{x}{\ell_2 M}\right).
\end{split}
\end{equation}
Before evaluating the $x$-integral, we use the well known formula for the Ramanujan sum 
$$c_\ell(hd)=\sum_{\substack{ab=\ell \\ b|hd}}\mu(a)b,$$
to get 
\begin{equation}\label{Expression}
\begin{split}
\mathcal{OD}_{\pm}^{MT}(\ell_1,\ell_2,N,M;q)=  \ \frac{2}{\phi^\ast(q)(\ell_1\ell_2)^{1/2}}&\sum_{d|q}\phi(d)\mu\left(\frac{q}{d}\right)\sum_{a\geqslant 1}\frac{\mu(a)}{a^2}\sum_{b\geqslant 1}\frac{(\ell_1\ell_2,ab)}{b} \\ & \times \sum_{\substack{h\neq 0 \\ b|hd}}\int_0^\infty \Lambda^\pm(x,hd,\ell_1,\ell_2,ab;q)dx.
\end{split}
\end{equation}
It is convenient to replace the condition $b|hd$ by $b|h$. If $d=1$, there is nothing to do. Now if $d=q$, we use the fact that the integral is supported on $x\asymp \ell_2 M$, the $h$-summation to $|h|\leqslant LN/q$ and $b\leqslant LN$ to obtain that up to an error term of $O(L^2q^{-1+\varepsilon})$, we can assume that $(b,q)=1$. Once we have done this, we can also remove the condition $(b,q)=1$ for the same cost.

%%%%%%%%%%%%%%%%%%%%%%%%%%%%%%%%%%%%%%%%%%%%%%%%%%%%%%%%%%%%%%%%%%%%%%%%%%%%%%%%%%%%%%%%%%%%%%%%%%%%%%%%%%%%%%%%%%%%%%%%%%%%%%%%%%%%%%%%%%%%%%%%%%%%%%%%%%%%%%%%%%%%%%%%%%%%%%%%%%%%%%%%%%%%%%%%%%%%%%%%%%%%%%%%%%%%%%%%
%%%%%%%%%%%%%%%%%%%%%%%%%%%%%%%%%%%%%%%%%%%%%%%%%%%%%%%%%%%%%%%%%%%%%%%%%%%%%%%%%%%%%%%%%%%%%%%%%%%%%%%%%%%%%%%%%%%%%%%%%%%%%%%%%%%%%%%%%%%%%%%%%%%%%%%%%%%%%%%%%%%%%%%%%%%%%%%%%%%%%%%%%%%%%%%%%%%%%%%%%%%%%%%%%%%%%%%%
%%%%%%%%%%%%%%%%%%%%%%%%%%%%%%%%%%%%%%%%%%%%%%%%%%%%%%%%%%%%%%%%%%%%%%%%%%%%%%%%%%%%%%%%%%%%%%%%%%%%%%%%%%%%%%%%%%%%%%%%%%%%%%%%%%%%%%%%%%%%%%%%%%%%%%%%%%%%%%%%%%%%%%%%%%%%%%%%%%%%%%%%%%%%%%%%%%%%%%%%%%%%%%%%%%%%%%%%
%%%%%				EVALUATION OF THE X-INTEGRAL					 %%%%%
%%%%%%%%%%%%%%%%%%%%%%%%%%%%%%%%%%%%%%%%%%%%%%%%%%%%%%%%%%%%%%%%%%%%%%%%%%%%%%%%%%%%%%%%%%%%%%%%%%%%%%%%%%%%%%%%%%%%%%%%%%%%%%%%%%%%%%%%%%%%%%%%%%%%%%%%%%%%%%%%%%%%%%%%%%%%%%%%%%%%%%%%%%%%%%%%%%%%%%%%%%%%%%%%%%%%%%%%
%%%%%%%%%%%%%%%%%%%%%%%%%%%%%%%%%%%%%%%%%%%%%%%%%%%%%%%%%%%%%%%%%%%%%%%%%%%%%%%%%%%%%%%%%%%%%%%%%%%%%%%%%%%%%%%%%%%%%%%%%%%%%%%%%%%%%%%%%%%%%%%%%%%%%%%%%%%%%%%%%%%%%%%%%%%%%%%%%%%%%%%%%%%%%%%%%%%%%%%%%%%%%%%%%%%%%%%%
%%%%%%%%%%%%%%%%%%%%%%%%%%%%%%%%%%%%%%%%%%%%%%%%%%%%%%%%%%%%%%%%%%%%%%%%%%%%%%%%%%%%%%%%%%%%%%%%%%%%%%%%%%%%%%%%%%%%%%%%%%%%%%%%%%%%%%%%%%%%%%%%%%%%%%%%%%%%%%%%%%%%%%%%%%%%%%%%%%%%%%%%%%%%%%%%%%%%%%%%%%%%%%%%%%%%%%%%
\subsection{Evaluation of the $x$-Integral}
For this evaluation, we need to separate the $\pm$ case. Using the integral representation for $V$ from \eqref{DefinitionV} and the Mellin inversion formula for $W$, we have for the minus case
\begin{alignat}{1}
&\int\limits_0^\infty \Lambda^-(x,hd,\ell_1,\ell_2,ab;q)dx = \frac{1}{(2\pi i)^3}\int\limits_{(\ast)}\int\limits_{(\ast)}\widetilde{W_{N,M}}(w_1,w_2)\nonumber \\ & \times \int\limits_{(\ast)}\frac{G(s)q^{2s}}{\ell_1^{-s-w_1}\ell_2^{-s-w_2}}   \int\limits_0^\infty x^{1/2-s-w_2}\frac{\left(\log x-\lambda_{\ell_2,ab}\right)\left(\log(x+hd)-\lambda_{\ell_1,ab}\right)}{\left(x+hd\right)^{1/2+s+w_1}}\nonumber \\ &  \times (\delta_{h>0}+\delta_{h<0}\delta_{x>-hd})\frac{dx}{x}\frac{ds}{s}dw_2dw_1,\label{Evaluationxintegral-}
\end{alignat}
while for the plus case, we clearly have zero if $h<0$ and otherwise
\begin{equation}\label{Evaluationxintegral+} 
\begin{split}
&\int\limits_0^\infty \Lambda^+(x,hd,\ell_1,\ell_2,ab;q)dx = \frac{1}{(2\pi i)^3}\int\limits_{(\ast)}\int\limits_{(\ast)}\widetilde{W_{N,M}}(w_1,w_2)\int\limits_{(\ast)}\frac{G(s)q^{2s}}{\ell_1^{-s-w_1}\ell_2^{-s-w_2}} \\ 
& \times \int\limits_0^{\infty} x^{1/2-s-w_2}\frac{\left(\log x-\lambda_{\ell_2,ab}\right)\left(\log(hd-x)-\lambda_{\ell_1,ab}\right)}{\left(hd-x\right)^{1/2+s+w_1}}\delta_{x<hd}\frac{dx}{x}\frac{ds}{s}dw_2dw_1.
\end{split}
\end{equation}
Here $(\ast)$ means that we choose contours such that the $x$-integral is convergent. More precisely, for \eqref{Evaluationxintegral-}, we need to impose 
\begin{equation}\label{Contour23}
\begin{array}{ccc}
\Re e (s+w_2)<1/2 \ , \ \Re e (2s+w_1+w_2)>0 & \mathrm{if} & h>0 \\ 
\Re e (s+w_1)<1/2 \ , \ \Re e (2s+w_1+w_2)>0 & \mathrm{if} & h<0,
\end{array}
\end{equation}
and for \eqref{Evaluationxintegral+}, we must have
\begin{equation}\label{Contour23+}
\Re e (s+w_1) < 1/2 \ \ \mathrm{and} \ \ \Re e (s+w_2)<1/2.
\end{equation}
In order to perform this computation and to deal later with real Dirichlet series, we put the logarithm factors in a more appropriate form, namely 
$$(\log x-\lambda_{\ell_2,ab})(\log(\pm x+hd)-\lambda_{\ell_1,ab})=\mathfrak{D}_{\gamma}\cdot \left(\frac{x(\ell_2,ab)^2}{\ell_2(ab)^2}\right)^{u_2}\left(\frac{(\pm x+hd)(\ell_1,ab)^2}{\ell_1(ab)^2}\right)^{u_1},$$
where 
$$\mathfrak{D}_\gamma := (\partial_{u_1}+2\gamma)(\partial_{u_2}+2\gamma)|_{u_1=u_2=0}.$$
Assuming that \eqref{Contour23} and \eqref{Contour23+} hold, we can rewrite the two last lines of \eqref{Evaluationxintegral-} in the form
\begin{equation}\label{Evaluationxintegral2}
\begin{split}
\mathfrak{D}_\gamma\cdot\left\{ \frac{(\ell_1,ab)^{2u_1}(\ell_2,ab)^{2u_2}}{\ell_1^{u_1}\ell_2^{u_2}a^{2u_1+2u_2}b^{2u_1+2u_2}} \int\limits_0^\infty x^{1/2-s-w_2+u_2}\frac{\delta_{h<0}\delta_{x>-hd}+\delta_{h>0}}{(x+hd)^{1/2+s+w_1-u_1}}\frac{dx}{x} \right\},
\end{split}
\end{equation}
and the last line of \eqref{Evaluationxintegral+} equals
\begin{equation}\label{Evaluationxintegral2+}
\begin{split}
\mathfrak{D}_\gamma \cdot\left\{ \frac{(\ell_1,ab)^{2u_1}(\ell_2,ab)^{2u_2}}{\ell_1^{u_1}\ell_2^{u_2}a^{2u_1+2u_2}b^{2u_1+2u_2}}\int\limits_0^\infty x^{1/2-s-w_2+u_2}\frac{\delta_{x<hd}}{(hd-x)^{1/2+s+w_1-u_1}}\frac{dx}{x} \right\}.
\end{split}
\end{equation}
Now using $(2.21)$ and $(2.19)$ in \cite{Mellin}, we obtain that the Mellin transform \eqref{Evaluationxintegral2} equals
\begin{equation}\label{FinalMellinTransform1}
\begin{split}
\mathfrak{D}_\gamma \cdot & \left\{ \frac{(\ell_1,ab)^{2u_1}(\ell_2,ab)^{2u_2}}{\ell_1^{u_1}\ell_2^{u_2}a^{2u_1+2u_2}b^{2u_1+2u_2}(|h|d)^{2s+w_1+w_2-u_1-u_2}}\right. \\ \times & \left\{ \begin{array}{ccc} \frac{\Gamma(2s+w_1+w_2-u_1-u_2)\Gamma(\frac{1}{2}-s-w_2+u_2)}{\Gamma(\frac{1}{2}+s+w_1-u_1)} & \mathrm{if} & h>0, \\ & &  \\ \frac{\Gamma(2s+w_1+w_2-u_1-u_2)\Gamma(\frac{1}{2}-s-w_1+u_1)}{\Gamma(1/2+s+w_2-u_2)} & \mathrm{if} & h<0.
\end{array}\right.
\end{split}
\end{equation}
Similarly, we use $(2.20)$ in \cite{Mellin} for \eqref{Evaluationxintegral2+} and we obtain 
\begin{equation}\label{FinalMellinTransform2}
\begin{split}
\mathfrak{D}_\gamma\cdot & \left\{ \frac{(\ell_1,ab)^{2u_1}(\ell_2,ab)^{2u_2}}{\ell_1^{u_1}\ell_2^{u_2}a^{2u_1+2u_2}b^{2u_1+2u_2}(|h|d)^{2s+w_1+w_2-u_1-u_2}}\right.\\ \times & \left\{ \begin{array}{ccc} \frac{\Gamma(\frac{1}{2}-s-w_1+u_1)\Gamma(\frac{1}{2}-s-w_2+u_2)}{\Gamma(1-2s-w_1-w_2+u_1+u_2)} & \mathrm{if} & h>0, \\ & &  \\ 0 & \mathrm{if} & h<0.
\end{array}\right.
\end{split}
\end{equation}
According to \eqref{Contour23} and \eqref{Contour23+}, we choose finally the following contours
\begin{equation}\label{Contour23(1)}
\Re e (s), \Re e (w_1), \Re e (w_2) = \left\{ \begin{array}{ccc} \varepsilon \ , 0 \ , \varepsilon & \mathrm{if}& h>0 \\  & &  \\ \varepsilon \ , \varepsilon, 0 &\mathrm{if}& h<0,
\end{array}\right.
\end{equation}
assuming of course that the $u_i'$ s variables are sufficiently small compared to $\varepsilon$.

%%%%%%%%%%%%%%%%%%%%%%%%%%%%%%%%%%%%%%%%%%%%%%%%%%%%%%%%%%%%%%%%%%%%%%%%%%%%%%%%%%%%%%%%%%%%%%%%%%%%%%%%%%%%%%%%%%%%%%%%%%%%%%%%%%%%%%%%%%%%%%%%%%%%%%%%%%%%%%%%%%%%%%%%%%%%%%%%%%%%%%%%%%%%%%%%%%%%%%%%%%%%%%%%%%%%%%%%
%%%%%%%%%%%%%%%%%%%%%%%%%%%%%%%%%%%%%%%%%%%%%%%%%%%%%%%%%%%%%%%%%%%%%%%%%%%%%%%%%%%%%%%%%%%%%%%%%%%%%%%%%%%%%%%%%%%%%%%%%%%%%%%%%%%%%%%%%%%%%%%%%%%%%%%%%%%%%%%%%%%%%%%%%%%%%%%%%%%%%%%%%%%%%%%%%%%%%%%%%%%%%%%%%%%%%%%%
%%%%%%%%%%%%%%%%%%%%%%%%%%%%%%%%%%%%%%%%%%%%%%%%%%%%%%%%%%%%%%%%%%%%%%%%%%%%%%%%%%%%%%%%%%%%%%%%%%%%%%%%%%%%%%%%%%%%%%%%%%%%%%%%%%%%%%%%%%%%%%%%%%%%%%%%%%%%%%%%%%%%%%%%%%%%%%%%%%%%%%%%%%%%%%%%%%%%%%%%%%%%%%%%%%%%%%%%
%%%%%				ASSEMBLING THE PARTITION OF UNITY			 %%%%%
%%%%%%%%%%%%%%%%%%%%%%%%%%%%%%%%%%%%%%%%%%%%%%%%%%%%%%%%%%%%%%%%%%%%%%%%%%%%%%%%%%%%%%%%%%%%%%%%%%%%%%%%%%%%%%%%%%%%%%%%%%%%%%%%%%%%%%%%%%%%%%%%%%%%%%%%%%%%%%%%%%%%%%%%%%%%%%%%%%%%%%%%%%%%%%%%%%%%%%%%%%%%%%%%%%%%%%%%
%%%%%%%%%%%%%%%%%%%%%%%%%%%%%%%%%%%%%%%%%%%%%%%%%%%%%%%%%%%%%%%%%%%%%%%%%%%%%%%%%%%%%%%%%%%%%%%%%%%%%%%%%%%%%%%%%%%%%%%%%%%%%%%%%%%%%%%%%%%%%%%%%%%%%%%%%%%%%%%%%%%%%%%%%%%%%%%%%%%%%%%%%%%%%%%%%%%%%%%%%%%%%%%%%%%%%%%%
%%%%%%%%%%%%%%%%%%%%%%%%%%%%%%%%%%%%%%%%%%%%%%%%%%%%%%%%%%%%%%%%%%%%%%%%%%%%%%%%%%%%%%%%%%%%%%%%%%%%%%%%%%%%%%%%%%%%%%%%%%%%%%%%%%%%%%%%%%%%%%%%%%%%%%%%%%%%%%%%%%%%%%%%%%%%%%%%%%%%%%%%%%%%%%%%%%%%%%%%%%%%%%%%%%%%%%%%
\subsection{Assembling the Partition of Unity}
The partition of unity is an obstruction for the computation of the second main term, so we need to rebuild it. This step requires some preparations. We return to Expression \eqref{Expression} (recall that we have removed $b|hd\leftrightarrow b|h$) and separate the case $h<0$ from $h>0$ by writing $\mathcal{OD}_{\pm}^{MT}(\ell_1,\ell_2,N,M;q)=\mathcal{OD}_{\pm,h>0}^{MT}(\ell_1,\ell_2,N,M;q)+\mathcal{OD}_{\pm,h<0}^{MT}(\ell_1,\ell_2,N,M;q)$, recalling that $\mathcal{OD}_{+,h<0}^{MT}=0$. In order to have a symmetric situation, we may group the terms as follow :
\begin{equation}\label{Definition1C(l1,l2,N,M)}
\begin{split}
\mathcal{OD}^{MT}:=\sum_{\pm}\mathcal{OD}_\pm^{MT} & \ = \left(\mathcal{OD}_{-,h>0}^{MT}+\frac{1}{2}\mathcal{OD}_{+,h>0}^{MT}\right)+\left(\mathcal{OD}_{-,h<0}^{MT}+\frac{1}{2}\mathcal{OD}_{+,h>0}^{MT}\right) \\ 
& \ =: \mathcal{C}_1(\ell_1,\ell_2,N,M;q)+\mathcal{C}_2(\ell_1,\ell_2,N,M;q).
\end{split}
\end{equation}
In $\mathcal{OD}_{+,h>0}^{MT}$ appearing in the second term, we just change the $w_1,w_2$-contours to have the same as $\mathcal{OD}_{-,h<0}^{MT}$ (see \eqref{Contour23(1)}). Inserting the results \eqref{FinalMellinTransform1} and \eqref{FinalMellinTransform2}, we obtain
\begin{equation}\label{Definition2C(l1,l2,N,M)}
\begin{split}
\mathcal{C}_1(\ell_1,\ell_2,N,M;q)& \ = \frac{2}{\phi^\ast(q)}\sum_{d|q}\phi(d)\mu\left(\frac{q}{d}\right)\sum_{a\geqslant 1}\frac{\mu(a)}{a^2}\sum_{b\geqslant 1}\frac{(\ell_1\ell_2,ab)}{b}\sum_{\substack{h\geqslant 1 \\ b|h}} \\ \times & \ \frac{1}{(2\pi i)^3}\int\limits_{(0)}\int\limits_{(\varepsilon)}\widetilde{W_{N,M}}(w_1,w_2)\int\limits_{(\varepsilon)}\frac{G(s)q^{2s}}{\ell_1^{1/2-s-w_1}\ell_2^{1/2-s-w_2}} \\ 
\times & \ \mathfrak{D}_\gamma \cdot \left\{\frac{(\ell_1,ab)^{2u_1}(\ell_2,ab)^{2u_2}H_1(s,w_1,w_2,u_1,u_2)}{\ell_1^{u_1}\ell_2^{u_2}(ab)^{2u_1+2u_2}(hd)^{2s+w_1+w_2-u_1-u_2}}\right\}\frac{ds}{s}dw_2dw_1.
\end{split}
\end{equation}
where 
\begin{equation}\label{Definition H1}
\begin{split}
H_1(s,w_1,w_2,u_1,u_2):= & \ \frac{\Gamma(2s+w_1+w_2-u_1-u_2)\Gamma(\frac{1}{2}-s-w_2+u_2)}{\Gamma(\frac{1}{2}+s+w_1-u_1)} \\ & + \ \frac{1}{2}\frac{\Gamma(\frac{1}{2}-s-w_1+u_1)\Gamma(\frac{1}{2}-s-w_2+u_2)}{\Gamma(1-2s-w_1-w_2+u_1+u_2)}.
\end{split}
\end{equation}
The definition of $\mathcal{C}_2(\ell_1,\ell_2,N,M)$ is the same, but with $\Re e (w_1)=\varepsilon$, $\Re e (w_2)=0$ and $H_2$ instead of $H_1$ with 
\begin{equation}\label{Definition H2}
\begin{split}
H_2(s,w_1,w_2,u_1,u_2):= & \ \frac{\Gamma(2s+w_1+w_2-u_1-u_2)\Gamma(\frac{1}{2}-s-w_1+u_1)}{\Gamma(\frac{1}{2}+s+w_2-u_2)} \\ & + \ \frac{1}{2}\frac{\Gamma(\frac{1}{2}-s-w_1+u_1)\Gamma(\frac{1}{2}-s-w_2+u_2)}{\Gamma(1-2s-w_1-w_2+u_1+u_2)}.
\end{split}
\end{equation}

%%%%%%%%%%%%%%%%%%%%%%%%%%%%%%%%%%%%%%%%%%%%%%%%%%%%%%%%%%%%%%%%%%%%%%%%%%%%%%%%%%%%%%%%%%%%%%%%%%%%%%%%%%%%%%%%%%%%%%%%%%%%%%%%%%%%%%%%%%%%%%%%%%%%%%%%%%%%%%%%%%%%%%%%%%%%%%%%%%%%%%%%%%%%%%%%%%%%%%%%%%%%%%%%%%%%%%%%
%%%%%%%%%%%%%%%%%%%%%%%%%%%%%%%%%%%%%%%%%%%%%%%%%%%%%%%%%%%%%%%%%%%%%%%%%%%%%%%%%%%%%%%%%%%%%%%%%%%%%%%%%%%%%%%%%%%%%%%%%%%%%%%%%%%%%%%%%%%%%%%%
%%%%%				SHIFTING THE S-CONTOUR						 %%%%%
%%%%%%%%%%%%%%%%%%%%%%%%%%%%%%%%%%%%%%%%%%%%%%%%%%%%%%%%%%%%%%%%%%%%%%%%
%%%%%%%%%%%%%%%%%%%%%%%%%%%%%%%%%%%%%%%%%%%%%%%%%%%%%%%%%%%%%%%%%%%%%%%%

\subsubsection{Shifting the $s$-contour}\label{Paragraphs-contour}
The goal here is to move the $s$-line on the right to make the $h$-summation absolutely convergent and bring up the zeta function. We will see that we catch some poles whose contributions seem to be big. Fortunately, the arithmetical sum over $d|q$ cancels these extra factors and this is the reason why we did not separate it at the beginning of Section \ref{SectionFourthMoment}. We treat here only $\mathcal{C}_1(\ell_1,\ell_2,N,M;q)$ since the other is completely similar by changing $w_1\leftrightarrow w_2$ in our arguments. Since we deal with the $s,w_1,w_2$-integrals, we can put the differential operator $\mathfrak{D}_\gamma$ outside and only focus on the following quantity :
\begin{alignat}{1}
\mathcal{I} :=  \frac{1}{\phi^\ast(q)}&\sum_{d|q}\phi(d)\mu\left(\frac{q}{d}\right)\sum_{b|h}\frac{1}{(2\pi i)^3}\int\limits_{(0)}\int\limits_{(\varepsilon)}\widetilde{W}_{N,M}(w_1,w_2)\label{I} \\ & \times \int\limits_{(\varepsilon)}\frac{G(s)q^{2s}H_1(s,w_1,w_2,u_1,u_2)}{\ell_1^{1/2-s-w_1+u_1}\ell_2^{1/2-s-w_2+u_2}(hd)^{2s+w_1+w_2-u_1-u_2}}\frac{ds}{s}dw_2dw_1.\nonumber
\end{alignat}
We now move the $s$-line of integration to $\Re e (s)=1/2-\varepsilon/3$, passing a simple pole at $s=1/2-w_2+u_2$ coming from the factor $\Gamma(1/2-s-w_2+u_2)$ in the function $H_1(s,w_1,w_2,u_1,u_2)$. Note that since we moved to the right, the residue has to be taken with the minus sign. Hence we obtain that $\mathcal{I}=-\mathcal{R}+\mathcal{I}'$ where $\mathcal{I}'$ is the same as \eqref{I} but with $\Re e (s)=1/2-\varepsilon/3$ and the residue part is given by 
\begin{alignat}{1}
\mathcal{R}=-\frac{3}{2\phi^\ast(q)}\sum_{d|q}\phi(d)\mu\left(\frac{q}{d}\right)&\sum_{b|h}\frac{1}{(2\pi i)^2}\int\limits_{(0)}\int\limits_{(\varepsilon)}\widetilde{W_{N,M}}(w_1,w_2)\label{Residue_I} \\ & \times \frac{G(1/2-w_2+u_2)q^{1-2w_2+2u_2}}{\ell_1^{w_2-w_1+u_1-u_2}(hd)^{1+w_1-w_2+u_2-u_1}}dw_2dw_1. \nonumber
\end{alignat}
In $\mathcal{R}$, we shift the $w_1$-contour to $\Re e (w_1)=2\varepsilon$, passing no pole. Since $\Re e (1+w_1-w_2+u_2-u_1)>1$, we can switch the $h$-summation with these two integrals, obtaining
\begin{alignat}{1}
\mathcal{R}= & -\frac{3}{2\phi^\ast(q)}\sum_{d|q}\phi(d)\mu\left(\frac{q}{d}\right)\frac{1}{(2\pi i)^2}\int\limits_{(2\varepsilon)}\int\limits_{(\varepsilon)}\widetilde{W_{N,M}}(w_1,w_2)\label{Residue_I2} \\ & \times\zeta(1+w_1-w_2+u_2-u_1)\frac{G(1/2-w_2+u_2)q^{1-2w_2+2u_2}}{\ell_1^{w_2-w_1+u_1-u_2}(bd)^{1+w_1-w_2+u_2-u_1}}dw_2dw_1. \nonumber
\end{alignat}
Now we deal with $\mathcal{I}'$. Since $\Re e (2s+w_1+w_2-u_1-u_2)>1$, we can also switch the $h$-summation with the three integrals. Once we have done this, we move the $s$-line to $\Re e (s)=\varepsilon$, passing two poles : one at $2s+w_1+w_2-u_1-u_2=1$ coming from the new factor $\zeta(2s+w_1+w_2-u_1-u_2)$ and the other again at $s=1/2-w_2+u_2$. Hence we have (this time the residues have to be taken with positive signs) $\mathcal{I}'=\mathcal{I}''+\mathcal{R}' + \mathscr{R}$ where $\mathcal{R}'$ is the same as \eqref{Residue_I2}, but with $\Re e (w_1)=0$ instead of $2\varepsilon$ and $\mathscr{R}$ is the residue at $2s+w_1+w_2-u_1-u_2=1$. In summary, we obtained the following decomposition of \eqref{I} :
\begin{equation}\label{DecompositionmathcalI}
\mathcal{I}=\mathcal{I}'' + \mathcal{R}'-\mathcal{R}+\mathscr{R}.
\end{equation}
%%%%%%%%%%%%%%%%%%%%%%%%%%%%%%%%%%%%%%%%%%%%%%%%%%%%%%%%%%%%%%%%%%%%%%%%
%%%%%%%%%%%%%%%%%%%%%%%%%%%%%%%%%%%%%%%%%%%%%%%%%%%%%%%%%%%%%%%%%%%%%%%%
%%%%%			LEMMA FOR THE SIZE OF THE RESIDUES				 %%%%%
%%%%%%%%%%%%%%%%%%%%%%%%%%%%%%%%%%%%%%%%%%%%%%%%%%%%%%%%%%%%%%%%%%%%%%%%
%%%%%%%%%%%%%%%%%%%%%%%%%%%%%%%%%%%%%%%%%%%%%%%%%%%%%%%%%%%%%%%%%%%%%%%%
\begin{lemme}\label{LemmaRR} With the above notations, we have 
$$|\mathcal{R}'-\mathcal{R}|+|\mathscr{R}|=O((qb)^{-1+\varepsilon}).$$
\end{lemme}
\begin{proof} We begin with $\mathcal{R}'-\mathcal{R}.$ Since the only difference between these two expressions is the $w_1$-contour, we want to shift it in $\mathcal{R}$ to $\Re e (w_1)=0$. Before doing this, we switch the arithmetic sum over $d$ with the $w_i$-integrals, obtaining
\begin{alignat*}{1}
\mathcal{R} = & -\frac{3}{2\phi^\ast(q)}\frac{1}{(2\pi i)^2}\int\limits_{(2\varepsilon)}\int\limits_{(\varepsilon)}\widetilde{W_{N,M}}(w_1,w_2) \frac{G(1/2-w_2+u_2)q^{1-2w_2+2u_2}}{\ell_1^{w_2-w_1+u_1-u_2}b^{1+w_1-w_2+u_2-u_1}} \\ & \times \zeta(1+w_1-w_2+u_2-u_1)\left(\frac{\phi(q)}{q^{1+w_1-w_2+u_2-u_1}}-1\right)dw_2dw_1.
\end{alignat*}
From the obvious observation that 
$$\frac{\phi(q)}{q^{1+w_1-w_2+u_2-u_1}}-1=\left(q^{w_2-w_1+u_1-u_2}-1\right)-\frac{1}{q^{1+w_1-w_2+u_2-u_1}},$$
we can separate $\mathcal{R}$ as a sum of two terms $\mathcal{R}=\mathcal{R}_1+\mathcal{R}_2$, according to the above decomposition. In the second expression, we can average trivially over the $w_i$-integrals, obtaining the bound $O((qb)^{-1+\varepsilon})$. In $\mathcal{R}_1$, since the pole of the zeta function at $w_1-u_1=w_2-u_2$ is cancelled by the factor $q^{w_2-w_1+u_1-u_2}-1$, we can shift the $w_1$-line to $\Re e (w_1)=0$. Writing the same decomposition for $\mathcal{R}'$, namely $\mathcal{R}'=\mathcal{R}_1'+\mathcal{R}_2'$, we obtain that $\mathcal{R}_1'=\mathcal{R}_1$ and $\mathcal{R}_2'=O((qb)^{-1+\varepsilon})$.
We play the same game for $\mathscr{R}$. We have, after summing over $d|q$,
\begin{alignat*}{1}
\mathscr{R}=& \ \frac{1}{2\phi^\ast(q)(2\pi i)^2}\int\limits_{(0)}\int\limits_{(\varepsilon)}\widetilde{W_{N,M}}(w_1,w_2)\frac{G(\frac{1-w_1-w_2+u_1+u_2}{2})q^{1-w_1-w_2+u_1+u_2}}{\ell_1^{\frac{w_2-w_1+u_1-u_2}{2}}\ell_2^{\frac{w_1-w_2+u_2-u_1}{2}}b} \\ & \times H_1\left(\frac{1-w_1-w_2+u_1+u_2}{2},w_1,w_2,u_1,u_2\right)\left(\frac{\phi(q)}{q}-1\right)dw_2dw_1,
\end{alignat*} 
which is bounded by $O((qb)^{-1+\varepsilon})$.
\end{proof}
%%%%%%%%%%%%%%%%%%%%%%%%%%%%%%%%%%%%%%%%%%%%%%%%%%%%%%%%%%%%%%%%%%%%%%%%
%%%%%%%%%%%%%%%%%%%%%%%%%%%%%%%%%%%%%%%%%%%%%%%%%%%%%%%%%%%%%%%%%%%%%%%%
%%%%%					END OF THE LEMMA							 %%%%%
%%%%%%%%%%%%%%%%%%%%%%%%%%%%%%%%%%%%%%%%%%%%%%%%%%%%%%%%%%%%%%%%%%%%%%%%
%%%%%%%%%%%%%%%%%%%%%%%%%%%%%%%%%%%%%%%%%%%%%%%%%%%%%%%%%%%%%%%%%%%%%%%%
We substitute  the decomposition \eqref{DecompositionmathcalI} of $\mathcal{I}$ together with Lemma \ref{LemmaRR} in the expression \eqref{Definition2C(l1,l2,N,M)} of $\mathcal{C}_1(\ell_1,\ell_2,N,M;q)$. After doing this, we only retain in the $d$-summation the case where $d=q$ ; the other contributes $O(q^{-1+\varepsilon})$. We collect the previous computations in the following proposition.
%%%%%%%%%%%%%%%%%%%%%%%%%%%%%%%%%%%%%%%%%%%%%%%%%%%%%%%%%%%%%%%%%%%%%%%%
%%%%%%%%%%%%%%%%%%%%%%%%%%%%%%%%%%%%%%%%%%%%%%%%%%%%%%%%%%%%%%%%%%%%%%%%
%%%%%			PROPOSITION FOR C(L1,L2,N,M;q)				     %%%%%
%%%%%%%%%%%%%%%%%%%%%%%%%%%%%%%%%%%%%%%%%%%%%%%%%%%%%%%%%%%%%%%%%%%%%%%%
%%%%%%%%%%%%%%%%%%%%%%%%%%%%%%%%%%%%%%%%%%%%%%%%%%%%%%%%%%%%%%%%%%%%%%%%
\begin{proposition}\label{PropositionC(N,M)} The quantity defined by \eqref{Definition2C(l1,l2,N,M)} is equal, up to $O(q^{-1+\varepsilon})$, to
\begin{alignat*}{1}
\mathcal{C}_1(\ell_1,\ell_2,N,M;q)& \ =2 \sum_{a,b \geqslant 1}\frac{\mu(a)(\ell_1\ell_2,ab)}{(\ell_1\ell_2)^{1/2}a^2 b}\frac{1}{(2\pi i)^3}\int\limits_{(0)}\int\limits_{(0)}\frac{\widetilde{W_{N,M}}(w_1,w_2)}{\ell_1^{-w_1}\ell_2^{-w_2}} \\
\times & \ \int\limits_{(\varepsilon)}\frac{G(s)}{(\ell_1\ell_2)^{-s}} 
\mathfrak{D}_\gamma \cdot \left\{\frac{(\ell_1,ab)^{2u_1}(\ell_2,ab)^{2u_2}}{a^{2u_1+2u_2}} \right. \\  \times & \ \left.\frac{\zeta(2s+w_1+w_2-u_1-u_2)H_1(s,w_1,w_2,u_1,u_2)}{\ell_1^{u_1}\ell_2^{u_2}b^{2s+w_1+w_2+u_1+u_2}q^{w_1+w_2-u_1-u_2}}\right\}\frac{ds}{s}dw_2dw_1.
\end{alignat*}
\end{proposition}
\begin{remq}
The previous Proposition is also valid for $\mathcal{C}_2(\ell_1,\ell_2,N,M;q)$ but with $H_2$ replaced of $H_1$.
\end{remq}

%%%%%%%%%%%%%%%%%%%%%%%%%%%%%%%%%%%%%%%%%%%%%%%%%%%%%%%%%%%%%%%%%%%%%%%%%%%%%%%%%%%%%%%%%%%%%%%%%%%%%%%%%%%%%%%%%%%%%%%%%%%%%%%%%%%%%%%%%%%%%%%%%%%%%%%%%%%%%%%%%%%%%%%%%%%%%%%%%%%%%%%%%%%%%%%%%%%%%%%%%%%%%%%%%%%%%%%%
%%%%%%%%%%%%%%%%%%%%%%%%%%%%%%%%%%%%%%%%%%%%%%%%%%%%%%%%%%%%%%%%%%%%%%%%%%%%%%%%%%%%%%%%%%%%%%%%%%%%%%%%%%%%%%%%%%%%%%%%%%%%%%%%%%%%%%%%%%%%%%%%%%%%%%%%%%%%%%%%%%%%%%%%%%%%%%%%%%%%%%%%%%%%%%%%%%%%%%%%%%%%%%%%%%%%%%%%
%%%%%%%%%%%%%%%%%%%%%%%%%%%%%%%%%%%%%%%%%%%%%%%%%%%%%%%%%%%%%%%%%%%%%%%%%%%%%%%%%%%%%%%%%%%%%%%%%%%%%%%%%%%%%%%%%%%%%%%%%%%%%%%%%%%%%%%%%%%%%%%%%%%%%%%%%%%%%%%%%%%%%%%%%%%%%%%%%%%%%%%%%%%%%%%%%%%%%%%%%%%%%%%%%%%%%%%%
%%%%%				ADDING THE MISSING PAIRS						 %%%%%
%%%%%%%%%%%%%%%%%%%%%%%%%%%%%%%%%%%%%%%%%%%%%%%%%%%%%%%%%%%%%%%%%%%%%%%%
%%%%%%%%%%%%%%%%%%%%%%%%%%%%%%%%%%%%%%%%%%%%%%%%%%%%%%%%%%%%%%%%%%%%%%%%
%%%%%%%%%%%%%%%%%%%%%%%%%%%%%%%%%%%%%%%%%%%%%%%%%%%%%%%%%%%%%%%%%%%%%%%%

\subsubsection{Adding the missing pairs $(N,M)$}
We recall that at this step, the variables $N$ and $M$ belong to the set 
$$\mathscr{O}=\left\{ (N,M) \ | \ 1\leqslant M\leqslant N, \ \frac{N}{M}\leqslant q^{1-2\theta-2\eta} , \ NM\leqslant q^{2+\varepsilon}\right\}.$$
If we could add all the other pairs $(N,M)$ to complete the partition of unity, then we could use the following lemma (see § 6 \cite{young})
%%%%%%%%%%%%%%%%%%%%%%%%%%%%%%%%%%%%%%%%%%%%%%%%%%%%%%%%%%%%%%%%%%%%%%%%
%%%%%%%%%%%%%%%%%%%%%%%%%%%%%%%%%%%%%%%%%%%%%%%%%%%%%%%%%%%%%%%%%%%%%%%%
%%%%%				LEMMA PARTITION OF UNITY					     %%%%%
%%%%%%%%%%%%%%%%%%%%%%%%%%%%%%%%%%%%%%%%%%%%%%%%%%%%%%%%%%%%%%%%%%%%%%%%
%%%%%%%%%%%%%%%%%%%%%%%%%%%%%%%%%%%%%%%%%%%%%%%%%%%%%%%%%%%%%%%%%%%%%%%%
\begin{lemme}\label{LemmaPartition}Let $F(s_1,s_2)$ be a holomorphic function in the strip $a<\Re e (s_i)<b$ with $a<0<b$ that decays rapidly to zero in each variable (in the imaginary direction). Then we have 
$$\sum_{N,M}\frac{1}{(2\pi i)^2}\int_{(\ast)}\int_{(\ast)}\widetilde{W_{N,M}}(s_1,s_2)F(s_1,s_2)ds_2ds_1 = F(0,0).$$
\end{lemme}
\begin{proof}
Let $f(x,y)$ be the inverse Mellin transform of $F(s_1,s_2)$ ; then the left handside equals
\begin{alignat*}{1}
\sum_{N,M} \int_0^\infty\int_0^\infty &f(x,y)\left(\frac{1}{(2\pi i)^2}\int_{(\ast)}\int_{(\ast)}\widetilde{W_{N,M}}(s_1,s_2)x^{s_1}y^{s_2}ds_2ds_1\right)\frac{dxdy}{xy} \\ & = \sum_{N,M}\int_0^\infty\int_0^\infty f(x,y)W_{N,M}(x^{-1},y^{-1})\frac{dxdy}{xy} \\ & = \int_0^\infty\int_0^\infty f(x,y)\frac{dxdy}{xy} = F(0,0),
\end{alignat*}
and the lemma is proved.
\end{proof}
In order to apply Lemma \ref{LemmaPartition}, we have the following result which allows us to add all the missing pairs ($N,M$) at the cost of a negligible error.
%%%%%%%%%%%%%%%%%%%%%%%%%%%%%%%%%%%%%%%%%%%%%%%%%%%%%%%%%%%%%%%%%%%%%%%%
%%%%%%%%%%%%%%%%%%%%%%%%%%%%%%%%%%%%%%%%%%%%%%%%%%%%%%%%%%%%%%%%%%%%%%%%
%%%%%			LEMMA TO ADD THE MISSING PAIRS					 %%%%%
%%%%%%%%%%%%%%%%%%%%%%%%%%%%%%%%%%%%%%%%%%%%%%%%%%%%%%%%%%%%%%%%%%%%%%%%
%%%%%%%%%%%%%%%%%%%%%%%%%%%%%%%%%%%%%%%%%%%%%%%%%%%%%%%%%%%%%%%%%%%%%%%%
\begin{lemme}\label{LemmeAdding} The quantity defined in Proposition \ref{PropositionC(N,M)} satisfies the following bound 
$$\mathcal{C}_i(\ell_1,\ell_2,N,M;q)\ll q^\varepsilon L^2\min\left\{ \left(\frac{M}{N}\right)^{1/2} , \left( \frac{q^2}{MN}\right)^{C} , \frac{(NM)^{1/2}}{q}\right\},$$
where in the second estimation, the implied constant depends on $C$.
\end{lemme}
\begin{proof} This lemma is obtained by moving suitably the different lines of integration. By suitably, we mean that we need to avoid the poles coming from the three different factors (we focus on $\mathcal{C}_1$)
$$\zeta(2s+w_1+w_2-u_1-u_2) \ , \ \Gamma(2s+w_1+w_2-u_1-u_2) \ , \ \Gamma(\frac{1}{2}-s-w_2+u_2).$$
In other words, after each manipulation, we must have (recall that $u_i$ are arbitrarily small)
$$0<\Re e (2s+w_1+w_2)<1 \ \ \mathbf{and} \ \ \Re e(s+w_i)<1/2, \ i=1,2.$$
For the first bound, we just shift the $w_2$-contour to $\Re e (w_2)=1/2-2\varepsilon$ and then, the $w_1$-contour to $\Re e (w_1)=-1/2+2\varepsilon$. 

For the second bound, we fix a constant $C>1$ and we shift the $w_i$-contours to $\Re e (w_i)=-\varepsilon /4^{2C}$. The first step is to move to $\Re e (s)=1/2$ and then to $\Re e (w_i)=-1/2+\varepsilon/4^{2C}$. The second step is : $\Re e (s)=1-2\varepsilon/4^{2C}$ and then $\Re e (w_i)=-1+4\varepsilon/4^{2C}$. Again, the third step is $\Re e (s)=3/2-8\varepsilon/4^{2C}$ and $\Re e (w_i)=-3/2+16\varepsilon /4^{2C}$. It follows that after the $j^{th}$ step, we are at ($j\geqslant 2$)
$$\Re e (s)=\frac{j}{2}-\frac{4^{j-1}\varepsilon}{2\cdot4^{2C}} \hspace{0.4cm} \mathrm{and} \hspace{0.4cm} \Re e (w_i)=-\frac{j}{2}+\frac{4^{j-1}\varepsilon}{4^{2C}}.$$
Taking $j=[2C]$ finishes the proof. 

The last part is obtained by shifting $\Re e (w_i)$ to $1/2-2\varepsilon$.
\end{proof}
This Lemma allows us to sum over all $(N,M)$, getting (recall Decomposition \eqref{Definition1C(l1,l2,N,M)})
\begin{equation}\label{DefinitionC(l1l2)}
\begin{split}
\mathcal{OD}^{MT}(\ell_1,\ell_2;q):= & \ \sum_{N,M}\mathcal{OD}^{MT}(\ell_1,\ell_2,N,M;q)=\sum_{i=1}^2\sum_{N,M}\mathcal{C}_i(\ell_1,\ell_2,N,M)\\ = & \ \frac{2}{2\pi i}\int_{(\varepsilon)}\mathcal{F}(s,\ell_1,\ell_2;q)\frac{ds}{s},
\end{split}
\end{equation}
where the function $s\mapsto \mathcal{F}(s,\ell_1,\ell_2;q)$ is defined by 
\begin{equation}\label{DefinitionMathcalF}
\mathcal{F}(s,\ell_1,\ell_2;q):=\mathfrak{D}_\gamma \cdot \left\{ \frac{G(s)H(s,u_1,u_2)\zeta(2s-u_1-u_2)\mathcal{L}(s,u_1,u_2,\ell_1,\ell_2)}{(\ell_1\ell_2)^{1/2}q^{-u_1-u_2}}\right\},
\end{equation}
with 
$$H(s,u_1,u_2):= H_1(s,0,0,u_1,u_2)+H_2(s,0,0,u_1,u_2),$$
$$\mathcal{L}(s,u_1,u_2;\ell_1,\ell_2)=\ell_1^{s-u_1}\ell_2^{s-u_2}L(s,u_1,u_2;\ell_1,\ell_2),$$
and $L(s,u_1,u_2;\ell_1,\ell_2)$ is the Dirichlet series given by 
\begin{equation}\label{DefinitionSerieDirichlet}
L(s,u_1,u_2;\ell_1,\ell_2):= \sum_{a\geqslant 1}\sum_{b\geqslant 1}\frac{\mu(a)(\ell_1,ab)^{1+2u_1}(\ell_2,ab)^{1+2u_2}}{a^{2+2u_1+2u_2}b^{1+2s+u_1+u_2}}.
\end{equation}

%%%%%%%%%%%%%%%%%%%%%%%%%%%%%%%%%%%%%%%%%%%%%%%%%%%%%%%%%%%%%%%%%%%%%%%
%%%%%%%%%%%%%%%%%%%%%%%%%%%%%%%%%%%%%%%%%%%%%%%%%%%%%%%%%%%%%%%%%%%%%%%%%%%%%%%%%%%%%%%%%%%%%%%%%%%%%%%%%%%%%%%%%%%%%%%%%%%%%%%%%%%%%%%%%%%%%%%%%%%%%%%%%%%%%%%%%%%%%%%%%%%%%%%%%%%%%%%%%%%%%%%%%%%%%%%%%%%%%%%%%%%%%%%%
%%%%%			A SYMMETRY FOR THE FUNCTION F(S,l1,l2;q)			 %%%%%
%%%%%%%%%%%%%%%%%%%%%%%%%%%%%%%%%%%%%%%%%%%%%%%%%%%%%%%%%%%%%%%%%%%%%%%%
%%%%%%%%%%%%%%%%%%%%%%%%%%%%%%%%%%%%%%%%%%%%%%%%%%%%%%%%%%%%%%%%%%%%%%%%
%%%%%%%%%%%%%%%%%%%%%%%%%%%%%%%%%%%%%%%%%%%%%%%%%%%%%%%%%%%%%%%%%%%%%%%%
%%%%%%%%%%%%%%%%%%%%%%%%%%%%%%%%%%%%%%%%%%%%%%%%%%%%%%%%%%%%%%%%%%%%%%%

\subsection{A Symmetry for the Function $\mathcal{F}(s,\ell_1,\ell_2;q)$}\label{SectionSymmetry}
As we note actually in Expression \eqref{DefinitionMathcalF}, we have completely removed powers of $q$ in the $s$-aspect in \eqref{DefinitionC(l1l2)}. Thus the usual method of evaluation consisting in shifting the $s$-contour to the left (as in section \ref{SectionComputationDiagPart}) to get a negative power of $q$ and taking the residues passed along way cannot work here. As it turns out, we will be able to evaluate explicitly the $s$-part through a residue at zero since the function $s\mapsto \mathcal{F}(s,\ell_1,\ell_2;q)$ is even in $s$. This affirmation does not follow directly from definitions \eqref{DefinitionMathcalF} and \eqref{DefinitionSerieDirichlet} and recquires a finer analysis on the Dirichlet series $\mathcal{L},$ the functional equation for the Riemann zeta function and a crucial identity for the function $H$.

%%%%%%%%%%%%%%%%%%%%%%%%%%%%%%%%%%%%%%%%%%%%%%%%%%%%%%%%%%%%%%%%%%%%%%%%
%%%%%%%%%%%%%%%%%%%%%%%%%%%%%%%%%%%%%%%%%%%%%%%%%%%%%%%%%%%%%%%%%%%%%%%%
%%%%%%%%%%%%%%%%%%%%%%%%%%%%%%%%%%%%%%%%%%%%%%%%%%%%%%%%%%%%%%%%%%%%%%%%
%%%%%			ANALYSIS OF THE DIRICHLET SERIES					 %%%%%
%%%%%%%%%%%%%%%%%%%%%%%%%%%%%%%%%%%%%%%%%%%%%%%%%%%%%%%%%%%%%%%%%%%%%%%%
%%%%%%%%%%%%%%%%%%%%%%%%%%%%%%%%%%%%%%%%%%%%%%%%%%%%%%%%%%%%%%%%%%%%%%%%
%%%%%%%%%%%%%%%%%%%%%%%%%%%%%%%%%%%%%%%%%%%%%%%%%%%%%%%%%%%%%%%%%%%%%%%%
\subsubsection{Analysis of $\mathcal{L}(s,u_1,u_2,\ell_1,\ell_2)$}
 We will use the following notations 
\begin{equation}\label{Notations}
r=2+2u_1+2u_2 \hspace{0.4cm}\mathrm{and}\hspace{0.4cm} t=1+2s+u_1+u_2
\end{equation}
and factorize $\mathcal{L}$ as an infinite product (recall that $\ell_1,\ell_2$ are cubefree and coprime)
\begin{equation}\label{Factorization}
\mathcal{L}=\prod_{p\nmid \ell_1\ell_2}\mathcal{L}_p\prod_{p ||\ell_1}\mathcal{L}_p\prod_{p^2 | \ell_1}\mathcal{L}_p \prod_{p || \ell_2}\mathcal{L}_p\prod_{p^2 | \ell_2}\mathcal{L}_p,
\end{equation}
where for each prime $p$, we have
$$\mathcal{L}_p = p^{v_p(\ell_1)(s-u_1)}p^{v_p(\ell_2)(s-u_2)}\sum_{\substack{0\leqslant a\leqslant 1 \\ b\geqslant 0}}\frac{(-1)^a p^{\sum_{i=1}^2\min(a+b,v_p(\ell_i))(1+2u_i)}}{p^{ar+bs}}.$$
We will compute the above expression according to the different cases appearing in decomposition \eqref{Factorization}. When $p \nmid \ell_1\ell_2$, we easily get 
\begin{equation}\label{Lp pnmid l1l2}
\mathcal{L}_p = \left(1-\frac{1}{p^{r}}\right)\left(1-\frac{1}{p^t}\right)^{-1}.
\end{equation}
If $p || \ell_i$, we have
\begin{equation}\label{p||li}
\mathcal{L}_p = \left(\frac{1}{p^{-s+u_i}}-\frac{1}{p^{t-s+u_i}}+\frac{1}{p^{t-s-1-u_i}}-\frac{1}{p^{r-s-1-u_i}}\right)\left(1-\frac{1}{p^t}\right)^{-1}.
\end{equation}
Finally, assuming $p^2 | \ell_i$, we obtain
\begin{equation}\label{p^2|li}
\begin{split}
&\left( \frac{1}{p^{-2s+2u_i}}+\frac{1}{p^{t-2s-1}}-\frac{1}{p^{r-2s-1}}-\frac{1}{p^{t-2s+2u_i}}-\frac{1}{p^{2t-2s-1}}\right. \\  & \ \left.+ \frac{1}{p^{2t-2s-2-2u_i}}+\frac{1}{p^{t+r-2s-1}}-\frac{1}{p^{t+r-2s-2-2u_i}}\right)\left(1-\frac{1}{p^t}\right)^{-1}.
\end{split}
\end{equation}
From \eqref{Lp pnmid l1l2}, \eqref{p||li}, \eqref{p^2|li} and the change of variables \eqref{Notations}, we infer the following factorization
\begin{equation}\label{FirstFactorisation}
\mathcal{L}(s,u_1,u_2;\ell_1,\ell_2)=\frac{\zeta(1+2s+u_1+u_2)}{\zeta(2+2u_1+2u_2)}\bm{\delta}(\ell_1;s,u_1,u_2)\bm{\delta}(\ell_2;s,u_2,u_1),
\end{equation}
where $n\mapsto \bm{\delta}(n;s,u_1,u_2)$ is the multiplicative function supported on cubefree integers and whose values on $p$ and $p^2$ are given by
%%%%%%%%%%%%%%%%%%%%%%%%%%%%%%%%%%%%%%%%%%%%%%%%%%%%%%%%%%%%%%%%%%%%%%%%
%%%%%%%%%%%%%%%%%%%%%%%%%%%%%%%%%%%%%%%%%%%%%%%%%%%%%%%%%%%%%%%%%%%%%%%%
%%%%%		VALEURES DE LA FONCTION MULTIPLICATIVE DELTA			 %%%%%
%%%%%%%%%%%%%%%%%%%%%%%%%%%%%%%%%%%%%%%%%%%%%%%%%%%%%%%%%%%%%%%%%%%%%%%%
%%%%%%%%%%%%%%%%%%%%%%%%%%%%%%%%%%%%%%%%%%%%%%%%%%%%%%%%%%%%%%%%%%%%%%%%
\begin{equation}\label{Definitiondelta1}
\begin{split}
\bm{\delta}(p;s,u_1,u_2) := & \ \left\{\frac{1}{p^{s+u_2}}\left(1-\frac{1}{p^{1+2u_1}}\right)+\frac{1}{p^{-s+u_1}}\left(1-\frac{1}{p^{1+2u_2}}\right)\right\} \\ & \times \ \left(1-\frac{1}{p^{2+2u_1+2u_2}}\right)^{-1},
\end{split}
\end{equation}
\begin{equation}\label{Definitiondelta2}
\begin{split}
\bm{\delta}(p^2;s,u_1,u_2):= & \ \left\{ \frac{1}{p^{2s+2u_2}}\left(1-\frac{1}{p^{1+2u_1}}\right)+\frac{1}{p^{-2s+2u_1}}\left(1-\frac{1}{p^{1+2u_2}}\right) \right. \\ & \ \left. +\frac{1}{p^{u_1+u_2}}\left(1+\frac{1}{p^{2u_1+2u_2}}-\frac{1}{p^{1+2u_1}}-\frac{1}{p^{1+2u_2}}\right)\right\} \\ & \times \ \left(1-\frac{1}{p^{2+2u_1+2u_2}}\right)^{-1}.
\end{split}
\end{equation}

%%%%%%%%%%%%%%%%%%%%%%%%%%%%%%%%%%%%%%%%%%%%%%%%%%%%%%%%%%%%%%%%%%%%%%%%
%%%%%%%%%%%%%%%%%%%%%%%%%%%%%%%%%%%%%%%%%%%%%%%%%%%%%%%%%%%%%%%%%%%%%%%%
%%%%%%%%%%%%%%%%%%%%%%%%%%%%%%%%%%%%%%%%%%%%%%%%%%%%%%%%%%%%%%%%%%%%%%%%
%%%%%				PARITY OF F(s,l1,l2;q)						 %%%%%
%%%%%%%%%%%%%%%%%%%%%%%%%%%%%%%%%%%%%%%%%%%%%%%%%%%%%%%%%%%%%%%%%%%%%%%%
%%%%%%%%%%%%%%%%%%%%%%%%%%%%%%%%%%%%%%%%%%%%%%%%%%%%%%%%%%%%%%%%%%%%%%%%
%%%%%%%%%%%%%%%%%%%%%%%%%%%%%%%%%%%%%%%%%%%%%%%%%%%%%%%%%%%%%%%%%%%%%%%%
\subsubsection{Parity of $\mathcal{F}(s,\ell_1,\ell_2;q)$}
We split the differential operator $\mathfrak{D}_\gamma = \sum_{i=0}^2\mathfrak{D}_\gamma^i$ with $\mathfrak{D}_\gamma^0=4\gamma^2$, $\mathfrak{D}_\gamma^1 = 2\gamma (\partial_{u_1}+\partial_{u_2})|_{u_i=0}$, $\mathfrak{D}_\gamma^2 = \partial^2_{u_i=0}$ and separate the function $\mathcal{F}(s,\ell_1,\ell_2;q)=\sum_{i=0}^2\mathcal{F}_i(s,\ell_1,\ell_2;q)$ according to this decomposition. We will show that each $\mathcal{F}_i$ is even in $s$. For this, we exploit the factorization \eqref{FirstFactorisation} and define 
\begin{alignat}{1}
&A(s,u_1,u_2;q):= \frac{G(s)H(s,u_1,u_2)\zeta(2s-u_1-u_2)\zeta(1+2s+u_1+u_2)}{\zeta(2+2u_1+2u_2)q^{-u_1-u_2}},\nonumber \\ &B(s,u_1,u_2;\ell_1,\ell_2) := \frac{\bm{\delta}(\ell_1;s,u_1,u_2)\bm{\delta}(\ell_2;s,u_2,u_1)}{(\ell_1\ell_2)^{1/2}}, \label{DefinitionA(s,u1,u2)}
\end{alignat}
in order to have
$$\mathcal{F}(s,\ell_1,\ell_2;q)=\mathfrak{D}_\gamma\cdot\left\{A(s,u_1,u_2;q)B(s,u_1,u_2;\ell_1,\ell_2)\right\}.$$
We also mention the functional equation for the Riemann zeta function
\begin{equation}\label{EquationFonctionnellezeta}
\zeta(1+2s)=\pi^{1/2+2s}\zeta(-2s)\frac{\Gamma(-s)}{\Gamma(\frac{1}{2}+s)},
\end{equation}
and a crucial identity for the function $H(s,u_1,u_2)$ (see $(8.5)$-$(8.6)$ in \cite{young})
\begin{equation}\label{CrucialIdentity}
H(s,u_1,u_2)=\pi^{1/2}\frac{\Gamma\left(\frac{-u_1-u_2+2s}{2}\right)}{\Gamma\left(\frac{1+u_1+u_2-2s}{2}\right)}\frac{\Gamma\left(\frac{\frac{1}{2}+u_1-s}{2}\right)\Gamma\left(\frac{\frac{1}{2}+u_2-s}{2}\right)}{\Gamma\left(\frac{\frac{1}{2}-u_1+s}{2}\right)\Gamma\left(\frac{\frac{1}{2}-u_2+s}{2}\right)}.
\end{equation}
%%%%%%%%%%%%%%%%%%%%%%%%%%%%%%%%%%%%%%%%%%%%%%%%%%%%%%%%%%%%%%%%%%%%%%%%
%%%%%%%%%%%%%%%%%%%%%%%%%%%%%%%%%%%%%%%%%%%%%%%%%%%%%%%%%%%%%%%%%%%%%%%%
%%%%%				LEMMA FOR THE PARITY							 %%%%%
%%%%%%%%%%%%%%%%%%%%%%%%%%%%%%%%%%%%%%%%%%%%%%%%%%%%%%%%%%%%%%%%%%%%%%%%
%%%%%%%%%%%%%%%%%%%%%%%%%%%%%%%%%%%%%%%%%%%%%%%%%%%%%%%%%%%%%%%%%%%%%%%
\begin{lemme}\label{LemmaParity} Each of the following functions are even in $s$ : $A(s,0,0;q)$,  $B(s,0,0;\ell_1,\ell_2)$, $(\partial_{u_1}+\partial_{u_2})|_{u_i=0}B$, $\partial_{u_i=0}A$, $\partial^2_{u_1u_2=0}A$ and $\partial^2_{u_1u_2=0}B$.
\end{lemme}
\begin{proof}
We begin with $A(s,0,0;q).$ Recalling Definition \eqref{DefinitionG(s)} of $G(s)$ and using the Identity \eqref{CrucialIdentity}, we have
$$A(s,0,0;q) = Q(s)\frac{\pi^{1/2-2s}\Gamma\left(\frac{\frac{1}{2}+s}{2}\right)^2\Gamma\left(\frac{\frac{1}{2}-s}{2}\right)^2}{\zeta(2)\Gamma(1/4)^4}\zeta(2s)\zeta(1+2s)\frac{\Gamma(s)}{\Gamma(\frac{1}{2}-s)}.$$
Applying now the functional equation \eqref{EquationFonctionnellezeta} to $\zeta(1+2s)$, we obtain
$$A(s,0,0;q)=Q(s)\pi\frac{\Gamma\left(\frac{\frac{1}{2}+s}{2}\right)^2\Gamma\left(\frac{\frac{1}{2}-s}{2}\right)^2}{\zeta(2)\Gamma(1/4)^4}\zeta(2s)\zeta(-2s)\frac{\Gamma(s)\Gamma(-s)}{\Gamma(\frac{1}{2}-s)\Gamma(\frac{1}{2}+s)},$$
which is of course even. For the function $B(s,0,0;\ell_1,\ell_2)$, we easily see from Definitions \eqref{Definitiondelta1} and \eqref{Definitiondelta2} that it is even, since each local factor is even. To compute the others, it will be very convenient to express them as logarithm derivatives. To be more precise, we can express $\partial_{u_i=0}A$ as
$$\partial_{u_i=0}A=A(s,0,0;q)\left(\frac{H_{u_i}}{H}+\frac{\zeta'(1+2s)}{\zeta(1+2s)}-\frac{\zeta'(2s)}{\zeta(2s)}-2\frac{\zeta'(2)}{\zeta(2)}+\log q\right).$$
On one hand, we have using \eqref{CrucialIdentity},
\begin{alignat*}{1}
2\frac{H_{u_i}}{H}= -\frac{\Gamma'(s)}{\Gamma(s)}-\frac{\Gamma'(\frac{1}{2}-s)}{\Gamma(\frac{1}{2}-s)}+\frac{\Gamma'\left(\frac{\frac{1}{2}-s}{2}\right)}{\Gamma\left(\frac{\frac{1}{2}-s}{2}\right)}+\frac{\Gamma'\left(\frac{\frac{1}{2}+s}{2}\right)}{\Gamma\left(\frac{\frac{1}{2}+s}{2}\right)}.
\end{alignat*}
On the other hand, we have by applying the logarithm derivative to \eqref{EquationFonctionnellezeta},
$$2\frac{\zeta'(1+2s)}{\zeta(1+2s)}= 2\log(\pi)-2\frac{\zeta'(-2s)}{\zeta(-2s)}-\frac{\Gamma'(-s)}{\Gamma(-s)}-\frac{\Gamma'(\frac{1}{2}+s)}{\Gamma(\frac{1}{2}+s)}.$$
It follows that $\partial_{u_i} A$ is even. Similarily, we can compute $\partial^2_{u_1u_2=0}A$ in a fancy way :
\begin{alignat*}{1}
\partial^2_{u_1u_2=0}&A = A(s,0,0;q)\left\{ \left(\frac{H_{u_1}}{H}+\frac{\zeta'(1+2s)}{\zeta(1+2s)}-\frac{\zeta'(2s)}{\zeta(2s)}-2\frac{\zeta'(2)}{\zeta(2)}+\log q\right)^2\right. \\+ & \left. \partial_{u_2=0}\left(\frac{H_{u_1}(s,0,u_2)}{H(s,0,u_2)}+\frac{\zeta'(1+2s+u_2)}{\zeta(1+2s+u_2)}-\frac{\zeta'(2s-u_2)}{\zeta(2s-u_2)}-\frac{\zeta'(2+2u_2)}{\zeta(2+2u_2)}\right) \right\}.
\end{alignat*}
We already know that the first line is even. For the second, we have by \eqref{CrucialIdentity}, 
$$\frac{H_{u_1}(s,0,u_2)}{H(s,0,u_2)}=-\frac{1}{2}\frac{\Gamma'\left(\frac{-u_2+2s}{2}\right)}{\Gamma\left(\frac{-u_2+2s}{2}\right)}-\frac{1}{2}\frac{\Gamma'\left(\frac{1+u_2-2s}{2}\right)}{\Gamma\left(\frac{1+u_2-2s}{2}\right)}+\frac{1}{2}\frac{\Gamma'\left(\frac{\frac{1}{2}-s}{2}\right)}{\Gamma\left(\frac{\frac{1}{2}-s}{2}\right)}+\frac{1}{2}\frac{\Gamma'\left(\frac{\frac{1}{2}+s}{2}\right)}{\Gamma\left(\frac{\frac{1}{2}+s}{2}\right)},$$
and using again \eqref{EquationFonctionnellezeta}, we infer
$$\frac{\zeta'(1+2s+u_2)}{\zeta(1+2s+u_2)}=\log(\pi)-\frac{\zeta'(-2s-u_2)}{\zeta(-2s-u_2)}-\frac{1}{2}\frac{\Gamma'(\frac{-2s-u_2}{2})}{\Gamma(\frac{-2s-u_2}{2})}-\frac{1}{2}\frac{\Gamma'(\frac{1+2s+u_2}{2})}{\Gamma(\frac{1+2s+u_2}{2})}.$$
Hence the parenthesis in the second line of $\partial^2_{u_1u_2=0}$ is preserved under the action of $\partial_{u_2=0}$. It remains to evaluate $(\partial_{u_1}+\partial_{u_2})|_{u_i=0}B$ and $\partial^2_{u_1u_2=0}B$. In this case precisely, it is very useful to express as logarithm derivatives since $B(s,u_1,u_2;\ell_1,\ell_2)$ can be written as a product of the primes dividing $\ell_1\ell_2$ and the logarithm derivative transforms this product into a sum in which each term will be even. Indeed, we compute
$$(\partial_{u_1}+\partial_{u_2})|_{u_i=0}B=B(s,0,0;\ell_1,\ell_2)\sum_{i=1}^2\sum_{\P | \ell_i}\frac{(\partial_{u_1}+\partial_{u_2})|_{u_i=0}\bm{\delta}(\P^{v_\P(\ell_i)};s,u_1,u_2)}{\bm{\delta}(\P^{v_\P(\ell_i)};s,0,0)},$$
and it is easy to check that each term appearing in the sum above is even. We mention that the $\partial_{u_i=0}\bm{\delta}$ are not individually even; it is $(\partial_{u_1}+\partial_{u_2})|_{u_i=0}$ that creates the symmetry (see \eqref{Sym}). Finally, we have for the last one (recall that $u_1$ and $u_2$ are swapped when we deal with $\ell_2$)
\begin{alignat*}{1}
\partial^2_{u_1u_2=0}B= & \ B(s,0,0;\ell_1,\ell_2)\left\{ \left(\frac{\bm{\delta}_{u_1}(\ell_1;s,0,0)}{\bm{\delta}(\ell_1;s,0,0)}+\frac{\bm{\delta}_{u_2}(\ell_2;s,0,0)}{\bm{\delta}(\ell_2;s,0,0)}\right)\right.\\ \times &\left. \left(\frac{\bm{\delta}_{u_2}(\ell_1;s,0,0)}{\bm{\delta}(\ell_1;s,0,0)}+\frac{\bm{\delta}_{u_1}(\ell_2;s,0,0)}{\bm{\delta}(\ell_2;s,0,0)}\right)+\partial_{u_2}|_{u_2=0} \left(\frac{\bm{\delta}_{u_1}(\ell_1;s,0,u_2)}{\bm{\delta}(\ell_1;s,0,u_2)}\right) \right. \\  + & \left. \partial_{u_1}|_{u_1=0} \left(\frac{\bm{\delta}_{u_2}(\ell_2;s,u_1,0)}{\bm{\delta}(\ell_2;s,u_1,0)}\right) \right\}.
\end{alignat*}
Using the symmetry 
\begin{equation}\label{Sym}
\frac{\bm{\delta}_{u_1}(\ell_i;-s,0,0)}{\bm{\delta}(\ell_i;-s,0,0)}=\frac{\bm{\delta}_{u_2}(\ell_i;s,0,0)}{\bm{\delta}(\ell_i,s,0,0)},
\end{equation}
we remark that the product of the two parentheses is invariant under $s\leftrightarrow -s$ since it just switches the two factors. We conclude this Lemma by checking that the local value (at a prime $\P$) of the two order two terms is given by 
$$\sum_{i=1}^2\left(\frac{\bm{\delta}_{u_1,u_2}(\P^{v_\P(\ell_i)};s,0,0)}{\bm{\delta}(\P^{v_\P(\ell_i)};s,0,0)}-\frac{\bm{\delta}_{u_1}(\P^{v_\P(\ell_i)};s,0,0)\bm{\delta}_{u_2}(\P^{v_\P(\ell_i)};s,0,0)}{\bm{\delta}(\P^{v_\P(\ell_i)};s,0,0)^2}\right)$$
and each individual term is even by a direct computation and \eqref{Sym} (using of course \eqref{Definitiondelta1} and \eqref{Definitiondelta2}).
\end{proof}
%%%%%%%%%%%%%%%%%%%%%%%%%%%%%%%%%%%%%%%%%%%%%%%%%%%%%%%%%%%%%%%%%%%%%%%%
%%%%%%%%%%%%%%%%%%%%%%%%%%%%%%%%%%%%%%%%%%%%%%%%%%%%%%%%%%%%%%%%%%%%%%%%
%%%%%					END OF THE LEMMA							 %%%%%
%%%%%%%%%%%%%%%%%%%%%%%%%%%%%%%%%%%%%%%%%%%%%%%%%%%%%%%%%%%%%%%%%%%%%%%
%%%%%%%%%%%%%%%%%%%%%%%%%%%%%%%%%%%%%%%%%%%%%%%%%%%%%%%%%%%%%%%%%%%%%%%

%%%%%%%%%%%%%%%%%%%%%%%%%%%%%%%%%%%%%%%%%%%%%%%%%%%%%%%%%%%%%%%%%%%%%%%%
%%%%%%%%%%%%%%%%%%%%%%%%%%%%%%%%%%%%%%%%%%%%%%%%%%%%%%%%%%%%%%%%%%%%%%%%
%%%%%			PROPOSITION FINAL FOR THE PARITY					 %%%%%
%%%%%%%%%%%%%%%%%%%%%%%%%%%%%%%%%%%%%%%%%%%%%%%%%%%%%%%%%%%%%%%%%%%%%%%
%%%%%%%%%%%%%%%%%%%%%%%%%%%%%%%%%%%%%%%%%%%%%%%%%%%%%%%%%%%%%%%%%%%%%%%
\begin{proposition}\label{PropositionParity} The function $\mathcal{F}_i(s,\ell_1,\ell_2;q)$ is even for $i=0,1,2$.
\end{proposition}
\begin{proof}
We do not mention the arguments of the functions and write $A_{u_i}$ instead of $\partial_{u_i=0}A$. Since $A_{u_1}=A_{u_2}$, we have
$$\mathcal{F}_0 = 4\gamma^2 AB, \ \mathcal{F}_1=2\gamma (A_{u_1}+A_{u_2})B+ A(B_{u_1}+B_{u_2}),$$
$$\mathcal{F}_2= A_{u_1u_2}B+(B_{u_1}+B_{u_2})A_{u_i}+AB_{u_1u_2},$$
and the conclusion follows directly from Lemma \ref{LemmaParity}.
\end{proof}
\begin{cor} The off-diagonal main term \eqref{DefinitionC(l1l2)} equals
\begin{equation}\label{ValueODMT}
\mathcal{OD}^{MT}(\ell_1,\ell_2;q)= \sum_{i=0}^2\mathrm{Res}_{s=0}\left\{\frac{\mathcal{F}_i(s,\ell_1,\ell_2;q)}{s}\right\}.
\end{equation}
\end{cor}

\subsubsection{A note on odd characters}\label{RemkOdd}In this paper, we concentrated exclusively on even characters. The contribution of the odd characters carries through in the same way with slight changes that we mention now. First of all, the function $G(s)$ defined in \eqref{DefinitionG(s)} becomes
$$G(s)=\pi^{-2s}\frac{\Gamma\left(\frac{3/2+s}{2}\right)^4}{\Gamma(3/4)^4}Q(s),$$
so we need to remove the original $G(s)$ in the diagonal main term \eqref{ExpressionDiagPart1}. The estimations of the error terms as did in Sections \ref{Sectionl-adic} and \ref{SectionSpectral} carry through as before. The most significant change appears in the treatment of the off-diagonal main term. Here, the last gamma factor coming from the dual terms (c.f \eqref{FinalMellinTransform2}) is subtracted in the definition of $H$ instead to be added. Fortunately, the parity of the function $\mathcal{F}(s,\ell_1,\ell_2;q)$ is preserved trough a similar identity (apply Lemma 8.4 in \cite{young} with $a=1/2-s+u_1$ and $b=1/2-s+u_2$)
\begin{equation}\label{IdentityH-}
H(s,u_1,u_2)=\pi^{1/2}\frac{\Gamma\left(\frac{2s-u_1-u_2}{2}\right)}{\Gamma\left(\frac{1-2s+u_1+u_2}{2}\right)}\frac{\Gamma\left(\frac{\frac{3}{2}-s+u_1}{2}\right)\Gamma\left(\frac{\frac{3}{2}-s+u_2}{2}\right)}{\Gamma\left(\frac{\frac{3}{2}+s-u_1}{2}\right)\Gamma\left(\frac{\frac{3}{2}+s-u_2}{2}\right)}.
\end{equation}

\section{The Mollified Fourth Moment}\label{SectionMollification}%%%%%%%%%%%%%%%%%%%%%%%%%%%%%%%%%%%%%%%%%%%%%%%%%%%%%%%%%%%%%%%%%%%%%%%%%%%%%%%%%%%%%%%%%%%%%%%%%%%%%%%%%%%%%%%%%%%%%%%%%%%%%%%%%%%%%%%%%%%%%%%%%%%%%%%%%%%%%%%%%%%%%%%%%%%%%%%%%%%%%%%%%%%%%%%%%%%%%%%%%%%%%%%%%%%%%%%%
In this last section, we exploit Theorem \ref{FirstTheorem} to establish an asymptotic formula for a mollified fourth moment of the form
\begin{equation}\label{DefinitionM4}
\mathscr{M}^4(q):= \frac{1}{\phi^\ast (q)}\sum_{\substack{\chi \ (\mathrm{mod \ }q) \\ \chi\neq 1}}|L(\chi,\tfrac{1}{2})M(\chi)|^4,
\end{equation}
where $M(\chi)$ is our mollifier which presents as a short linear form 
\begin{equation}\label{DefinitionMollifier}
M(\chi):= \sum_{\ell\geqslant 1}\frac{\bm{x}(\ell)\chi(\ell)}{\ell^{1/2}},
\end{equation}
and the coefficients $\bm{x}(\ell)$ are given by 
\begin{equation}\label{Definitionx(l)}
\bm{x}(\ell):= \mu(\ell)\delta_{\ell\leqslant L}P\left(\frac{\log\left(\frac{L}{\ell}\right)}{\log L}\right),
\end{equation}
for some suitable polynomial $P\in\mathbb{R}[X]$ that satisfies $P(0)=0$ and $P(1)=1$. The parameter $L$ will be a small power of $q$ ($L=q^\lambda$ with $\lambda>0$) and $\mu$ is the Möbius function. Now for $P(X)=\sum_{k\geqslant 1} a_kX^k\in\mathbb{R}[X]$, we define 
\begin{equation}\label{DefinitionPL}
\widehat{P_L}(s):= \sum_{k\geqslant 1}a_k\frac{k!}{(s\log L)^k}.
\end{equation}
Then we have the following integral representation which can be easily deduced using contour shift.
\begin{lemme}\label{LemmaIntegralRepresentationP()}
For $L>0$ not an integer and $\ell\in\mathbb{N}$, we have 
$$\delta_{\ell\leqslant L}P\left(\frac{\log\left(\frac{L}{\ell}\right)}{\log L}\right)=\frac{1}{2\pi i}\int_{(2)}\frac{L^s}{\ell^s}\widehat{P_L}(s)\frac{ds}{s}.$$
\end{lemme}

%%%%%%%%%%%%%%%%%%%%%%%%%%%%%%%%%%%%%%%%%%%%%%%%%%%%%
%%%%%%%%%%%%%%%%%%%%%%%%%%%%%%%%%%%%%%%%%%%%%%%%%%%%%
%%%%%%%%%%%%%%%%%%%%%%%%%%%%%%%%%%%%%%%%%%%%%%%%%%%%%
%%%%%%%%%%			REDUCTION TO THE TWISTED FOURTH MOMENT         %%%%%%%%%%
%%%%%%%%%%%%%%%%%%%%%%%%%%%%%%%%%%%%%%%%%%%%%%%%%%%%%
%%%%%%%%%%%%%%%%%%%%%%%%%%%%%%%%%%%%%%%%%%%%%%%%%%%%%
%%%%%%%%%%%%%%%%%%%%%%%%%%%%%%%%%%%%%%%%%%%%%%%%%%%%%
\subsection{Reduction to the Twisted Fourth Moment}
Opening the fourth power in \eqref{DefinitionM4}, we obtain
$$\mathscr{M}^4(q)=\mathop{\sum\sum}_{a,b,c,d}\frac{\bm{x}(a)\bm{x}(b)\bm{x}(c)\bm{x}(d)}{(abcd)^{1/2}}\mathscr{T}^4(ab,cd;q).$$
To get the primality condition between $ab$ and $cd$, we explicit the coefficients $\bm{x}$ and then use the integral representation provided by Lemma \ref{LemmaIntegralRepresentationP()}
\begin{equation}\label{M4(2)}
\begin{split}
\mathscr{M}^4(q)= & \ \frac{1}{(2\pi i)^4}\int\limits_{(2)}\int\limits_{(2)}\int\limits_{(2)}\int\limits_{(2)}\prod_{i=1}^4 L^{z_i}\widehat{P_L}(z_i) \\ & \times \mathop{\sum\sum}_{a,b,c,d}\frac{\mu(a)\mu(b)\mu(c)\mu(d)}{a^{1/2+z_1}b^{1/2+z_2}c^{1/2+z_3}d^{1/2+z_4}}\mathscr{T}^4(ab,cd;q)\frac{dz_1dz_2dz_3dz_4}{z_1z_2z_3z_4}.
\end{split}
\end{equation}
For the sum in the second line, we group the variables $ab=\ell_1$, $cd=\ell_2$ and then set $d=(\ell_1,\ell_2)$, getting that this sum equals
\begin{equation}\label{Sum}
\mathop{\sum\sum\sum}_{\substack{d\geqslant 1 \ (\ell_1,\ell_2)=1 \\ (d\ell_1\ell_2,q)=1}}\frac{\mu_{2,z_1-z_2}(d\ell_1)\mu_{2,z_3-z_4}(d\ell_2)}{d^{1+z_1+z_3}\ell_1^{1/2+z_1}\ell_2^{1/2+z_3}}\mathscr{T}^4(\ell_1,\ell_2;q),
\end{equation}
where  for any complex number $\nu\in\mathbb{C}$, $\mu_{2,\nu}$ is the inverse of the generalized divisor function $\sigma_{\nu}(n)=\sum_{d|n}d^\nu$ for the Dirichlet convolution, namely
\begin{equation}\label{Definitionmu2}
\mu_{2,\nu}(n) = \sum_{ab=n}\mu(a)\mu(b)b^\nu.
\end{equation}
In particular this is a multiplicative function supported on cubefree integers and whose values on prime powers are given by 
\begin{equation}\label{ValueprimePower}
\mu_{2,\nu}(p)= -1-p^\nu \ \ , \ \ \mu_{2,\nu}(p^2)= p^\nu \ \, \ \ \mu_{2,\nu}(p^j)=0 \ , \ \forall \ j\geqslant 3.
\end{equation}
Inserting \eqref{Sum} in \eqref{M4(2)}, we see (by shifting the $z_i$-line to $\Re e (z_i)=C>1$) that we can assume that $\ell_i\leqslant L^{2+\varepsilon}$ for an error cost of $O(q^{-100})$ because $L$ is a positive power of $q$. We are now in position to apply Theorem \ref{FirstTheorem} to $\mathscr{T}^4(\ell_1,\ell_2;q)$. Once we applied the Theorem, we can again remove the condition $\ell_i\leqslant L^{2+\varepsilon}$ for the same cost and sum over all $\ell_i$. The decomposition into a diagonal, off-diagonal and error term leads to the following decomposition
\begin{equation}\label{DecompositionM4}
\mathscr{M}^4(q)=\mathscr{M}_D^4(q)+\mathscr{M}^4_{OD}(q)+\mathscr{E}(L,q),
\end{equation}
where
\begin{equation}\label{DefinitionM4D}
\begin{split}
\mathscr{M}^4_D(q) := & \ \frac{1}{(2\pi i)^4}\int\limits_{(2)}\int\limits_{(2)}\int\limits_{(2)}\int\limits_{(2)}\prod_{i=1}^4 L^{z_i}\widehat{P_L}(z_i)\mathop{\sum\sum\sum}_{\substack{d\geqslant 1 \ (\ell_1,\ell_2)=1 \\ (d\ell_1\ell_2,q)=1}}\frac{\mu_{2,z_1-z_2}(d\ell_1)\mu_{2,z_3-z_4}(d\ell_2)}{d^{1+z_1+z_3}\ell_1^{1/2+z_1}\ell_2^{1/2+z_3}} \\ & \times \frac{1}{2}\left(\mathscr{T}_D(\ell_1,\ell_2;q)+\mathscr{T}_D^-(\ell_1,\ell_2;q)\right)\frac{dz_1dz_2dz_3dz_4}{z_1z_2z_3z_4},
\end{split}
\end{equation}
\begin{equation}\label{DefinitionM4OD}
\begin{split}
\mathscr{M}^4_{OD}(q) := & \ \frac{1}{(2\pi i)^4}\int\limits_{(2)}\int\limits_{(2)}\int\limits_{(2)}\int\limits_{(2)}\prod_{i=1}^4 L^{z_i}\widehat{P_L}(z_i)\mathop{\sum\sum\sum}_{\substack{d\geqslant 1 \ (\ell_1,\ell_2)=1 \\ (d\ell_1\ell_2,q)=1}}\frac{\mu_{2,z_1-z_2}(d\ell_1)\mu_{2,z_3-z_4}(d\ell_2)}{d^{1+z_1+z_3}\ell_1^{1/2+z_1}\ell_2^{1/2+z_3}} \\ & \times \frac{1}{2}\left(\mathcal{OD}^{MT}(\ell_1,\ell_2;q)+\mathcal{OD}^{MT,-}(\ell_1,\ell_2;q)\right)\frac{dz_1dz_2dz_3dz_4}{z_1z_2z_3z_4},
\end{split}
\end{equation}
\begin{equation}\label{DefinitionM4D}
\begin{split}
\mathscr{E}(L,q) := & \ \frac{1}{(2\pi i)^4}\int\limits_{(2)}\int\limits_{(2)}\int\limits_{(2)}\int\limits_{(2)}\prod_{i=1}^4 L^{z_i}\widehat{P_L}(z_i)\mathop{\sum\sum\sum}_{\substack{d\geqslant 1 \ (\ell_1,\ell_2)=1 \\ (d\ell_1\ell_2,q)=1}}\frac{\mu_{2,z_1-z_2}(d\ell_1)\mu_{2,z_3-z_4}(d\ell_2)}{d^{1+z_1+z_3}\ell_1^{1/2+z_1}\ell_2^{1/2+z_3}} \\ & \times O\left(q^\varepsilon\frac{(\ell_1\ell_2)^{3/2}L^{10}}{q^\eta}\right)\frac{dz_1dz_2dz_3dz_4}{z_1z_2z_3z_4},
\end{split}
\end{equation}
where $\mathscr{T}_D(\ell_1,\ell_2;q)$, $\mathcal{OD}^{MT}(\ell_1,\ell_2;q)$ are respectively given by \eqref{ExpressionDiagPart1}, \eqref{ValueODMT} and $\mathscr{T}^-,\mathcal{OD}^{MT,-}$ are the contribution of the odd characters (see § \ref{RemkOdd} for the necessary changes). We can immediately evaluate the error term $\mathcal{E}(L,q)$. For this, we move the $z_i$-lines to $\Re e (z_i)=2+\varepsilon$, making all summations absolutely convergent, obtaining therefore 
\begin{equation}\label{FinalErrorTerm}
\mathscr{E}(L,q)=O\left(q^\varepsilon\frac{L^{18}}{q^\eta}\right),
\end{equation}
which makes sense as long as 
\begin{equation}\label{FinalLambdaCondition}
\lambda < \frac{\eta}{18}=\frac{1-6\theta}{18\cdot 14}=\frac{11}{8064}\approx \frac{1}{733}.
\end{equation}

%%%%%%%%%%%%%%%%%%%%%%%%%%%%%%%%%%%%%%%%%%%%%%%%%%%%%%%%%%%%%%%%%%%%%%%
%%%%%%%%%%%%%%%%%%%%%%%%%%%%%%%%%%%%%%%%%%%%%%%%%%%%%%%%%%%%%%%%%%%%%%%
%%%%%%%%%%%%%%%%%%%%%%%%%%%%%%%%%%%%%%%%%%%%%%%%%%%%%%%%%%%%%%%%%%%%%%%
%%%%%			COMPUTATION OF THE RESIDUE AT s=0				%%%%%
%%%%%%%%%%%%%%%%%%%%%%%%%%%%%%%%%%%%%%%%%%%%%%%%%%%%%%%%%%%%%%%%%%%%%%%
%%%%%%%%%%%%%%%%%%%%%%%%%%%%%%%%%%%%%%%%%%%%%%%%%%%%%%%%%%%%%%%%%%%%%%%
%%%%%%%%%%%%%%%%%%%%%%%%%%%%%%%%%%%%%%%%%%%%%%%%%%%%%%%%%%%%%%%%%%%%%%%
\subsection{Evaluation of $\mathscr{M}^4_D(q)$}\label{SectionM4D}
We focus on $\mathscr{T}_D(\ell_1,\ell_2;q)$ since the other gives the same result. Indeed, the change is on the function $G(s)$ but we will see that the terms which contribute in the asymptotic formula only involve $G(0)$, which is equal to $1$ in any cases. We now recall that $\mathscr{T}_D(\ell_1,\ell_2;q)$ is given by the following residue (up to some error term, see Proposition \ref{PropositionDiag})
$$2\mathrm{Res}_{s=0}\left\{\frac{G(s)q^{2s}}{16s^5\zeta(2+4s)}F(\ell_1\ell_2;s)H(s)\right\},$$
where we factorize $\zeta(1+2s)=(2s)^{-1}H(s)$ with $H(0)=1$. Since it is a pole of order five, this residue can be expressed as a linear combination in which the sum of the order of derivation of each function (except $s^{-5}$) is four, but it turns out that only the terms where $G(s)H(s)\zeta(2+4s)^{-1}$ are not derived that contribute in our asymptotic formula; the contribution of the others are at most $O_\lambda(\log^{-1}q)$. Hence we infer
\begin{equation}\label{M4(q)FirstDecomposition}
\mathscr{M}^4_D(q)=\frac{1}{8\zeta(2)}\sum_{i+j=4}(i!j!)^{-1}(2\log q)^j\mathscr{M}_D^4(i)+O_\lambda\left(\frac{1}{\log q}\right),
\end{equation}
where 
\begin{equation}\label{DefinitionM(i,q)}
\mathscr{M}_D^4(i):= \frac{1}{(2\pi i)^4}\int\limits_{(2)}\int\limits_{(2)}\int\limits_{(2)}\int\limits_{(2)}\prod_{k=1}^4L^{z_k}\widehat{P_L}(z_k)\partial^i_{s=0}\mathcal{L}(s,z_1,z_2,z_3,z_4)\frac{dz_1dz_2dz_3dz_4}{z_1z_2z_3z_4},
\end{equation}
and $\mathcal{L}(s,z_1,...,z_4)$ is the Dirichlet series defined by (recall the definition of $F(\ell_1\ell_2;s)$ given in Proposition \ref{PropositionDiag})
\begin{equation}\label{DefinitionL(s,z1,...,z4)}
\mathcal{L}(s,z_1,z_2,z_3,z_4):= \mathop{\sum\sum\sum}_{d\geqslant 1 \ (\ell_1,\ell_2)=1}\frac{\mu_{2,z_1-z_2}(d\ell_1)\mu_{2,z_3-z_4}(d\ell_2)f(\ell_1;1+2s)f(\ell_2;1+2s)}{d^{1+z_1+z_3}\ell_1^{1+z_1+s}\ell_2^{1+z_3+s}}.
\end{equation}
Writing $\mathcal{L}(s,z_1,z_2,z_3,z_4)$ as an Euler product using \eqref{ValueprimePower} and \eqref{Valuesf} and examining the polar parts leads to (see also \cite[Lemma 2.24 \& Corollary 2.25]{blomer2018})
\begin{lemme}\label{LemmaDirichletSerie1} The Dirichlet series $\mathcal{L}(s,z_1,z_2,z_3,z_4)$ factorizes as 
\begin{equation}\label{Eq22}
\mathcal{L}(s,z_1,z_2,z_3,z_4)= \mathscr{P}(s,z_1,z_2,z_3,z_4)\prod_{i=1}^2\left(\frac{\zeta(1+z_1+z_{i+2})\zeta(1+z_2+z_{i+2})}{\zeta^2(1+z_i+s)\zeta^2(1+z_{i+2}+s)}\right),
\end{equation}
where $\mathscr{P}(s,z_1,z_2,z_3,z_4)$ is an explicit Euler product which is absolutely convergent in the region $\Re e (s), \Re e  (z_i)\geqslant -\kappa$ for some $\kappa>0$.
\end{lemme}
It will also be convenient to isolate the polar parts of the various zeta functions appearing in Lemma \ref{LemmaDirichletSerie1}. Namely, we write 
\begin{equation}\label{PolarPartDirichletSerie1}
\mathcal{L}(s,z_1,z_2,z_3,z_4)=\frac{(z_1+s)^2(z_2+s)^2(z_3+s)^2(z_4+s)^2}{(z_1+z_3)(z_1+z_4)(z_2+z_3)(z_2+z_4)}\mathscr{F}(s,z_1,z_2,z_3,z_4),
\end{equation}
where this time $\mathscr{F}(s,z_1,z_2,z_3,z_4)$ is an holomorphic function which does not vanish in a domain that we describe now: From the prime number Theorem, we know that there exists an absolute constant $c>0$ such that the Riemann zeta function does not vanish in 
$$\Omega = \left\{ s \in \mathbb{C} \ \big| \ \Re e (s) \geqslant 1-\frac{c}{\log \left(2+|\Im m (s)|\right)}\right\}.$$
The function $\mathscr{F}$ is therefore holomorphic in the domain $$\{\Re e (s),\Re e (z_i)\geqslant -\kappa\}\cap\{ 1+z_i+s \in \Omega \ , i=1,...,4 \}.$$
We insert now the factorization \eqref{PolarPartDirichletSerie1} in \eqref{DefinitionM(i,q)} and apply the operator $\partial^i_{s=0}$. In this linear combination, we again retain only the terms where $\mathscr{F}$ still not derived since the others will also not contribute in the the formula of Theorem \ref{SecondTheorem}. This is of course not obvious right now, but it is enough to convince yourself to apply exactly the same calculations that follow from now on, but with $j_s+j_1+...+j_4<4$ and with at least one derivative of $\mathscr{F}$ at $s=0$ in expressions \eqref{RedefinitionM4}, \eqref{DefinitionM4(j1,...,j4,F)} below. It follows that \eqref{M4(q)FirstDecomposition} can be written in the form
\begin{equation}\label{RedefinitionM4}
\mathscr{M}_D^4(q)=\frac{1}{8\zeta(2)}\mathop{\sum\sum\sum\sum}_{\substack{j_s+j_1+j_2+j_3+j_4=4 \\ 0\leqslant j_k \leqslant 2 \ , \ k=1,...,4}}(2\log q)^{j_s}\mathscr{C}_{j_s,j_1,...,j_4}\mathscr{M}_D^4(j_1,...,j_4;\mathscr{F}) + O_\lambda\left(\frac{1}{\log q}\right),
\end{equation}
where 
$$\mathscr{C}_{j_s,j_1,...,j_4} :=\bm{\alpha}(j_1,...,j_4)C_{j_s,j_1,...,j_4},$$
$$C_{j_s,j_1,...,j_4}:=\frac{1}{j_s!j_1!j_2!j_3!j_4!},$$
$$\bm{\alpha}(j_1,...,j_4):= \prod_{i=1}^4 \alpha(j_i) \ \ , \ \alpha(0)=1 \, \ \alpha(1)=\alpha(2)=2,$$
and
\begin{equation}\label{DefinitionM4(j1,...,j4,F)}
\begin{split}
\mathscr{M}_D^4(j_1,...,j_4;\mathscr{F}):= & \ \frac{1}{(2\pi i)^4}\int\limits_{(\ast)}\int\limits_{(\ast)}\int\limits_{(\ast)}\int\limits_{(\ast)}\prod_{i=1}^4L^{z_i}\widehat{P_L}(z_i)\mathscr{F}(0,z_1,z_2,z_3,z_4) \\  \times & \ \frac{z_1^{2-j_1}z_2^{2-j_2}z_3^{2-j_3}z_4^{2-j_4}}{(z_1+z_3)(z_1+z_4)(z_2+z_3)(z_2+z_4)}\frac{dz_4dz_3dz_2dz_1}{z_4z_3z_2z_1}.
\end{split}
\end{equation}
By $(\ast)$ under the integrals, we mean that $1+z_i\in\Omega$ with $\Re e (z_i)>0$ and furthermore, we want that the real parts sufficiently small so that all future manipulations are justified, for example $1+z_1+...+z_4$ also belongs to $\Omega$.

\subsubsection{Shifting the $z_i$-contours}
\noindent We focus now on the calculation of $\mathscr{M}_D^4(j_1,...,j_4;\mathscr{F})$ for a fixed multi-index $(j_1,...,j_4)$ such that $j_1+...+j_4\leqslant 4$. We also choose the polynomial in \eqref{Definitionx(l)} to be $P(X)=X^2\footnote{Of course we could take a general polynomial $P(X)=\sum_{k\geqslant 1}a_kX^k$ and try to minimize the coefficients at the end. In \cite[Prop 6.5]{blomer2018}, the authors found $P(X)=X$ for the mollification of the second moment of twisted $L$-functions $L(f\otimes\chi)$, where $f$ is a cuspidal Hecke eigenform. This polynomial does not work in the case where $f=E$ is the Eisenstein series mentioned in Section \ref{SectionIntro} because of the pole of the zeta function. The choice $P(X)=X^2$ appears to be the simplest adaptation to our treatment}$. Using the Definition \eqref{DefinitionPL} of $\widehat{P_L}$ and we get 
\begin{alignat}{1}
\mathscr{M}_D^4(j_1,...,j_4;\mathscr{F})=\frac{16}{(\log L)^8(2\pi i)^4}\int\limits_{(\ast)}\int\limits_{(\ast)}\int\limits_{(\ast)}\int\limits_{(\ast)}&\frac{L^{z_1+z_2+z_3+z_4}\mathscr{F}(0,z_1,...,z_4)}{(z_1+z_3)(z_1+z_4)(z_2+z_3)(z_2+z_4)}\nonumber \\ \times & \frac{dz_4dz_3dz_3dz_1}{z_4^{1+j_4}z_3^{1+j_3}z_2^{1+j_2}z_1^{1+j_1}}. \label{M2(q,i)}
\end{alignat}
We start by shifting the $z_4$-contour left to zero such that $\Re e (z_1+z_2+z_3+z_4)<0$, passing three poles : one of order $1+j_4$ at $z_4=0$ and two of order one at $z_4=-z_1$ and $z_4=-z_2$. Since $\Re e (z_1+z_2+z_3+z_4)<0$, the resulting integral is at most $O(\log^{-8}L)$. We will analyze separately the three poles and find out that each of them contributes.
\vspace{0.2cm}

%%%   					 THE POLE AT Z4=-Z1					       %%%
%%%%%%%%%%%%%%%%%%%%%%%%%%%%%%%%%%%%%%%%%%%%%%%%%%%%%%%%%%%%%%%%%%%%%%%%%%%%%%%%%%%%%%%%%%%%%%%%%%%%%%%%%%%%%%%%%%%%%%%%%%%%%%%%%%%%%%%%%%%%%%%%
\noindent $\mathbf{The \ pole \ at}$ $z_4=-z_1$ $\mathbf{:}$ Since it is a simple pole, the residue at $z_4=-z_1$ is given by 
\begin{equation}\label{Polez4=-z1}
\frac{16(-1)^{1+j_4}}{(\log L)^8(2\pi i)^3}\int\limits_{(\ast)}\int\limits_{(\ast)}\int\limits_{(\ast)}\frac{L^{z_2+z_3}\mathscr{F}(0,z_1,z_2,z_3,-z_1)}{(z_1+z_3)(z_2+z_3)(z_2-z_1)}\frac{dz_3dz_2dz_1}{z_3^{1+j_3}z_2^{1+j_2}z_1^{2+j_1+j_4}}.
\end{equation}
In this integral, we move the $z_3$-line such that $\Re e (z_2+z_3)<0$, passing a pole of order $1+j_3$ at $z_3=0$ and one of order one at $z_3=-z_2$. We immediately see that the one at $z_3=-z_2$ contributes at most $O(\log^{-8}L)$. The residue at $z_3=0$ is given by the following linear combination (again we do not take in account the $z_3$-derivatives of $\mathscr{F}$ 
$$16 (-1)^{1+j_4}\mathop{\sum\sum}_{k+\ell+n=j_3}\bm{\beta}(k,\ell,n)\frac{(\log L)^{k-8}}{(2\pi i)^2}\int\limits_{(\ast)}\int\limits_{(\ast)}\frac{L^{z_2}\mathscr{F}(0,z_1,z_2,0,-z_1)}{(z_2-z_1)z_1^{3+\ell+j_1+j_4}z_2^{2+n+j_2}}dz_2dz_1,$$
where for any $(a,b,c)\in\mathbb{N}^3$, we defined 
\begin{equation}\label{Definitionbeta}
\bm{\beta}(a,b,c):=\frac{(-1)^{b+c}}{a!}.
\end{equation}
We fix $k+n+\ell=j_3$ and we move the $z_2$-line to $\Re e (z_2)<0$, passing two poles : one at $z_2=z_1$ of order $1$ and the other at $z_2=0$ of order $2+n+j_2$. The last one is easily see to be bounded by $O((\log L)^{1+k+n+j_2-8})=O((\log L)^{1+j_2+j_3-\ell-8})$ and thus, will contribute at the end at most $O(\log^{-3}q)$ (recall that $L=q^\lambda$). The residue at $z_2=z_1$ equals
$$16(-1)^{1+j_4}\mathop{\sum\sum}_{k+\ell+n=j_3}\bm{\beta}(k,\ell,n)\frac{(\log L)^{k-8}}{2\pi i}\int\limits_{(\ast)}\frac{L^{z_1}\mathscr{F}(0,z_1,z_1,0,-z_1)}{z_1^{5+j_1+j_2+j_4+\ell+n}}dz_1.$$
Finally shifting to $\Re e (z_1)<0$ and we obtain that the above sum is
\begin{equation}
16\mathscr{F}(0,0,0,0,0)\bm{\gamma}(j_1,j_2,j_3,j_4)(\log L)^{j_1+j_2+j_3+j_4-4}+O\left((\log L)^{j_1+j_2+j_3+j_4-5}\right),
\end{equation}
with 
\begin{equation}\label{Definitiongamma}
\bm{\gamma}(j_1,j_2,j_3,j_4):=(-1)^{1+j_4}\mathop{\sum\sum}_{k+\ell+n=j_3}\frac{\bm{\beta}(k,\ell,n)}{(4+j_1+j_2+j_4+\ell+n)!}.
\end{equation}

%%%					THE POLE AT Z4=-Z2							   %%%
%%%%%%%%%%%%%%%%%%%%%%%%%%%%%%%%%%%%%%%%%%%%%%%%%%%%%%%%%%%%%%%%%%%%%%%%%%%%%%%%%%%%%%%%%%%%%%%%%%%%%%%%%%%%%%%%%%%%%%%%%%%%%%%%%%%%%%%%%%%%%%%%

\noindent $\mathbf{The \ pole \ at}$ $z_4=-z_2$ : It is not difficult to see that in fact, the pole at $z_4=-z_2$ has the same main term as in the previous case. In fact, applying the changes $z_1\leftrightarrow z_2$ and we see that this residue is given by \eqref{Polez4=-z1}, but with the first two variables switched in $\mathscr{F}$. This is not a real problem since the main term only involves $\mathscr{F}(0,0,0,0,0)$.

\vspace{0.2cm}

%%%					THE POLE AT Z4=0								   %%%
%%%%%%%%%%%%%%%%%%%%%%%%%%%%%%%%%%%%%%%%%%%%%%%%%%%%%%%%%%%%%%%%%%%%%%%%%%%%%%%%%%%%%%%%%%%%%%%%%%%%%%%%%%%%%%%%%%%%%%%%%%%%%%%%%%%%%%%%%%%%%%%%
\noindent $\mathbf{The \ pole \ at}$ $z_4=0$ : We return to Expression \eqref{M2(q,i)}. The residue at $z_4=0$ is given by the linear combination (we do not mention the derivatives of $\mathscr{F}$)
$$16\mathop{\sum\sum}_{k+\ell+n=j_4}\bm{\beta}(k,\ell,n)(\log L)^{k-8}\mathcal{A}(j_1,...,j_4,k,\ell,n),$$
where $\bm{\beta}(k,\ell,n)$ is defined by \eqref{Definitionbeta} and 
$$\mathcal{A}(j_1,...,j_4,k,\ell,n):=\frac{1}{(2\pi i)^3}\int\limits_{(\ast)}\int\limits_{(\ast)}\int\limits_{(\ast)}\frac{L^{z_1+z_2+z_3}\mathscr{F}(0,z_1,z_2,z_3,0)dz_3dz_2dz_1}{(z_1+z_3)(z_2+z_3)z_3^{1+j_3}z_2^{2+j_2+n}z_1^{2+j_1+\ell}}.$$
We now move the $z_3$-line such that $\Re e (z_3)<-\Re e(z_1+z_2)$, passing three poles : one at $z_3=0$ of order $1+j_3$, one at $z_3=-z_1$ of order 1 and the last at $z_3=-z_2$, that is also simple. We thus get the decomposition
$$\mathcal{A}(j_1,...,j_4,k,\ell,n)=\sum_{i=0}^2\mathcal{R}_i(j_1,...,j_4,k,\ell,n)+O(1).$$
\vspace{0.2cm}
\noindent$\mathbf{Treatment \ of}$ $\mathcal{R}_0(j_1,...,j_4,k,\ell,n)$ : It is routine now to see that 
$$\mathcal{R}_0=\mathop{\sum\sum}_{a+b+c=j_3}\bm{\beta}(a,b,c)\frac{(\log L)^a}{(2\pi i)^2}\int\limits_{(\ast)}\int\limits_{(\ast)}\frac{L^{z_1+z_2}\mathscr{F}(0,z_1,z_2,0,0)}{z_1^{3+j_1+\ell+b}z_2^{3+j_2+n+c}}dz_2dz_1.$$
Now moving the $z_2$-line to $\Re e (z_2)<-\Re e (z_1)$ and then the $z_1$-contour to $\Re e (z_1)<0$ and we obtain that 
\begin{equation}\label{R0}
\begin{split}
\mathcal{R}_0=\mathop{\sum\sum}_{a+b+c=j_3}\bm{\beta}(a,b,c)&\frac{\mathscr{F}(0,0,0,0,0)(\log L)^{4+j_1+j_2+j_3+\ell+n}}{(2+j_2+n+c)!(2+j_1+\ell+b)!} \\ & + O\left((\log L)^{3+j_1+j_2+j_3+\ell+n}\right).
\end{split}
\end{equation}
\vspace{0.2cm}
\noindent$\mathbf{Treatment \ of}$ $\mathcal{R}_1(j_1,...,j_4,k,\ell,n)$ : Observing that 
$$\mathcal{R}_1=\frac{(-1)^{1+j_3}}{(2\pi i)^2}\int\limits_{(\ast)}\int\limits_{(\ast)}\frac{L^{z_2}\mathscr{F}(0,z_1,z_2,-z_1,0)}{(z_2-z_1)z_1^{3+j_1+j_3+\ell}z_2^{2+j_2+n}}dz_2dz_1,$$
we can proceed as in the previous case (the pole at $z_4=-z_1$) and obtain
\begin{equation}
\begin{split}
\mathcal{R}_1= & \ \frac{(-1)^{1+j_3}\mathscr{F}(0,0,0,0,0)(\log L)^{4+j_1+j_2+j_3+\ell+n}}{(4+j_1+j_2+j_3+\ell+n)!} \\ & + \ O\left((\log L)^{3+j_1+j_2+j_3+\ell+n}\right).
\end{split}
\end{equation}
\vspace{0.2cm}
\noindent $\mathbf{Treatment \ of}$ $\mathcal{R}_2(j_1,...,j_4,k,\ell,n)$ : We find exactly the same term, i.e. $\mathcal{R}_2=\mathcal{R}_1$.
\subsubsection{Assembling the main terms} We define 
\begin{equation}\label{eta}
\bm{\eta}(j_1,j_2,j_3,j_4):=\mathop{\sum\sum\sum}_{\substack{k+\ell+n = j_4 \\ a+b+c = j_3}}\frac{\bm{\beta}(k,\ell,n)\bm{\beta}(a,b,c)}{(2+j_2+n+c)!(2+j_1+\ell+b)!},
\end{equation}
and
\begin{equation}\label{sigma}
\mathfrak{S}(j_1,j_2,j_3,j_4):=2\bm{\gamma}(j_1,...,j_4)+2\bm{\gamma}(j_1,j_2,j_4,j_3)+\bm{\eta}(j_1,...,j_4).
\end{equation}
Then we obtain
\begin{proposition}\label{PropositionPrefinal} The quantity defined by \eqref{DefinitionM4(j1,...,j4,F)} equals
$$\mathscr{M}_D^4(i_1,...,j_4;\mathscr{F})=16\frac{\mathscr{F}(0,...,0)\mathfrak{S}(j_1,...,j_4)}{(\log L)^{4-(j_1+...j_4)}}+O\left(\frac{1}{(\log L)^{5-(j_1+...+j_4)}}\right),$$
where $\mathfrak{S}(j_1,...,j_4)$ is defined by \eqref{sigma}.
\end{proposition}
In order to finalize our computation, we have
\begin{lemme}\label{LemmaF(0,...,0)} The value of $\mathscr{F}(s,z_1,...,z_4)$ at $(0,...,0)$ is $\zeta(2)$.
\end{lemme}
\begin{proof} Examining \eqref{Eq22} and \eqref{PolarPartDirichletSerie1} and we see that 
$$\mathscr{F}(0,0,0,0,0)=\mathscr{P}(0,0,0,0,0).$$
To prove the result, it is enough to show that for each prime $p$, we have $$\mathscr{P}_p(0,...,0) = \left( 1-\frac{1}{p^2}\right)^{-1}.$$
By \eqref{Eq22}, the local factor at $p$ of $\mathscr{P}(s,z_1,...,z_4)$ is given by the local factor of $\mathcal{L}(s,z_1,...,z_4)$ divided by the one of the right handside involving the zeta functions. Since we evaluate at $(0,...,0)$, we obtain
$$\mathscr{P}_p(0,...,0)=\mathcal{L}_p(0,...,0)\left(1-\frac{1}{p}\right)^{-4}.$$
Thus, it is enough to show that 
$$\mathcal{L}_p(0,...,0)\left(1-\frac{1}{p^2}\right)=\left(1-\frac{1}{p}\right)^4.$$
Writing $\mu_2$ for $\mu_{2,0}$, we have
\begin{alignat*}{1}
\mathcal{L}_p (0,...,0)= & \ \mathop{\sum\sum}_{\substack{0\leqslant d+\ell_i \leqslant 2 \\ \ell_1\ell_2=0}}\frac{\mu_2(p^{d+\ell_1})\mu_2(p^{d+\ell_2})f(p^{\ell_1};1)f(p^{\ell_2};1)}{p^{d+\ell_1+\ell_2}} \\ = & \ 2\mathop{\sum\sum}_{0\leqslant d+\ell\leqslant 2}\frac{\mu_2(p^{\ell+d})\mu_2(p^d)f(p^{\ell};1)}{p^{d+\ell}}-\sum_{0\leqslant d\leqslant 2}\frac{\mu_2(p^d)^2}{p^d}.
\end{alignat*}
Using \eqref{Valuesf} and \eqref{ValueprimePower} (with $\nu=0$), we obtain that this expression is 
$$1+\frac{4}{p}-\frac{8}{p+1}-\frac{8}{p(p+1)}+\frac{1}{p^2}+\frac{2}{p}\left(\frac{3-p^{-1}}{1+p}\right).$$
Finally, multiplying by $(1-p^{-2})=p^{-2}(p^2-1)$ leads to the desired factor $(1-p^{-1})^4.$
\end{proof}
We now replace $\mathscr{F}(0,...,0)$ by $\zeta(2)$ in Proposition \ref{PropositionPrefinal} and then insert the value of $\mathcal{M}_D^4(j_1,...,j_4)$ in \eqref{RedefinitionM4}. Writing $\log q=\lambda^{-1}\log L$, we get 
\begin{proposition} We have the following asymptotic formula for the diagonal main term appearing in decomposition \eqref{DecompositionM4}
\begin{equation}\label{AsymptoticValueDiagonal}
\mathscr{M}_D^4(q)=2\mathop{\sum\sum\sum\sum}_{\substack{j_s+j_1+...+j_4=4  \\ 0\leqslant j_k \leqslant 2 , \ k=1,...,4}}\left(\frac{2}{\lambda}\right)^{j_s}\mathscr{C}_{j_s,j_1,...,j_4}\mathfrak{S}(j_1,...,j_4)+O_\lambda\left(\frac{1}{\log q}\right).
\end{equation}
\end{proposition}

\subsection{Evaluation of $\mathscr{M}^4_{OD}(q)$} We proceed in a completely analogous way as in the previous section. First of all, we also restrict the computation to $\mathcal{OD}^{MT}(\ell_1,\ell_2;q)$ since the dual term gives the same result. We will begin by evaluating the residue \eqref{ValueODMT} up to some terms that will not contribute. After that, we will return to \eqref{DefinitionM4OD} and find an appropriate expression for a certain Dirichlet series in order to localize the various poles. Finally, we will see that the resulting expression matches perfectly with \eqref{DefinitionM4(j1,...,j4,F)} whose value has already been established in Proposition \ref{PropositionPrefinal}.

\subsubsection{Computation of the residue of $\mathcal{F}(s,\ell_1,\ell_2;q)$ at $s=0$}
We recall that 
$$\mathcal{OD}^{MT}(\ell_1,\ell_2;q)=\sum_{i=0}^2\mathrm{Res}_{s=0}\left\{\frac{\mathcal{F}_i(s,\ell_1,\ell_2;q)}{s}\right\}.$$
A very pleasant fact is that $\mathrm{Res}_{s=0}\mathcal{F}_{i}$ will not contribute in the final asymptotic formula unless $i=2$; the two others will be at most $O_\lambda(\log^{-1}q)$. The heuristic reason is the following : In § \ref{Averaging}, we will express our main term as the $z_i$-integral in which a certain differential operator (see \eqref{DefinitionOperatorDifferentielle}) depending on $s,u_1,u_2$ and $\log q$ acts on a function. If we look this operator, we remark that for each term, the sum of the order of differentiation plus the power of $\log q$ is $4$. If we take in count the residue of $\mathcal{F}_i$ with $i\leqslant 1$, then we just add some lower 'order' terms to \eqref{DefinitionOperatorDifferentielle}. We therefore focus on $i=2$.
By Proposition \ref{PropositionParity} and Lemma \ref{LemmaParity}, we see that each part of $\mathcal{F}_2$ is even, so we can take the residue at $s=0$ for each of them separately. We first isolate the polar part in the function $A$ around $s=0$ by writing
\begin{equation}\label{IsolatePolarPartA}
A(s,u_1,u_2;q)= \frac{q^{u_1+u_2}\mathscr{A}(s,u_1,u_2)}{(2s+u_1+u_2)(2s-u_1-u_2)},
\end{equation}
where $\mathscr{A}(s,u_1,u_2)$ is entire and does not vanish in a neighborhood of $\Re e (s)=\Re e (u_i)=0.$ We now easily get 
$$\partial_{u_1=0}A = \frac{\mathscr{A}(s,0,0)}{4s^2}\left(\log q +\frac{\mathscr{A}_{u_1}(s,0,0)}{\mathscr{A}(s,0,0)}\right),$$
and 
\begin{alignat*}{1}
\partial^2_{u_1u_2=0}A = \frac{\mathscr{A}(s,0,0)}{4s^2}&\left\{ \left(\log q +\frac{\mathscr{A}_{u_1}(s,0,0)}{\mathscr{A}(s,0,0)}\right)\left(\log q +\frac{\mathscr{A}_{u_2}(s,0,0)}{\mathscr{A}(s,0,0)}\right) \right. \\ + & \ \left.\frac{1}{2s^2}+\partial_{u_2=0}\left(\frac{\mathscr{A}_{u_1}(s,u_1,0)}{\mathscr{A}(s,u_1,0)}\right)\right\}.
\end{alignat*}
Writing
\begin{equation}\label{DefinitionB1B2}
\begin{split}
B_0(s;\ell_1,\ell_2) : = & \ B(s,0,0;\ell_1,\ell_2), \\
B_1(s;\ell_1,\ell_2):= & \ (\partial_{u_1=0}+\partial_{u_2=0})B(s,u_1,u_2;\ell_1,\ell_2), \\ B_2(s;\ell_1,\ell_2) := & \ \partial^2_{u_1u_2=0}B(s,u_1,u_2;\ell_1,\ell_2),
\end{split}
\end{equation}
we infer that the contribution of $\mathcal{F}_2(s,\ell_1,\ell_2;q)$ to our final asymptotic formula comes from the residue at $s=0$ of the following quantity (in fact we drop out all factors where we derive $\mathscr{A}$)
\begin{equation}\label{FinalResidue}
\frac{\mathscr{A}(s,0,0)}{4s^3}\left\{ B_0(s;\ell_1,\ell_2)\left(\frac{1}{2s^2}+\log^2 q\right)+B_1(s;\ell_1,\ell_2)\log q +B_2(s;\ell_1,\ell_2)\right\},
\end{equation}
which is 
\begin{equation}\label{ValueFinalResidue}
\begin{split}
\frac{\mathscr{A}(0,0,0)}{8}&\left(\frac{1}{4!}\partial^4_{s=0}B_0(s;\ell_1,\ell_2) +\log^2(q) \partial^2_{s=0}B_0(s;\ell_1,\ell_2)\right. \\ & \ +\log (q)\partial^2_{s=0}B_1(s;\ell_1,\ell_2)+ \partial^2_{s=0}B_2(s;\ell_1,\ell_2)\Big) + E(\ell_1,\ell_2;q),
\end{split}
\end{equation}
where the error term $E(\ell_1,\ell_2;q)$ is such that when we average it over the $\ell_i,d$ in \eqref{DefinitionM4OD}, we obtain $O((\log q)^{-1})$ (here $E(\ell_1,\ell_2;q)$ contains all term where $\mathscr{A}$ is derived at least one time). By construction, we have $\mathscr{A}(0,0,0)=2\zeta(0)\zeta(2)^{-1}=-\zeta(2)^{-1}$ (see \eqref{DefinitionA(s,u1,u2)}, \eqref{CrucialIdentity}, \eqref{IsolatePolarPartA}) and thus, taking into account the error coming from $\mathcal{F}_i$, $i\leqslant 1$, we can summarize all previous computation in
%%%%%%%%%%%%%%%%%%%%%%%%%%%%%%%%%%%%%%%%%%%%%%%%%%%%%%%%%%%%%%%%%%%%%%%
%%%%%%%%%%%%%%%%%%%%%%%%%%%%%%%%%%%%%%%%%%%%%%%%%%%%%%%%%%%%%%%%%%%%%%%
%%%%%		PROPOSITION INDIVIDUAL ASYMPTOTIC FORMULA			%%%%%
%%%%%%%%%%%%%%%%%%%%%%%%%%%%%%%%%%%%%%%%%%%%%%%%%%%%%%%%%%%%%%%%%%%%%%%
%%%%%%%%%%%%%%%%%%%%%%%%%%%%%%%%%%%%%%%%%%%%%%%%%%%%%%%%%%%%%%%%%%%%%%%
\begin{proposition}\label{PropositionlocalAsymptotic} Let $\mathcal{OD}^{MT}(\ell_1,\ell_2;q)$ defined by \eqref{ValueODMT}. Then we have the formula
$$\mathcal{OD}^{MT}(\ell_1,\ell_2;q)=c(\ell_1,\ell_2;q)+\mathcal{E}(\ell_1,\ell_2;q),$$
where $c(\ell_1,\ell_2;q)$ is given by \eqref{ValueFinalResidue} with $\mathscr{A}(0,0,0)=-\zeta(2)^{-1}$ and the error term $\mathcal{E}(\ell_1,\ell_2;q)$ is such that when we average it over $\ell_i$ in \eqref{DefinitionM4OD}, we get $O_\lambda
((\log q)^{-1})$.
\end{proposition}

%%%%%%%%%%%%%%%%%%%%%%%%%%%%%%%%%%%%%%%%%%%%%%%%%%%%%%%%%%%%%%%%%%%%%%%
%%%%%%%%%%%%%%%%%%%%%%%%%%%%%%%%%%%%%%%%%%%%%%%%%%%%%%%%%%%%%%%%%%%%%%%
%%%%%%%%%%%%%%%%%%%%%%%%%%%%%%%%%%%%%%%%%%%%%%%%%%%%%%%%%%%%%%%%%%%%%%%
%%%%%				  FINAL COMPUTATIONS					      	 %%%%%
%%%%%%%%%%%%%%%%%%%%%%%%%%%%%%%%%%%%%%%%%%%%%%%%%%%%%%%%%%%%%%%%%%%%%%%
%%%%%%%%%%%%%%%%%%%%%%%%%%%%%%%%%%%%%%%%%%%%%%%%%%%%%%%%%%%%%%%%%%%%%%%
%%%%%%%%%%%%%%%%%%%%%%%%%%%%%%%%%%%%%%%%%%%%%%%%%%%%%%%%%%%%%%%%%%%%%%%

%%%%%%%%%%%%%%%%%%%%%%%%%%%%%%%%%%%%%%%%%%%%%%%%%%%%%%%%%%%%%%%%%%%%%%%
%%%%%%%%%%%%%%%%%%%%%%%%%%%%%%%%%%%%%%%%%%%%%%%%%%%%%%%%%%%%%%%%%%%%%%%
%%%%%					AVERAGING OVER li						%%%%%
%%%%%%%%%%%%%%%%%%%%%%%%%%%%%%%%%%%%%%%%%%%%%%%%%%%%%%%%%%%%%%%%%%%%%%%
%%%%%%%%%%%%%%%%%%%%%%%%%%%%%%%%%%%%%%%%%%%%%%%%%%%%%%%%%%%%%%%%%%%%%%%
\subsubsection{Averaging over $\ell_i$}\label{Averaging}
\noindent We come back to \eqref{DefinitionM4OD} and insert the quantity $\mathcal{OD}^{MT}$ given by \eqref{ValueFinalResidue}. We can remove the primality condition $(k\ell_1\ell_2,q)=1$ for an acceptable error term and we obtain our off-diagonal main term
\begin{equation}\label{FinalOffDiagMT}
\begin{split}
\mathscr{M}_{OD}^{4}(q)=\frac{-1}{8\zeta(2)(2\pi i)^4}\int\limits_{(2)}\int\limits_{(2)}\int\limits_{(2)}\int\limits_{(2)} \prod_{i=1}^4 L^{z_i}\widehat{P_L}(z_i)&\mathfrak{D}_q^4 \cdot \left\{\mathcal{L}(s,u_1,u_2,z_1,...,z_4)\right\} \\ & \times \ \frac{dz_4 dz_3dz_2dz_1}{z_4z_3z_2z_1},
\end{split}
\end{equation}
where $\mathcal{L}(s,u_1,u_2,z_1,...,z_4)$ is the Dirichlet series defined by (recall Definitions \eqref{DefinitionB1B2} and \eqref{DefinitionA(s,u1,u2)})
\begin{equation}\label{DefinitionDirichlet2}
\begin{split}
&\mathcal{L}(s,u_1,u_2,z_1,...,z_4) \\ & := \mathop{\sum\sum\sum}_{k\geqslant 1 \ (\ell_1,\ell_2)=1}\frac{\mu_{2,z_1-z_2}(k\ell_1)\mu_{2,z_3-z_4}(k\ell_2)\bm{\delta}(\ell_1;s,u_1,u_2)\bm{\delta}(\ell_2;s,u_2,u_1)}{k^{1+z_1+z_3}\ell_1^{1+z_1}\ell_2^{1+z_3}},
\end{split}
\end{equation}
and $\mathfrak{D}^4_q$ is an order four differential operator given by 
\begin{equation}\label{DefinitionOperatorDifferentielle}
\mathfrak{D}_q^4 := \left.\left( \frac{1}{4!}\partial^4_{s}+\log^2 (q) \partial^2_s +\log (q) \partial^2_s (\partial_{u_1}+\partial_{u_2})+\partial^2_s\partial^2_{u_1u_2}\right)\right|_{s=u_i=0}.
\end{equation}
It is very important now to have an adequate expression for \eqref{DefinitionDirichlet2} in a way to locate the poles and their orders for future contour shift in the $z_i$-integrals. The classical method, as in Section \ref{SectionM4D}, is to compute for each prime $\P$ the local factor $\mathcal{L}_\P$ at $\P$. This is a quite tedious calculation, but is not difficult since all arithmetic functions are cubefree and we already computed their values on prime powers (see \eqref{ValueprimePower}, \eqref{Definitiondelta1} and \eqref{Definitiondelta2}). We do not want to figure out all details, but by close examination of the polar part in the local factor, we can conclude that $\mathcal{L}$ admits the following factorization
\begin{equation}\label{Facto1}
\begin{split}
\mathcal{L}=\mathscr{P}(s,u_1,u_2,z_1,...,z_4)&\prod_{i=1}^2\left(\frac{\zeta(1+z_i+z_3)\zeta(1+z_i+z_4)}{\zeta(1+s+z_i+u_2)\zeta(1-s+z_i+u_1)}\right) \\ \times & \ \prod_{i=3}^4\left(\frac{1}{\zeta(1+s+z_i+u_1)\zeta(1-s+z_i+u_2)}\right),
\end{split}
\end{equation}
where $\mathscr{P}$ is an Euler product absolutely convergent in a "good" neighborhood of the domain of holomorphy of the above product. Furthermore, if we factorize now the poles of the zeta functions, then we can rewrite \eqref{Facto1} as
\begin{equation}\label{Facto2}
\mathcal{L}=\mathscr{F}(s,u_1,u_2,z_1,...,z_4)Q(s,u_1,u_2,z_1,...,z_4),
\end{equation}
where $\mathscr{F}$ is an entire function in a neighborhood of $(0,...,0)$ which does not vanish and $Q\in\mathbb{C}(s,u_1,u_2,z_1,...,z_4)$ is the rational function defined by 
\begin{equation}\label{DefQ}
\begin{split}
Q(s,u_1,u_2,z_1,...,z_4):= & \ (z_1+s+u_2)(z_1-s+u_1)(z_2+s+u_2)(z_2-s+u_1) \\ 
\times & \ \frac{(z_3+s+u_1)(z_3-s+u_2)(z_4+s+u_1)(z_4-s+u_2)}{(z_1+z_3)(z_1+z_4)(z_2+z_3)(z_2+z_4)}.
\end{split}
\end{equation}
Using our classical argument concerning the derivatives of $\mathscr{F}$ and we obtain
\begin{alignat*}{1}
\mathscr{M}_{OD}^{4}(q)= & \ \frac{-1}{8\zeta(2)(2\pi i)^4}\int\limits_{(2)}\int\limits_{(2)}\int\limits_{(2)}\int\limits_{(2)}\prod_{i=1}^4L^{z_i}\widehat{P_L}(z_i)\mathscr{F}(0,0,0,z_1,...,z_4) \\ & \times \mathfrak{D}_q^4 \cdot \left\{ Q(s,u_1,u_2,z_1,...,z_4)\right\}\frac{dz_4dz_3dz_2dz_1}{z_4z_3z_2z_1}+O_\lambda\left(\frac{1}{\log q}\right) \\ =: & \ \sum_{i=1}^4\mathscr{M}_{OD}^{4}(i;q)+O_\lambda\left(\frac{1}{\log q}\right),
\end{alignat*}
where this decomposition takes care of the separation of the operator $\mathfrak{D}_q^4$ in \eqref{DefinitionOperatorDifferentielle}. We will compute each term separately.

%%%%%%%%%%%%%%%%%%%%%%%%%%%%%%%%%%%%%%%%%%%%%%%%%%%%%%%%%%%%%%%%%%%%%%%
%%%%%%%%%%%%%%%%%%%%%%%%%%%%%%%%%%%%%%%%%%%%%%%%%%%%%%%%%%%%%%%%%%%%%%%
%%%%%				COMPUTATION OF M0D(i,q)						%%%%%
%%%%%%%%%%%%%%%%%%%%%%%%%%%%%%%%%%%%%%%%%%%%%%%%%%%%%%%%%%%%%%%%%%%%%%%
%%%%%%%%%%%%%%%%%%%%%%%%%%%%%%%%%%%%%%%%%%%%%%%%%%%%%%%%%%%%%%%%%%%%%%%

\subsubsection{Computation of $\mathscr{M}_{OD}^{4}(q)$}
\noindent We compute in this last section $\mathscr{M}_{OD}^{4}(i;q)$ for $i=1,...,4$. Fortunately, we will see that these main terms can be expressed as the same integral as \eqref{M2(q,i)}, which has already been computed. We start with $i=1$; first of all, we have
\begin{alignat*}{1}
\frac{1}{4!}\partial^4_{s=0} &(\mathrm{Num}(Q(s,0,0,z_1,...,z_4)))=\mathop{\sum\sum\sum\sum}_{\substack{i_1+j_1+...+i_4+j_4=4 \\ 0\leqslant i_k,j_k \leqslant 1}}\mathscr{B}_{i_1,j_1,...,i_4,j_4}\prod_{k=1}^4z_k^{2-i_k-j_k},
\end{alignat*}
where 
\begin{equation}\label{DefinitionmathscrB}
\mathscr{B}_{i_1,j_1,...,i_4,j_4}=\frac{(-1)^{j_1+...+j_4}}{i_1!j_1!\cdots i_4!j_4!}.
\end{equation}
Hence, writing explicitly $\widehat{P_L}(z_i)$ and we obtain
\begin{equation}
\mathscr{M}_{OD}^{4}(1;q)=\frac{-2}{\zeta(2)}\mathop{\sum\sum\sum\sum}_{\substack{i_1+j_1+...+i_4+j_4=4 \\ 0\leqslant i_k,j_k \leqslant 1}}\mathscr{B}_{i_1,j_1,...,i_4,j_4}\mathscr{M}(i_1+j_1,...,i_4+j_4;\mathscr{F}),
\end{equation}
where for any $(a,b,c,d)\in\mathbb{N}^4$, we define 
\begin{alignat*}{1}
\mathscr{M}(a,b,c,d;\mathscr{F}):=\frac{(\log L)^{-8}}{(2\pi i)^4}\int\limits_{(\ast)}\int\limits_{(\ast)}\int\limits_{(\ast)}\int\limits_{(\ast)}&\frac{L^{z_1+z_2+z_3+z_4}\mathscr{F}(0,0,0,z_1,...,z_4)}{(z_1+z_3)(z_1+z_4)(z_2+z_3)(z_2+z_4)} \\ \times & \ \frac{dz_4dz_3dz_2dz_1}{z_1^{1+a}z_2^{1+b}z_3^{1+c}z_4^{1+d}}.
\end{alignat*}
We remark that this integral is exactly the expression $\mathscr{M}_D^{4}(j_1,...,j_4;\mathscr{F})$ in \eqref{M2(q,i)} (modulo the factor 16), then its value is given by (see Proposition \ref{PropositionPrefinal})
\begin{equation}\label{ValueM1}
\begin{split}
\mathscr{M}(i_1+j_1&,...,i_4+j_4;\mathscr{F})= \mathscr{F}(0,...,0)\mathfrak{S}(i_1+j_1,i_2+j_2,i_3+j_3,i_4+j_4) \\ & \times \ (\log L)^{-4+\sum_{k=1}^4 i_k+j_k}+O\left((\log L)^{-5+\sum_{k=1}^4i_k+j_k}\right),
\end{split}
\end{equation}
where $\mathfrak{S}$ is given by \eqref{sigma}.  \\
Let $\mathscr{M}_{OD}^{4}(2;q)$ denotes the contribution of $\log^2 (q)\partial^2_{s=0}$, then we get
\begin{equation}\label{ValueM2}
\mathscr{M}_{OD}^{4}(2;q)= \frac{-4}{\zeta(2)}\log^2(q)\mathop{\sum\sum\sum\sum}_{\substack{i_1+j_1+...+i_4+j_4=2 \\ 0\leqslant i_k,j_k \leqslant 1}}\mathscr{B}_{i_1,j_1,...,i_4,j_4}\mathscr{M}(i_1+j_1,...,i_4+j_4;\mathscr{F}).
\end{equation}
We now compute the part with $\log(q)\partial^2_{s}(\partial_{u_1}+\partial_{u_2})|_{s=u_i=0}$. We remark that the action of $(\partial_{u_1}+\partial_{u_2})|_{u_i=0}$ consists of a sum of eight terms in which one of the eight factors is missing. We have therefore
\begin{equation}\label{ValueM3}
\mathscr{M}_{OD}^{4}(3;q)=\frac{-4}{\zeta(2)}\log (q)\sum_{\ell = 1}^4\mathop{\sum\sum\sum\sum}_{\substack{i_1+j_1+...i_4+j_4=2 \\ 0\leqslant i_k,j_k\leqslant 1 \\ i_\ell j_\ell=0}}\mathscr{B}_{i_1,...,j_4}\mathscr{M}^{(\ell)}(i_1+j_1,...,i_4+j_4;\mathscr{F}),
\end{equation} 
where $\mathscr{M}^{(\ell)}$ means that we add $1$ at the $\ell^{th}$ component. 

Finally, we can focus on the action of $\partial^2_{s}\partial^2_{u_1u_2}|_{s=u_i=0}$. We see first that $\partial^2_{u_i=0}(\mathrm{NUM}(Q))$ consists of a sum of sixteen terms  where there are two missing factors (one indexed by $i_n$ and the other by $i_\ell$). It follows that 
\begin{equation}\label{ValueM4}
\mathscr{M}_{OD}^{4}(4;q)=\frac{-4}{\zeta(2)}\sum_{n,\ell=1}^4\mathop{\sum\sum\sum\sum}_{\substack{i_1+j_1+...i_4+j_4=2 \\ 0\leqslant i_k,j_k\leqslant 1 \\ i_n=j_\ell=0}}\mathscr{B}_{i_1,...,j_4}\mathscr{M}^{(n,\ell)}(i_1+j_1,...,i_4+j_4;\mathscr{F}),
\end{equation}
where this time, $\mathcal{M}^{(n,\ell)}$ means that we add $1$ to $i_n$ and $1$ to $j_\ell$. To finalize the calculations, we have
\begin{lemme} The value of $\mathscr{F}(0,...,0)$ is given by the infinite product 
$$\mathscr{F}(0,...,0)=\zeta(2)\prod_{p}\left(1+2\frac{1+p^{-1}}{p^2(1-p^{-1})^3}\right).$$
\end{lemme} 
\begin{proof}
As in Lemma \ref{LemmaF(0,...,0)}, we have by examining \eqref{Facto1} and \eqref{Facto2},
$$\mathscr{F}(0,...,0)=\mathscr{P}(0,...,0)$$
and for each prime $p$,
$$\mathscr{P}_p (0,...,0)=\mathcal{L}_p(0,...,0)\left(1-\frac{1}{p}\right)^{-4}.$$
By \eqref{DefinitionDirichlet2}, we see that 
\begin{alignat*}{1}
\mathcal{L}_p(0,...,0)= & \ \mathop{\sum\sum}_{\substack{0\leqslant k+\ell_i\leqslant 2 \\ \ell_1\ell_2=0}}\frac{\mu_2(p^{k+\ell_1})\mu_2(p^{k+\ell_2})\bm{\delta}(p^{\ell_1};0,0,0)\bm{\delta}(p^{\ell_2};0,0,0)}{p^{k+\ell_1+\ell_2}} \\ 
= & \ 2\mathop{\sum\sum}_{0\leqslant k+\ell\leqslant 2}\frac{\mu_2(p^{k+\ell})\mu_2(p^{k})\bm{\delta}(p^{\ell};0,0,0)}{p^{k+\ell}}-\sum_{0\leqslant k\leqslant 2}\frac{\mu_2(p^k)^2}{p^k},
\end{alignat*}
with (see \eqref{Definitiondelta1} and \eqref{Definitiondelta2})
$$\bm{\delta}(p;0,0,0)=2\left(1-\frac{1}{p}\right)\left(1-\frac{1}{p^2}\right)^{-1} \ \mathrm{and} \ \bm{\delta}(p^2;0,0,0)=4\left(1-\frac{1}{p}\right)\left(1-\frac{1}{p^2}\right)^{-1}.$$
Hence (recall \eqref{ValueprimePower})
\begin{alignat*}{1}
\mathcal{L}_p(0,...,0)= & \ 1+\frac{4}{p}+\frac{1}{p^2}-\frac{8}{p}\left(1-\frac{1}{p}\right)\left(1-\frac{1}{p^2}\right)^{-1} \\ = & \ \left\{ 1-\frac{1}{p^4}-\frac{4}{p}-\frac{4}{p^3}+\frac{8}{p^2}\right\}\left(1-\frac{1}{p^2}\right)^{-1} \\ = & \ \left\{ \left(1-\frac{1}{p}\right)^4 +\frac{2}{p^2}\left(1-\frac{1}{p^2}\right)\right\}\left(1-\frac{1}{p^2}\right)^{-1}.
\end{alignat*}
Multiplying by the factor $(1-p^{-1})^{-4}$ finishes the proof.
\end{proof}

\bibliography{SourceMollification}
%\addcontentsline{toc}{section}{References}
\bibliographystyle{alpha}

\end{document}